\documentclass[11pt,a4paper]{amsart}
\usepackage{verbatim}

\usepackage[toc]{appendix}
\usepackage[T1]{fontenc}
\newcommand{\tail}{\textnormal{\texttt{Tail}}}
\usepackage{graphicx}
\usepackage{mdframed}
\usepackage{enumerate}
\usepackage{amsmath,amsfonts,amssymb}
\usepackage{color}
\def\loc{\operatorname{loc}}
\usepackage{cite}
\usepackage{latexsym}
\definecolor{citation}{rgb}{0.11,0.67,0.84}
\definecolor{formula}{rgb}{0.1,0.2,0.6}
\definecolor{url}{rgb}{0.11,0.67,0.84}
\usepackage{pgf,tikz}
\def\aaat{\textnormal{\texttt{a}}}
\usepackage{mathrsfs}
\usepackage{fancyhdr}
\usepackage{dutchcal}

\newcommand{\medint}{-\kern -,375cm\int}
\newcommand{\epsb}[1]{\varepsilon_{\textnormal{\texttt{#1}}}}

\newcommand{\medintinrigo}{-\kern -,315cm\int}
\makeatletter
\newcommand{\linethrough}{\mathpalette\@thickbar}
\newcommand{\@thickbar}[2]{{#1\mkern0mu\vbox{
    \sbox\z@{$#1#2\mkern-0.5mu$}%
    \dimen@=\dimexpr\ht\tw@-\ht\z@+2\p@\relax 
    \hrule\@height0.5\p@ 
    \vskip\dimen@
    \box\z@}}
}
\makeatother

\newcommand{\mathstrike}[1]{\ensuremath{\linethrough{#1}}}

\newcommand{\nra}[1]{\mathstrike{\lVert} #1 \rVert}
\newcommand{\snra}[1]{\mathstrike{[} #1 ]}

\newtheorem{theorem}{Theorem}[section]
\newtheorem{lemma}[theorem]{Lemma}

\newtheorem{proposition}[theorem]{Proposition}

\newtheorem{corollary}[theorem]{Corollary}
\newtheorem{definition}[theorem]{Definition}
\newtheorem{remark}[theorem]{Remark}
\numberwithin{equation}{section}

\usepackage{dsfont}

\newcommand{\reqnomode}{\tagsleft@false}

\usepackage{hyperref}

 \usepackage{comment}
 
 \includecomment{comment}
 \excludecomment{comment}

\def\sss{\textnormal{\texttt{s}}}

\newcommand{\rad}{\textnormal{\texttt{rad}}}

\def\dxy{\,{\rm d}x{\rm d}y}
\def\dyx{\,{\rm d}y{\rm d}x}
\def\dx{\,{\rm d}x}
\def\dlam{\,{\rm d}\lambda}
\def\BBB{\textnormal{B}}
\def\dz{\,{\rm d}z}

\def\elly{\textnormal{\texttt{l}}}

\def\dy{\,{\rm d}y}
\def\qq{\mathfrak{q}}
\def\aaa{\mathfrak{a}}
\def\bbb{\mathfrak{b}}
\def\ccc{\mathfrak{c}}

\def\er{\mathbb R}
\def\en{\mathbb N}
\def\tet{\textnormal{\texttt{t}}}

\def \diver{\,{\rm div}}
\def\dist{\,{\rm dist}}

\newcommand\ttB{\textnormal{\texttt{B}}}

\allowdisplaybreaks
\makeatletter
\DeclareRobustCommand*{\bfseries}{%
  \not@math@alphabet\bfseries\mathbf
  \fontseries\bfdefault\selectfont
  \boldmath
}

\makeatother

\newlength{\defbaselineskip}
\newcommand{\mint}{\mathop{\int\hskip -1,05em -\, \!\!\!}\nolimits}

\def \diver{\,{\rm div}}

\newcommand{\R}{\mathbb{R}}

\def\cchh{c_{\textnormal{\texttt{h}}}}

\def\bbt{\textnormal{\texttt{b}}}
\newcommand{\eps}{\varepsilon}
\def\MMM{\textnormal{\texttt{M}}}

\newcommand{\ti}[1]{\tilde{#1}}
\newcommand{\cac}[1]{\mathfrak c_{#1}}
\newcommand{\mf}[1]{\textnormal{\texttt{#1}}}

\newcommand{\ddim}{\textnormal{dim}_{\mathcal H}}

\newcommand{\dd}{\,{\rm d}}

\newcommand{\rrr}{\textnormal{\texttt{r}}}

\newcommand{\sE}{\scalebox{0.88}{$\mathds{E}$}}
\newcommand{\sA}{\scalebox{1}{$\mathds{A}$}}
\newcommand{\sAA}{\scalebox{1.13}{$\textnormal{\texttt{A}}$}}
\newcommand{\sAAA}{\scalebox{0.88}{$\textnormal{\texttt{A}}$}}

\newcommand{\rr}{\varrho}
\newcommand{\snr}[1]{\lvert #1\rvert}
\newcommand{\nr}[1]{\lVert #1 \rVert}
\newcommand{\rif}[1]{(\ref{#1})}
\newcommand{\tx}[1]{\textnormal{\texttt{#1}}}
\newcommand{\stackleq}[1]{\stackrel{\eqref{#1}}{ \leq}}
\newcommand{\stackgeq}[1]{\stackrel{\eqref{#1}}{ \geq}}

\def\loc{\operatorname{loc}}

\def\eqn#1$$#2$${\begin{equation}\label#1#2\end{equation}}

\newcommand{\data}{\textnormal{\texttt{data}}}

\def\XXint#1#2#3{{\setbox0=\hbox{$#1{#2#3}{\int}$}
     \vcenter{\hbox{$#2#3$}}\kern-.5\wd0}}

\title{}

\author[De Filippis]{Cristiana De Filippis}  \address{Cristiana De Filippis\\Dipartimento SMFI, Universit\`a di Parma\\ Parco Area delle Scienze 53/A, 43124 Parma, Italy} \email{\url{cristiana.defilippis@unipr.it}}
\author[Mingione]{Giuseppe Mingione}  \address{Giuseppe Mingione\\Dipartimento SMFI, Universit\`a di Parma, Parco Area delle Scienze 53/a, Campus, 43124 Parma, Italy} \email{\url{giuseppe.mingione@unipr.it}}
\author[Nowak]{Simon Nowak}  \address{Simon Nowak\\Fakult\"at f\"ur Mathematik, Universit\"at Bielefeld\\Postfach 100131, 33501 Bielefeld, Germany} \email{\url{simon.nowak@uni-bielefeld.de}}

\begin{document}



\title{Partial regularity in nonlocal systems II}

\begin{abstract}
Solutions to nonlinear nonlocal systems of order $2s>1$ in $\er^n$ are $C^{1,\alpha}$, for every $\alpha <2s-1$, outside a closed singular set whose Hausdorff dimension is less than $n-2$, and which is empty when  $n=2$.
 \end{abstract}

\maketitle

\vspace{4mm}
 
 \centerline{To Francesco Leonetti on his 70th, with friendship and gratitude}

%
\setcounter{tocdepth}{1}
{\small \tableofcontents}

\section{Introduction}
Our aim is to prove partial regularity of solutions to vectorial nonlocal problems by addressing the most challenging case, where the nonlinear dependence occurs with respect to the natural fractional difference quotient of the solution, relative to the order of differentiation of the equation\footnote{This is the second paper in our study of partial regularity for solutions to nonlocal vectorial integrodifferential equations, building on \cite{parte1}. The results and methods here do not supersede those of \cite{parte1}, which addresses a different class of equations—namely, those in \rif{eqweaksolpre00}. While \cite{parte1} emphasizes the precise dependence on external ingredients and employs a comprehensive set of nonlinear potential-theoretic techniques, such methods are not used here. Nonetheless, a few results from that work are used as preliminary technical tools here.}. The final outcome will be a fully nonlocal counterpart of the classical partial regularity results known in the local case as for instance exposed in the standard treatises \cite{giaorange, giagreen,giu}. Specifically, we shall consider vectorial integrodifferential equations\footnote{As in \cite{parte1}, we call them systems, with a slight abuse of terminology.} of the type 
	\eqn{nonlocaleqn}
	$$
	-\mathcal{N}_{\sAAA}u=0\qquad \mbox{in} \ \ \Omega \subset \er^{n}, 
	$$
where $n\geq 2$, $\Omega\subset \er^n$ is a bounded open set, and the operator $-\mathcal{N}_{\sAAA}$ is defined in the sense of distributions via
\eqn{theone}
$$
\langle -\mathcal{N}_{\sAAA} u,\varphi\rangle  :=\int_{\mathbb{R}^{n}}\int_{\mathbb{R}^{n}}\langle\sAA \left(\frac{u(x)-u(y)}{|x-y|^s}\right),\frac{\varphi(x)-\varphi(y)}{|x-y|^s}\rangle\frac{\dxy}{|x-y|^{n}}\,, \quad  0 < s < 1
$$
for every $\varphi \in C^{\infty}(\er^n;\er^{N})$ with compact support in $\Omega$; see Definition \ref{def:weaksol} below. 
The vector field $\sAA  \colon \er^{N} \to \er^{N}$ here satisfies the following growth and ellipticity  assumptions:
\eqn{bs.1}
$$
\left\{
\begin{array}{c}
\displaystyle
\  \sAA  \in  C^{1}_{\loc}(\er^{N};\er^N), \quad \sAA (-w)=-\sAA (w)\\ [4pt]\displaystyle
\ \snr{\sAA (w)}+ \snr{\partial \sAA (w)}|w|\le\Lambda \snr{w},\quad  \Lambda^{-1}\snr{\xi}^{2}\leq \langle  \partial \sAA (w)\xi,\xi\rangle,\\ [4pt]\displaystyle
\ \snr{\partial \sAA (w_{1})-\partial \sAA (w_{2})}\le  \Lambda\omega\left(\snr{w_{1}-w_{2}}\right)
\end{array}
\right.
$$
whenever $w, w_1, w_2, \xi \in \er^{N}$, where $\Lambda \geq 1$ is a fixed constant, while $\omega\colon [0, \infty) \to [0, 1]$ is a modulus of continuity\footnote{that is, $\omega(\cdot)$ is a non-decreasing, continuous, concave function such that $\omega(0)=0$. Note that in this paper, when referring to a couple of vectors $w= (w^\alpha),v=(v^\alpha) \in \er^N$, we denote the usual scalar product as $\langle w,v \rangle =  \sum_{1\leq \alpha \leq N}w^\alpha v^\alpha$.}.  We refer to Section \ref{notsec}  and to \cite[Section 3.1]{parte1} for the notation used in this paper. 
The notion of solution to \eqref{nonlocaleqn} involves the $\tail$ space \cite{BK, DKP, kokupa, KMS1} 
$$
L^{1}_{2s} := \left\{w\in L^{1}_{\loc}(\er^n;\er^{N}) \, \colon \, \int_{\er^n} \frac{\snr{w(x)}}{1+\snr{x}^{n+2s}} \dx< \infty\right\}\footnote{H\"older's inequality implies $W^{s,2}(\er^n;\er^{N})\subset L^2(\er^n;\er^{N})\subset L^{1}_{2s}$.}.
$$ 
Weak (energy) solutions\footnote{We shall often call weak solutions, simply,  solutions.}  to \eqref{nonlocaleqn} are then defined as follows:
\begin{definition} \label{def:weaksol}
Under assumptions \eqref{bs.1}, a map $u \in W^{s,2}_{\loc}(\Omega;\er^{N}) \cap L^{1}_{2s}$ is a weak solution of \eqref{nonlocaleqn} when
\eqn{weaksol}
$$
\langle -\mathcal{N}_{\sAAA} u,\varphi\rangle  =0
$$ 
		holds for every $\varphi \in W^{s,2}(\er^n;\er^{N})$ with compact support in $\Omega$\footnote{Note that \eqref{weaksol} in particular holds whenever $\varphi$ is of the form $\varphi=\eta v$ with $v \in W^{s,2}_{\loc}(\Omega;\er^{N})$ and $\eta \in C^{\infty}$ has compact support in $\Omega$.}.
	\end{definition}
Gradient partial regularity concerns regularity of solutions $u$ outside a {\em singular set} \, $\Omega\setminus \Omega_u$, which is indeed defined as the complement of the {\em regular set} 
\eqn{regularset}
$$
\Omega_u:=\left\{x\in \Omega\, \colon\, \mbox{$Du$ is H\"older continuous in some neighbourhood of $x$}\right\}.$$ 
The regular set $\Omega_u$ is by definition open and might be empty\footnote{On top of that, the existence of $Du$, which is not ensured by Definition \ref{def:weaksol}, is in fact already a regularity result.}. The main problem in this setting is to establish how large the singular set might actually be. An answer is in
\begin{theorem}[Partial regularity]\label{mainth}
	Assuming \eqref{bs.1}, if $u$ is a weak solution to  \eqref{nonlocaleqn} with $s>1/2$, then 
	\eqn{mainresult}
$$
\begin{cases}
\, u\in C^{1,\alpha}_{\loc}(\Omega_{u};\er^{N})\  \mbox{for every $\alpha < 2s-1$}\\
\,  \ddim(\Omega\setminus \Omega_u)\leq n-2-\theta \ \  \mbox{when $n\geq 3$}\\
\, \Omega =\Omega_u  \quad \mbox{when $n=2$}
\end{cases}
$$
where $\theta \in (0,1)$ depends only on $n,N,s,\Lambda$ \footnote{As usual, $\ddim$ denotes the Hausdorff dimension.}. 
\end{theorem}
The regular set can then be identified via additional regularity as in  the following:
\begin{theorem}[Regular set and higher differentiability]\label{mainth2} 
Under assumptions \eqref{bs.1}, let $u$ be a weak solution to  \eqref{nonlocaleqn}  with $s\geq 1/2$. There exists a number $\mathfrak s\equiv \mathfrak s(n,N,s,\Lambda) $ such that
\eqn{higherss} 
$$ Du\in W^{ \mathfrak s,2}_{\loc}(\Omega;\er^{N\times n}), \qquad s < \mathfrak s <1.$$ If $s>1/2$, then
\eqn{ilsingolare}
$$
\Omega_{u}= \Big\{
x \in \Omega \, \colon\, 
 \lim_{\rr\to 0}\,  \rr^{1-s}\nra{Du}_{L^2(B_{\rr}(x))}=0 \Big\}.
$$
\end{theorem}
As in the local case, a further characterization of $\Omega_u$ can be given using the classical concept of $\eps$-regularity via suitable excess functionals. For this we  recall the by now classical notion of $\tail$  \cite{BK, DKP, kokupa, KMS1} of a map $w\in L^{1}_{\loc}(\er^n;\er^{N})$, with respect to the ball $B_{\rr}(x)\subset \er^n$, i.e., 
$$
\tail(w;B_{\rr}(x)):= \rr^{2s}\int_{\er^n\setminus B_{\rr}(x)} \frac{\snr{w(y)}}{\snr{y-x}^{n+2s}} \dy.
$$ 
It follows that $w\in L^{1}_{2s}$ iff $\tail(w;B)$ is finite for every ball $B \subset \er^n$. 
\begin{definition}[Nonlocal affine excess] \label{leccesso}
Let $w \in L^{2}(B_{\rr}(x);\er^{N})\cap L^{1}_{2s}$, $s>1/2$, where $B_{\rr}(x)\subset \er^n$ is a ball; let $\ell\colon \er^n\to \er^N$ be an affine map. The (affine) excess functional $\tx{E}_{w}(\ell;x,\rr)$ is defined by 
\eqn{excl}
$$
\tx{E}_{w}(\ell;x,\rr):=\sqrt{\nra{w-\ell}_{L^{2}(B_{\rr}(x))}^2+\tail(w-\ell;B_{\rr}(x))^2} .
$$
\end{definition}
\begin{theorem}[Regular points and $\eps$-regularity]\label{mainth3} Under assumptions \eqref{bs.1}, let $u$ be a weak solution to \eqref{nonlocaleqn} with $s >1/2$. For every positive $\alpha< 2s-1$ there exists a universal threshold $\epsb{b}>0$, depending only on $n,N,s,\Lambda,\omega(\cdot)$ and $\alpha$, such that the following conditions are equivalent whenever $x\in \Omega$:
\eqn{condreg}
$$\rr^{-s}\tx{E}_{u}(\ell;x,\rr)<\epsb{b} \ \mbox{holds for some affine map $\ell$ and some ball $B_{2\rr}(x)\Subset \Omega$
}$$
\eqn{condreg2}
$$\mbox{There exists a ball $B_{r_{x}}(x) \Subset \Omega$ such that } Du\in C^{0, \alpha}(B_{r_{x}}(x);\er^{N\times n}).
$$
\end{theorem}
Some remarks are in order. 
\begin{itemize}
\item Two features of Theorems \ref{mainth}--\ref{mainth3} are worth noting. The first is the Hausdorff dimension estimate in \eqref{mainresult}, which coincides with the one available for local problems \cite{giaorange, giagreen}, i.e., systems of the type
\eqn{integer2}
$$-\diver A(Du)=0,$$
and does not depend on the fractional order $s$. An explanation for this fact can be given observing that the lower degree of regularization of the operator $-\mathcal{N}_{\sAAA}$, compared with \eqref{integer2}, is compensated by the fact that, in the nonlocal case, the nonlinearity acts on a fractional difference quotient rather than on the full gradient, and it is, in a sense, weaker. We refer to Section \ref{comparisons} for further discussion. A second feature concerns the technical side. Partial regularity results based on  $\eps$-regularity theorems are available in the literature in the nonlocal case \cite{dalio1, millot1, millot2, roberts, schi1, schi2, schi3}, but they typically appear in special situations in which linearity or particular structural assumptions allow the use of linear tools, such as the Caffarelli–Silvestre extension. By contrast, the genuinely nonlinear nature of the problem considered here requires the development of an intrinsically nonlocal and nonlinear approach.

\item Unless otherwise stated, essentially all the main results in this paper, with the exception of \eqref{higherss}  from  Theorem \ref{mainth2}, will be proved assuming that $s>1/2$. Such a condition guarantees, among other things, that every affine map belongs to the tail space $L^1_{2s}$, and therefore the excess functional $\tx{E}_{w}(\ell;x,\rr)$ in \eqref{excl} is well-defined.
\item The excess is continuous in the sense that
\eqn{contexcess}
$$
\mbox{the function }\   (\ell,x) \mapsto \tx{E}_{w}(\ell;x,\rr) \ \mbox{is continuous}.
$$ 
Specifically, if $B_{\rr}(x_k)\Subset \Omega$ is a sequence of balls such that $x_k\to x$  and $B_{\rr}(x)\Subset \Omega$, and if $\ell_k(y)\equiv \texttt{b}_k(y-x_k)+\texttt{a}_k$ is a sequence of affine maps such that $\texttt{b}_k\to \texttt{b}$ in $\er^{N\times n}$ and $\texttt{a}_k\to \texttt{a}$ in $\er^{N}$, then $\tx{E}_{w}(\ell_k;x_k,\rr)\to\tx{E}_{w}(\ell;x,\rr)$. This is essentially the content of \cite[Proposition 3.4]{parte1}. 
\item A key role is played by the higher differentiability property first discovered in \cite{KMS}, according to which weak solutions improve their differentiability from $W^{s,2}_{\loc}$ to $W^{s+\delta_0,2}_{\loc}$, for some universal $\delta_0>0$, depending only on $n,N,s,\Lambda$. This result allows to prove everywhere regularity in two dimensions \eqref{mainresult}$_3$ and to improve the Hausdorff dimension estimate in \eqref{mainresult}$_2$ in that the positive number $\theta\equiv \theta(n,N,s,\Lambda)$ appearing in \eqref{mainresult}$_2$ directly relates to $\delta_0$. Moreover, the borderline case $s=1/2$ achieved in \eqref{higherss} is again a consequence of the results in \cite{KMS}. The higher differentiability in \cite{KMS} also (indirectly) intervenes in the formulation of the fractional Harmonic Approximation Lemma \ref{shar} which is a crucial tool in the partial regularity proof.
\item In Theorem \ref{mainth3} the quantity $\epsb{b}$ is universal, in the sense that it does not depend on the solution, while the radius $r_x$ in \eqref{condreg2} does.  
\item As a consequence of \eqref{bs.1}, it follows that\footnote{in this paper we shall always denote $0_{\er^{N}}\equiv 0$.}
\eqn{sim}
$$
\sAA (0)= 0\,, \qquad \partial \sAA (-w)=\partial \sAA (w) \qquad \forall \, w \in \er^{N}.
$$ 
Moreover, the second assumption in \eqref{bs.1}$_1$ (oddness of $\sAA $) can be omitted. Indeed, for every $w\in \er^{N}$, split $\sAA $ as follows 
$$
\begin{cases}
\sAA (w)=\sAA_{\textnormal{odd}}(w)+\sAA_{\textnormal{even}}(w)\\[2pt]
 \displaystyle \sAA_{\textnormal{odd}}(w):=\frac{1}{2}[\sAA (w)-\sAA (-w)] \,, \quad \sAA_{\textnormal{even}}(w):=\frac{1}{2}[\sAA (w)+\sAA (-w)]\,. 
 \end{cases}
 $$
If $u$ is a solution to $- \mathcal{N}_{\sAAA} u=0$ in the sense of Definition \ref{def:weaksol}, then   
$$0=\langle - \mathcal{N}_{\sAAA} u,\varphi\rangle= \langle -\mathcal{N}_{\sAA_{\textnormal{odd}}}u,\varphi\rangle + \underbrace{\langle-\mathcal{N}_{\sAA_{\textnormal{even}}}u,\varphi\rangle}_{=0} = \langle-\mathcal{N}_{\sAA_{\textnormal{odd}}}u,\varphi\rangle$$
and therefore $u$ is also a solution to $- \mathcal{N}_{\sAA_{\textnormal{odd}}} u=0$. 
Noting also that $\sAA_{\textnormal{odd}}(-w)=-\sAA_{\textnormal{odd}}(w)$ for every $w\in \er^{N}$, it is now easy to see that $\sAA_{\textnormal{odd}}$ satisfies all the assumptions in \eqref{bs.1} and therefore it is actually not necessary to assume the second condition in \eqref{bs.1}$_1$.
\end{itemize}
As clear from the above statements, several quantities and results do not depend on either the operator or the solution 
$u$ considered; these quantities are referred to as “universal” \footnote{More precisely, a quantity will be called universal when, given the system in \rif{weaksol} under the assumptions in \rif{bs.1}, it will depend only on $n,N,s,\Lambda$ and sometimes on the modulus of continuity $\omega(\cdot)$ and on some other external regularity parameter, but not on the solution or on the specific system considered. This is for instance the case of the quantity $\epsb{b}$ in Theorem \ref{mainth3}.}. Dependence will instead often occur on the parameters
 $$
 \data \equiv (n,N,s,\Lambda),
 $$
which we call the \emph{structural ellipticity data of the problem}. Sometimes the constants involved in the estimates will depend only on a subset of the parameters included in $\data$, but for simplicity we will nevertheless indicate a dependence on the full set. 

 \subsection{Minima} Consider the integral functional of the Calculus of Variations 
$$
w \mapsto \mathcal F(w, \Omega):=  \int_{(\er^n\times \er^n)\setminus (\Omega^{\texttt{c}}\times \Omega^{\texttt{c}})} F\left(\frac{w(x)-w(y)}{|x-y|^s}\right) \frac{\dd  \mathcal L^{2n}(x,y)}{|x-y|^{n}} 
$$
where $F\colon \er^{N} \to \er$ is $C^2$-regular and convex, and satisfies 
 \eqn{cresce}
 $$
\snr{w}^2/\Lambda  \leq F(w) \leq  \Lambda \snr{w}^2, \qquad \Lambda \geq 1
 $$
 whenever $w \in \er^{N}$. 
  \begin{definition}\label{defimin}
Assume \eqref{cresce}. A map 
$u\in W^{s,2}_{\loc}(\Omega;\er^N)\cap L^1_{2s}$ is a local minimizer of 
$\mathcal F$ if, for every open set $\Omega'\Subset\Omega$, 
$\mathcal F(u,\Omega')<\infty$ and $
\mathcal F(u,\Omega')\leq \mathcal F(w,\Omega')
$ holds 
whenever 
$w\in W^{s,2}_{\loc}(\Omega;\er^N)\cap L^1_{2s}$ and 
$\operatorname{spt}(w-u)\Subset\Omega'$.
\end{definition}

   It turns out that \eqref{weaksol} is precisely the Euler-Lagrange equation of $\mathcal F$. 
 \begin{proposition}\label{equaziona}
Assume that $F\in C^1(\er^N)$ is convex and satisfies \eqref{cresce}. Then any
local minimizer of $\mathcal F$ in the sense of Definition \ref{defimin} satisfies
the weak formulation \eqref{weaksol} with $\sAA\equiv \partial F$. In particular,
if $\sAA=\partial F$ satisfies \eqref{bs.1}, then $u$ is a weak solution of
\eqref{nonlocaleqn} in the sense of Definition \ref{def:weaksol}.
\end{proposition}
 This last fact, whose (easy) derivation is in Section \ref{equazionasec}, allows us to apply the result for systems. 
 \begin{corollary} The regularity assertions of Theorems \ref{mainth}, \ref{mainth2} and \ref{mainth3} hold for any local minimizer of the functional $\mathcal F$ in the sense of Definition \ref{defimin} provided $F$ satisfies \eqref{cresce} and $\partial F\equiv \sAA$ satisfies \eqref{bs.1}. 
 \end{corollary}
Local minimizers according to Definition \ref{defimin} typically arise when solving boundary value problems
 $$
 u \mapsto \min_{w \in X} \mathds{F}(w)\,, \quad X:=   \{w\in W^{s,2}(\er^n;\er^N)\, \colon \, w \equiv u_0 \ \mbox{on}\ \er^n\setminus \Omega\} 
 $$
 where $u_0\in W^{s,2}(\er^n;\er^{N})$ is a fixed boundary datum (note that $W^{s,2}(\er^n;\er^{N}) \subset L^{1}_{2s} $) and $\Omega$ is a bounded domain, 
that is, $u \in X$ is a minimizer in the sense that 
$$
\mathds{F}(u) :=\int_{\er^n}\int_{\er^n} F\left(\frac{u(x)-u(y)}{|x-y|^s}\right) \frac{\dxy}{|x-y|^{n}}  
\leq \int_{\er^n}\int_{\er^n} F\left(\frac{w(x)-w(y)}{|x-y|^s}\right) \frac{\dxy}{|x-y|^{n}} =: \mathds{F}(w)
$$
holds whenever $w\in W^{s,2}(\er^n;\er^N)$ is such that $w-u\equiv 0$ outside $\Omega$. Also keeping \eqref{cresce} in mind, existence theorems follow using Direct Methods of the Calculus of Variations, see also \cite{DKP}.  
\subsection{Strategy}\label{strasec} The proof of Theorems  \ref{mainth}--\ref{mainth3} is quite long and delicate. A central point is a {\em nonlocal linearization} approach that allows us to linearize, in small balls, nonlinear systems like \eqref{nonlocaleqn} around a class of treatable, linear systems. This leverages a variant of a harmonic type approximation lemma previously released in \cite{parte1}, which is Lemma \ref{shar}\footnote{Harmonic type approximation lemmas were first introduced by De Giorgi in his fundamental work on the regularity of minimal surfaces  \cite{deg}. For applications and extension to local problems we refer for instance to  \cite{dumi,dust,dsv}. }. Contrary to what one might expect, we are not going to linearize around constant-coefficient systems. Rather, we employ the natural class of translation-invariant coefficients systems. These are  of the type 
\eqn{linearizzata0}
$$-\mathcal{L}_{\mathds{B}}h =g \in L^{\infty}, $$ 
where 
\eqn{weaksolLL}
$$
\langle -\mathcal{L}_{\mathds{B}} h,\varphi\rangle :=\int_{\mathbb{R}^{n}}\int_{\mathbb{R}^{n}}\langle \mathds{B}(x-y)\frac{h(x)-h(y)}{|x-y|^s},\frac{\varphi(x)-\varphi(y)}{|x-y|^s}\rangle\frac{\dxy}{|x-y|^{n}} , \quad 0<s<1
$$
for every $\varphi \in C^{\infty}(\er^n;\er^{N})$ with compact support, and $\mathds{B} \colon  \er^{n} \to \er^{N\times N}$ is measurable and  satisfies
\eqn{condib}
$$
\Lambda^{-1}\snr{\xi}^{2}\leq \langle  \mathds{B}(z)\xi,  \xi \rangle, 
\qquad \snr{ \mathds{B}(z)}\le \Lambda, \qquad  \mathds{B}(z)=  \mathds{B}(-z)
$$
for every choice of $z \in \er^n$, $\xi \in \er^{N}$, as in \eqref{bs.1}$_2$\footnote{The third condition in \eqref{condib} can always be dropped. Indeed, denoting 
$$
\mathds{B}_{\textnormal{sym}}(z):= \frac12[\mathds{B}(z)+\mathds{B}(-z)], \quad z \in \er^n
$$
we have $-\mathcal{L}_{\mathds{B}_{\textnormal{sym}}} h=g$ and $\mathds{B}_{\textnormal{sym}}$ still satisfies the first two conditions in \eqref{condib}. Note also that, more in general, the terminology translation invariant coefficients alludes to general coefficients $\mathds{B}(x,y)$ such that $\mathds{B}(x+h,y+h)=\mathds{B}(x,y)$ holds whenever $x,y,h\in \er^n$. In this paper we shall deal only with translation invariant coefficients of the type $\mathds{B}(x,y)\equiv \mathds{B}(x-y)$.}. These are in a sense the ``tangent'' systems in a nonlocal blow-up procedure, that we are going to describe below. The matrix $\mathds{B}(\cdot)$ is found via
 \eqn{affineappare}
 $$
  \mathds{B}(x-y)=\partial \sAA \left(\frac{\ell(\rr x)-\ell(\rr y)}{\snr{\rr x-\rr y}^{s}}\right), \quad \ell(x)=\bbt ( x-x_0)+ \texttt{a}, \  \texttt{a} \in \er^{N}, \ \bbt \in\er^{N\times n}$$
  so that 
  $$
  \mathds{B}(z) = \partial \sAA \left(\rr^{1-s}\bbt \frac{z}{|z|^s}\right)\,.
  $$
The affine map $\ell$ appearing in \eqref{affineappare} is, roughly speaking, the one realizing condition \eqref{condreg}
\eqn{smally} 
$$\rr^{-s}\tx{E}_{u}(\ell;x_0,\rr)<\epsb{b} \,, \qquad \epsb{b}\equiv \epsb{b}(\data,\omega(\cdot), \alpha) \in (0,1)\,.$$ 
It is precisely this {\em smallness condition} that allows us to linearize equation \eqref{nonlocaleqn} around \eqref{linearizzata0}. With $\alpha \in (0, 2s-1)$ fixed, the procedure can be summarized as follows. We blow-up $u-\ell$ in $B_{\rr}(x_0)$ obtaining a map, $v$, that solves a linearized system with measurable coefficients, that is 
\eqn{elr0}
$$
    \langle -\mathcal{L}_{\mathds{A}} v,\varphi\rangle:=  \int_{\mathbb{R}^{n}}\int_{\mathbb{R}^{n}}\langle\mathds A(x,y)\frac{v(x)-v(y)}{|x-y|^s},\frac{\varphi(x)-\varphi(y)}{|x-y|^s}\rangle\frac{\dxy}{|x-y|^{n}}=0
$$
holds for every $\varphi \in W^{s,2}(\er^n;\er^{N})$ with compact support in $\ttB_1\equiv B_1(0)$ \footnote{Throughout the procedure we are in fact using that $-\mathcal{L}_{\mathds{B}} \tilde  \ell=0$ holds for every affine map $\tilde \ell$ when $s>1/2$. See Lemma \ref{remarkino34}.}. The next step is to get an $L^2$-closeness information of the type 
\eqn{vicina}
$$ 
\nra{v-h}_{L^2(\ttB_{1/16})}\lesssim \eps
$$
for a suitably small $\eps$ that can be quantitatively determined in terms of $\data,\alpha$, where $h$ solves \eqref{linearizzata0} in $\ttB_{1/16}$ for a suitable $L^\infty$-function $g$. For this, after a localization procedure, we replace $v$ by $\tilde v:=\eta v$, where 
$\eta\equiv1$ in $\ttB_{1/4}$ and $\eta$ is supported in $\ttB_{1/2}$. The localization produces a function 
$g\in L^\infty(\ttB_{1/16};\er^N)$ such that
\eqn{vicini}
$$
\snr{\langle -\mathcal L_{\mathds{B}}\tilde{v} - g , \varphi \rangle} \lesssim [\omega(\epsb{b})]^{\delta} 
, \qquad \nr{g}_{L^\infty(\ttB_{1/16})} \lesssim 1, \quad \delta\equiv\delta(\data)\in (0,1)
$$
holds for every $\varphi\in C^{0,t}(\er^n;\er^N)$ supported in $\ttB_{1/16}$, such that $ [\varphi]_{0,t;\er^n}\leq 1$, where $t>s$ is again a universal quantity depending only on $\data$.  Here $\omega(\cdot)$ is the modulus of continuity appearing in \eqref{bs.1}. Using \eqref{vicini} we invoke the $s$-harmonic approximation machinery from Lemma \ref{shar} and find a solution $h$ to \eqref{linearizzata0} which is $L^2$-close to $\tilde{v}$, and therefore to $v$, in the sense of  \eqref{vicina}. Note that at this point the number $\eps\equiv\eps(\data, \alpha)$ required in \eqref{vicina} influences the choice of  the number $\epsb{b}\equiv \epsb{b}(\data,  \omega(\cdot),\alpha)$  introduced in \eqref{smally}  via  \eqref{vicini}; it is important to note that all the constants along this procedure are universal and only depend on $\data, \omega(\cdot), \alpha$. The numbers $n,N,s,\Lambda$ are to be considered as structural \emph{data} of the operator considered as well as $\omega(\cdot)$, while $\alpha$ is the rate of H\"older continuity of the solutions we are aiming at in partial regularity. The closeness condition \eqref{vicina} can now be used to get an {\em improvement-of-flatness} lemma in the sense that once the smallness condition \eqref{smally} is satisfied at the scale $B_{\rr}(x)$, then it quantitatively propagates at each scale $B_{\tet^i\rr}(x)$, $i\in \en$, where $\tet\equiv \tet(\data, \alpha) \in (0,1)$. Specifically, for every $i \in \en$, there exists an affine map $\ell_i$ attaining 
\eqn{campimi}
$$\tet ^{-is}\rr^{-s}\tx{E}_{u}(\ell_i;x,\tet ^i\rr)<\epsb{b}.
$$ 
Schematically, this comes along with a decay estimate for the full excess of the type 
\eqn{campi1}
$$
\tx{E}_{u}(\ell_i;x,\tet ^i \rr) \lesssim \tet ^{(1+\alpha)i}, \quad \mbox{for every integer $i\geq 0$}.
$$
We are now in the realm of classical partial regularity. Thanks to \eqref{contexcess} it can now be shown that \eqref{smally} is an {\em open condition}. This means that there exists a small ball $B_{r_{x}}(x)\subset \Omega$ such that 
\eqref{smally} holds replacing $x$ with $y$ whenever $y \in B_{r_{x}}(x)$. As a consequence, both \eqref{campimi} and \eqref{campi1} hold with the same constants involved whenever we are replacing $x$ with $y$. It follows that 
\eqn{esticam}
$$
\sup_{y \in B_{r_{x}}(x)} \inf_{\elly \textnormal{ affine}}  \nra{u-\elly}_{L^2(B_{\sigma}(y))}\lesssim \sigma^{1+\alpha}
 $$
 holds whenever $0<\sigma \leq \rr$. The decay estimate \eqref{esticam} implies that 
 \eqn{leadsto}
$$Du \in C^{0, \alpha}(B_{r_{x}}(x);\er^{N\times n})\,.$$ We have therefore proved that a point $x\in \Omega$ satisfying the smallness condition \eqref{smally} for some affine map $\ell$ is indeed a regular point. This proves \eqref{condreg} $\Rightarrow$ \eqref{condreg2} in Theorem \ref{mainth3} and preludes to the complete proof of Theorem \ref{mainth}. Indeed, this has to be converted into \eqref{mainresult} asserting that a.e. point is a regular point with an estimate on the size of non-regular points. For the dimensional estimate \eqref{mainresult}$_{2}$, and the everywhere regularity in two dimensions \eqref{mainresult}$_{3}$, we need essentially two ingredients. First, the higher differentiability result in \eqref{higherss}. This goes via an indirect use of fractional difference quotient methods in the setting of the aforementioned nonlocal self-improving properties (fractional Gehring's lemma); see Section \ref{highsec}. Specifically, we use that maps of the type
\eqn{approssima} 
$$
\texttt{w}_h(x)\stackrel{\textnormal{modulo cut-off}}{\approx} \frac{u(x+h)-u(x)}{|h|^\beta}, \qquad 0<\beta \leq 1, \ h \in \er^n\setminus\{0\}
$$
are themselves solutions to linear nonlocal systems of the type in \eqref{elr0}, i.e., 
\eqn{approssima2} 
$$-\mathcal{L}_{\mathds{A}}\texttt{w}_h= \texttt{f}_h \in L^{\infty}.  $$
See Proposition \ref{gradpre} and Lemma \ref{el : lem.loc} for the meaning of the symbol $\approx$ in \eqref{approssima} and for the identity of $\texttt{f}_h$ in \eqref{approssima2}. 
Therefore $\texttt{w}_h$ enjoys the regularity properties typical of solutions to systems as \eqref{approssima2}, as for instance the aforementioned higher differentiability, i.e.,
\eqn{proge}
$$
\mbox{$\texttt{w}_h\in W^{s+\delta_0,2}$ uniformly with respect to $h$, with $\delta_0\equiv \delta_0 (\data)>0$}.
$$
The estimates implicit in \eqref{proge}, combined with suitable choices of $\beta$ in \eqref{approssima} and basic embeddings in Besov spaces, imply the required higher gradient differentiability of $u$ stated in \eqref{higherss}. In order to conclude with \eqref{mainresult} the result in \eqref{higherss} can now be combined with the explicit characterization in \eqref{ilsingolare} and certain abstract facts about fine properties of fractional Sobolev functions obtained via nonlinear potential theoretic ideas. These are  detailed in Section \ref{thefines}; see in particular Lemma \ref{riducilemma}. 

The partial regularity scheme connecting \eqref{smally} to \eqref{leadsto} involves several technical
intermediate steps. For instance, we present in Theorem \ref{cdg} from Section \ref{homsec2} a derivation of {\em a priori regularity estimates} for solutions to
linear vectorial equations such as \eqref{linearizzata0}. Although a main point in Theorem \ref{cdg} is the explicit a priori estimate \eqref{dg.2bis}, which has to fit the partial regularity scheme described above in order to deliver the necessary information, we also aim at providing a fractional version of classical Campanato's methods \cite{cam1,cam2} in the context of translation invariant kernels,  thereby making no use of fundamental solutions and relying only on energy estimates.  Let us mention that in the scalar case regularity results for translation invariant operators are available, see for instance \cite[Section 2.4]{feros2} and the interesting work \cite{rose}. 

\subsection{Comparisons with the classical, local case} \label{comparisons}For a general introduction to partial regularity in the setting of nonlocal problems, and for parallels with the local theory, we refer to \cite{parte1}. There we treated the case of nonlocal vectorial equations of the type 
\eqn{eqweaksolpre00}
$$
			\int_{\er^{n}} \int_{\er^{n}} \langle a(x,y,u(x),u(y))\frac{u(x)-u(y)}{|x-y|^s},\frac{\varphi(x)-\varphi(y)}{|x-y|^s} \rangle \frac{\dxy}{|x-y|^{n}} =\int_{\Omega} \langle f, \varphi\rangle \dx 
$$
under suitable ellipticity and continuity assumptions on the tensor field $a(\cdot)$.
Systems as in \eqref{eqweaksolpre00} with $f\equiv 0$ are the nonlocal analogues of local systems of the type 
\eqn{integer}
$$
 -\diver\,  (a(x,u)Du)=0, 
$$
where 
$ \langle a(x,w)\xi, \xi\rangle  \approx \snr{\xi}^2$, $\xi \in \er^{N\times n}, w \in \er^N$, and 
$(x,w) \mapsto a(x,w)$ is uniformly continuous. 
These were first treated in the classical work of Morrey \cite{morrey} and Giusti \& Miranda \cite{giumi}. In both cases  \eqref{eqweaksolpre00} and \eqref{integer} the nonlinearity is weaker, as it is generated by a nonlinear dependence on the solution, rather than on its gradient/difference quotient. The system we are considering here in \eqref{weaksol} is instead a much more delicate case than \eqref{eqweaksolpre00} as the nonlinearity appears through the natural difference quotient of solutions, i.e., it is of maximal order. It is in fact the natural nonlocal analog of local systems of the type
\eqref{integer2}, 
featuring a nonlinear dependence on the gradient. The assumptions on the $C^1$-regular vector field $A\colon \er^{N\times n}\to \er^{N\times n}$ when considering \eqref{integer2} are usually
\eqn{ipotesilocali}
$$
\left\{
\begin{array}{c}
\displaystyle
\ \snr{A(z)}+ \snr{\partial A(z)}|z|\le\Lambda \snr{z},\quad  \Lambda^{-1}\snr{\xi}^{2}\leq \langle  \partial A(z)\xi,\xi\rangle,\\ [4pt]\displaystyle
\ \snr{\partial A(z_{1})-\partial A(z_{2})}\le \Lambda \omega\left(\snr{z_{1}-z_{2}}\right)
\end{array}
\right.
$$
for every $z, z_1, z_2, \xi\in \er^{N\times n}$, 
which obviously parallel \eqref{bs.1} \footnote{Assumptions \eqref{ipotesilocali} can be sometimes weakened, but the version presented here already contains the essence of what is needed for vectorial partial regularity.}. Systems as in \eqref{integer2} arise for instance as Euler-Lagrange equations of functionals such as 
$
v \mapsto \int F(Dv)\dx$, where $\langle \partial^2F(z)\xi, \xi\rangle  \approx \snr{\xi}^2 $ for $z,\xi\in \er^{N\times n}
$,  which are typically defined in classical Sobolev spaces $W^{1,2}$. For systems as in \eqref{integer2} singular solutions do appear \cite{N, sverakyan, savin}, and the best thing one can hope for is partial regularity as in \eqref{mainresult}, which in fact holds exactly in the same terms for $W^{1,2}$-solutions to \eqref{integer2} (taking $s=1$ in \eqref{mainresult}$_1$). Everywhere regularity only follows under special structure conditions, both in the local \cite{uh}, and in the nonlocal case \cite{now1}.   We refer to \cite{giaorange, giagreen} for the basic and classical techniques concerning \eqref{integer2} and \eqref{integer}, and to \cite{min08} for a (relatively) updated overview. 

\begin{remark}[Hausdorff dimension and interpolative nature of regular set characterization]{\em The Hausdorff dimension estimate \eqref{mainresult}$_{2}$ marks a difference from the results of \cite{parte1} where, for solutions to vectorial equations as 
\eqref{eqweaksolpre00}, the dimension of the singular set was seen not to exceed $n-2s-\theta$. 
This is basically due to the fact that in \eqref{eqweaksolpre00} $(v, w) \mapsto a(\cdot, v,w)$ is nondifferentiable and this inhibits higher differentiability of $u$ beyond $W^{s+\delta_0,2}$, for some small $\delta_0$ coming from an application of \cite{KMS} as  in \eqref{proge}. By contrast, the higher (fractional) differentiability \eqref{higherss} allows for a better dimension reduction as in \eqref{mainresult}$_{2}$.  This estimate is in accordance with the singular set estimate for fractional harmonic maps obtained in \cite{millot2}. A point  worth remarking is that the results of \cite{millot2} rely on the linearity of the operator via the use of the Caffarelli-Silvestre extension \cite{caff}, a tool that cannot obviously be employed in nonlinear cases as the present one. On a more speculative side, the nonlocal system \eqref{weaksol} can be viewed, from the standpoint of the differentiation order appearing in the nonlinearity, as an interpolation between the system in \eqref{integer}, where the differentiation order in the nonlinearity is zero, and the one in \eqref{integer2}, where it is one. Dimensional analysis now reveals a corresponding analogy in the characterization of the regular sets $\Omega_u$. Indeed, in the cases \eqref{integer} and \eqref{integer2}, the regular sets are given by
\eqn{integer3}
$$
\Big\{
x \in \Omega \, \colon\, 
 \lim_{\rr\to 0}\,  \rr\, \nra{Du}_{L^2(B_{\rr}(x))}=0 \Big\} \ \ \mbox{and} \ \ 
 \Big\{
x \in \Omega \, \colon\, 
 \lim_{\rr\to 0}\,  \nra{Du-(Du)_{B_{\rr}(x)}}_{L^2(B_{\rr}(x))}=0 \Big\},
$$
respectively. From a dimensional point of view, that is, by comparing the powers of $\rr$ involved in \eqref{integer3}, $\Omega_u$ defined in \eqref{ilsingolare} interpolates exactly between the two sets in \eqref{integer3}.
}
\end{remark}

\subsection{Organization of the paper} Basic notation and tools are established in Section \ref{presec}. There, among other things, we collect the basic properties of fractional Sobolev spaces that we shall employ here, together with the fractional harmonic approximation lemma needed to prove partial regularity (see Section \ref{harmonica}). In Section \ref{thefines} we prove a number of results linking the excess functional defined in \eqref{excl} to certain fine properties of fractional Sobolev functions. These results will be essential for proving partial regularity of solutions and related estimates on the dimension of the singular set. Sections \ref{homsec1} and \ref{homsec2} are instead devoted to the proof of a priori estimates for solutions to equations of the type in \eqref{linearizzata0}. In particular, we prove Theorem \ref{cdg}. The homogeneous case $g\equiv0$ is treated in Section \ref{homsec1}; the
non-homogeneous case, building on those results, is addressed in Section
\ref{homsec2}. Their content is instrumental to the proof of Theorem \ref{mainth}. Section \ref{highsec} deals instead with the regularity of solutions to the original system \eqref{weaksol}. Here, a delicate use of difference quotients is combined with the self-improving properties developed in \cite{KMS}. The core of the partial regularity proof is contained in Section \ref{coresec}. The main result there is the flatness improvement property described in Proposition \ref{cor.1}, which is achieved via the nonlocal linearization method already described in Section \ref{strasec}. In the concluding Section \ref{rs.sec} we finally employ the previously developed tools to prove all the main results.
\vspace{.3cm}

{\bf Acknowledgments.} This work is supported by the European Research Council, through the ERC StG project NEW, nr.~101220121 and by the University of Parma through the actions "Bando di Ateneo 2023 per la ricerca" (Mingione) and "Bando di Ateneo 2024 per la ricerca" (De Filippis). Simon Nowak is supported by the Deutsche Forschungsgemeinschaft (DFG, German Research Foundation) - SFB 1283/2 2021 - 317210226. 

\section{Preliminary material}\label{presec}
Most of the notation used in this paper was introduced in Part I \cite[Section 3.1]{parte1}, and we shall rely on it throughout, recalling here only the minimal essentials for the sake of readability. In the same spirit, we shall make extensive use of the preliminary results established in \cite{parte1}, reporting here only those definitions and lemmas that are most frequently used in this paper and that need to be readily available. As usual, we denote by $c$ a generic constant with $c\geq 1$, whose value may change from line to line when its precise value is not important. Particular occurrences may also be denoted by symbols such as $c_0,  \tilde c$ and the like. When we write, for instance, $c\equiv c (a,b,\delta)$,  this emphasizes that the constant  $c$ depends on the parameters $a,b,\delta$. 

\subsection{Basics and notation}\label{notsec} With $\mathcal  A \subset \er^{n}$ being a measurable subset such that  $0<|\mathcal A|<\infty$, and $w \colon \mathcal  A \to \er^{k}$, $k\geq 1$, being an integrable map, we set  
\eqn{integral}
$$(w)_{\mathcal  A}:=\frac{1}{|\mathcal A|}\int_{\mathcal  A}w\dx:= \mint_{\mathcal  A}w\dx,$$ which  represents the integral average of $w$ over $\mathcal  A$. We denote 
$$
B_r(x_{0}):= \{x \in \er^n  :   |x-x_{0}|< r\}, \ \ \ttB_{r}:= B_r(0)= \{x \in \er^n  :   |x|< r\}.
$$
As usual, when the center is not relevant in the context, we shall omit denoting it and the balls in question will be concentric. In such cases, we shall often indicate with $\BBB\subset \er^n$ a general ball, eventually denoting by $\rad(\BBB)$ its radius. 
Given $x \equiv (x_{i})_{1\leq i\leq n}\in \er^n$, we denote $|x|_{\infty}=\max_{i} \, |x_{i}|$ and 
$$ 
\begin{cases}
\, Q_r(x_0):= \{x \in \er^n  :   |x-x_0|_{\infty}< r\},\\
\, Q_{\textnormal{inn}}(B_r(x_0)):= \{x \in \er^n  :   |x-x_0|_{\infty}< r/\sqrt{n}\}\subset B_r(x_0)\subset Q_r(x_0)\,.
\end{cases}
$$
In the second line of the above display we have the inner hypercube of the ball $B_r(x_0)$, that is, the largest hypercube concentric with $B_r(x_0)$, with sides parallel to the coordinate axes, which is contained in $B_r(x_0)$. We shall also need to use a standard lattice of cubes of mesh $r>0$, and for this we shall canonically use
$
\mathcal L_r := \{Q_{r}(y) \subset \er^n \, \colon \, y\in (2r) \mathbb Z^n\}.
$
These are mutually disjoint cubes whose closures cover $\er^n$. We shall use the following elementary 
\begin{lemma}\label{lemmacubi} With $t \in (0,1/2)$ and $\sigma \in (0,1/4)$, consider the lattice $\{Q_{\sigma/\sqrt{n}}(z)\}_{z\in (2\sigma/\sqrt{n})\mathbb Z^n}$. 
With $\tx{I}:=\{y\in  (2\sigma/\sqrt{n})\mathbb Z^n \, \colon \,  \snr{y} \leq  t+ 2\sigma\}$, the family $\{\overline{Q_{\sigma/\sqrt{n}}(z)}\}_{z\in \tx{I}}$  covers $\ttB_t$ and $\# \tx{I} \approx_{n} (t/\sigma+2)^{n}$. 
\end{lemma} 
\begin{proof} It is sufficient to observe that $\{\overline{Q_{\sigma/\sqrt{n}}(z)}\}_{z\in (2\sigma/\sqrt{n})\mathbb Z^n}$ obviously covers $\ttB_t$ and that if $\snr{y} >  t+ 2\sigma$, then $Q_{\sigma/\sqrt{n}}(y) \cap \ttB_{t}=\emptyset$. For this just notice that $Q_{\sigma/\sqrt{n}}(y)\subset B_{\sigma}(y)$ and that $B_{\sigma}(y)\cap \ttB_{t}$ is empty. 
\end{proof}
\subsection{Fractional basics} For the basic properties of fractional Sobolev spaces we need we refer to \cite[Section 3.1]{parte1} and to the references therein. We only report the following basic definitions and inequalities. With $\mathcal  A \subset \er^{n}$ being this time an open subset and $w \colon \mathcal  A \to \er^{k}$, $k\geq 1$, being a measurable map, and with $t \in (0,1)$, $1\leq p<\infty$, we denote 
$$ 
[w]_{t,p;\mathcal A} :=\left(\int_{\mathcal A}\int_{\mathcal A}\frac{\snr{w(x)-w(y)}^{p}}{|x-y|^{n+pt}}\dxy\right)^{1/p} 
$$
which is the usual Gagliardo seminorm related to the fractional space $W^{t,p}(\mathcal A)$. Accordingly to \eqref{integral}, we denote 
$$
\nra{w}_{L^p(\mathcal A)} := \left(\mint_{\mathcal A}\snr{w}^p\dx\right)^{1/p} \,, \qquad 
\snra{w}_{t,p;\mathcal A} := \left(\frac{1}{\snr{\mathcal A}} \int_{\mathcal A}\int_{\mathcal A}\frac{\snr{w(x)-w(y)}^{p}}{|x-y|^{n+pt}}\dxy \right)^{1/p} 
$$
while the usual Gagliardo norm in $W^{t,p}(\mathcal A)$ is defined as 
$\nr{w}_{W^{t,p}(\mathcal A)}:= \nr{w}_{L^{p}(\mathcal A)}+[w]_{t,p;\mathcal A}$.
We shall employ several times the classical fractional Poincar\'e inequality \cite[(4.2)]{min03}, that is
\eqn{fp}
$$
\nra{w-(w)_{\BBB}}_{L^{p}(\BBB)} \lesssim_{n,t,p}  \snr{\BBB}^{t/n}  \snra{w}_{t,p;\BBB}
$$ 
whenever $\BBB \subset \er^n$ is a ball. With $\alpha \in (0,1]$ we denote the Hölder seminorm of $w$ as 
$$
[w]_{0,\alpha;\mathcal  A}:= \sup_{x,y\in \mathcal  A, x\neq y}\frac{\snr{w(x)-w(y)}}{|x-y|^{\alpha}}.
$$
We shall repeatedly use, even without mention, that
\eqn{faccia}
$$
\nra{w-(w)_{\mathcal  A}}_{L^{p}(\mathcal  A)} \leq 2 \nra{w-w_0}_{L^{p}(\mathcal  A)},
$$
 holds for every constant vector $w_0\in \er^k$ and $1 \leq p < \infty$. This is a direct consequence of the triangle inequality and Jensen’s inequality. In the special case $p=2$ the constant $2$ can be omitted. 

\subsection{Local minimizers and the Euler-Lagrange equation}\label{equazionasec}
For completeness, we rapidly sketch the proof of Proposition \ref{equaziona}, which is standard, but requires a few remarks due to the notion of minimality employed here with respect to the space of competitors used. We first  observe that the assumed convexity of $F$ and \eqref{cresce} imply that 
 \eqn{cresce2}
 $$
 |\partial F(w)| \lesssim_\Lambda  \snr{w}\,.
 $$
For this see for instance \cite{marc}.  
Next, consider open subsets $\Omega''\Subset \Omega'\Subset \Omega$ and $\varphi \in W^{s,2}(\er^n;\er^N)$ which is compactly supported in $\Omega''$; set $d:=\dist(\Omega'',\er^n\setminus \Omega')>0$. Again from \rif{cresce} and $ \mathcal F(u, \Omega') < \infty$, and the fact that $\varphi$ has compact support in $\Omega''$, it easily follows that, with $w_t:=u+t\varphi \in W^{s,2}_{\loc}(\Omega;\er^{N}) \cap L^{1}_{2s}$, $\mathcal F(w_t, \Omega')$ is finite for every choice of  $t\in (-1,1)\setminus\{0\}$.  We write
\begin{flalign*}
& F\left(\frac{w_t(x)-w_t(y)}{|x-y|^s}\right)-F\left(\frac{u(x)-u(y)}{|x-y|^s}\right)\\
& \qquad =
t \int_0^1 \langle \partial F\left(\frac{u(x)-u(y) +\lambda t (\varphi(x)-\varphi(y))}{|x-y|^s}\right), 
\frac{\varphi(x)-\varphi(y)}{|x-y|^s}\rangle \dd\lambda
\end{flalign*}
Integrating the above identity over $\er^{n}\times \er^{n}$, observing that $w_t\equiv u$ outside $\Omega'$, the minimality of $u$ yields
\begin{eqnarray}
\notag 
0 &\leq & \mathcal F(w_t, \Omega')-\mathcal F(u, \Omega') \\
&  = & t \int_{\er^n}\int_{\er^n}
\int_0^1 \langle \partial F\left(\frac{u(x)-u(y) +\lambda t (\varphi(x)-\varphi(y))}{|x-y|^s}\right), 
\frac{\varphi(x)-\varphi(y)}{|x-y|^s}\rangle \dd\lambda \,\frac{\dxy}{|x-y|^n}\notag \\
&  =: & t\textnormal{I}_t\,. \label{prometeo0} 
\end{eqnarray}
Observe that I$_t$ is finite for every choice of $t$; this will be proved below when eventually letting $t\to 0$. 
Dividing \eqref{prometeo0} by $t$ and letting $t\to 0$, by standard sign arguments we conclude with 
\eqn{prometeo}
$$
\int_{\mathbb{R}^{n}}\int_{\mathbb{R}^{n}}\langle\partial F \left(\frac{u(x)-u(y)}{|x-y|^s}\right),\frac{\varphi(x)-\varphi(y)}{|x-y|^s}\rangle\frac{\dxy}{|x-y|^{n}}=0
$$
provided we can justify the passage to the limit under the integral in \rif{prometeo0}. For this, we can split 
\begin{flalign*}
\textnormal{I}_t  &= \int_{\Omega'}\int_{\Omega'} [\ldots]_t
(x,y)\dxy +  \int_{\er^n\setminus \Omega'}\int_{\Omega' }  [\ldots]_t
(x,y)\dxy +  \int_{\Omega' } \int_{\er^n\setminus \Omega'} [\ldots]_t
(x,y)\dxy \\
& =: (\textnormal{II})_t +(\textnormal{III})_t +(\textnormal{IV})_t 
\end{flalign*}
and use Lebesgue's dominated convergence for each of the three pieces using the fact that $u \in W^{s,2}_{\loc}(\Omega;\er^N)\cap L^1_{2s}$ and \rif{cresce2}, ultimately proving that each term converges to the corresponding term coming from a similar splitting of the integral in \eqref{prometeo}.
The passage of $t\to 0$ in $(\textnormal{II})_t$ is indeed trivial as $u, \varphi \in W^{s,2}_{\loc}(\Omega;\er^N)$. The remaining terms are treated in the same way, so we simply demonstrate how to deal with $(\textnormal{III})_t$. On $\Omega'\times (\er^n\setminus \Omega')$ we have, again using that $\varphi$ is supported in $\Omega''$ and \rif{cresce2}, 
\begin{flalign*}
\snr{[\ldots]_t(x,y)} &\leq \frac{c(|u(x)|+|u(y)|)|\varphi(x)|}{|x-y|^{n+2s}} + \frac{c|\varphi(x)|^2}{|x-y|^{n+2s}}\\  & \leq 
c(n,s,d,\Omega', \Omega'')\left[\frac{(|u(x)|+|u(y)|)|\varphi(x)|}{1+|y|^{n+2s}} + \frac{|\varphi(x)|^2}{1+|y|^{n+2s}}\right]\,.
\end{flalign*}
Using Fubini and that $u\in L^1_{2s}$ it follows that the right-hand side is integrable and again Lebesgue dominated convergence applies. The term $(\textnormal{IV})_t$ can be treated similarly and \eqref{prometeo} is completely proven.

\subsection{Existence, solvability} In this section we consider a measurable matrix field $\sA\colon \er^{2n}\to \er^{N\times N}$ satisfying 
\eqn{linearell} 
$$
\begin{cases}
\Lambda^{-1}\snr{\xi}^{2} \leq \langle \sA(x,y)\xi,\xi\rangle\quad \mbox{and}\quad \snr{\sA(x,y)}
\le \Lambda,\\
\sA(x,y)=\sA(y,x)
\end{cases}
$$
hold for every choice of $\xi\in \er^{N}$ and for a.e. $(x,y)\in \er^{2n}$\footnote{As usual, assuming \eqref{linearell}$_2$ is not restrictive, as it is easily checked by passing to the symmetrized coefficients 
$$
\sA_{\textnormal{sym}}(x,y):= \frac12[\sA(x,y)+\sA(y,x)], \quad x,y \in \er^n.
$$}. Under assumptions \eqref{linearell}, we consider the linear operator 
$$
		\langle - \mathcal{L}_{\mathds{A}} u, \varphi \rangle =\int_{\er^{n}} \int_{\er^{n}}\langle  \sA(x,y)\frac{u(x)-u(y)}{|x-y|^s},\frac{\varphi(x)-\varphi(y)}{|x-y|^s}\rangle\frac{\dxy}{|x-y|^{n}}\,, 
$$
	 for every $\varphi\in C^{\infty}(\er^n;\er^{N})$ with compact support in $\Omega$. According to Definition \ref{def:weaksol}, we have 
	\begin{definition} \label{def:weaksol2}
Under assumptions \eqref{linearell} with $0<s<1$, a map $u \in W^{s,2}_{\loc}(\Omega;\er^{N}) \cap L^{1}_{2s}$ is a weak solution of  $-\mathcal{L}_{\scalebox{0.6}{$\sA$}}u=g$ in $\Omega$, with $g\in L^{2n/(n+2s)}(\Omega;\er^{N})$, when
	\eqn{solutio}
	$$
\langle - \mathcal{L}_{\mathds{A}} u, \varphi \rangle=
\int_{\Omega}\langle g,\varphi\rangle\dx$$
		holds for every $\varphi \in W^{s,2}(\er^n;\er^{N})$ with compact support in $\Omega$.
	\end{definition} 
	In addition to weak solutions, we recall the setting for solvability of Dirichlet problems of the type
\eqn{basicsolve}
$$
\begin{cases} 
\ -\mathcal{L}_{\mathds{A}}h=g\quad &\mbox{in} \ \ \Omega\\
\ h=v\quad &\mbox{in} \ \ \mathbb{R}^{n}\setminus \Omega\,,
\end{cases}
$$ 
where $\Omega \subset \er^n$ is a bounded domain. Our main reference here is \cite[Section 4]{parte1} along with \cite{kokupa} and \cite{bls}. With $\ti{\Omega}$ being another bounded open subset such that $\Omega \Subset \ti{\Omega}$, and $v\in L^{1}_{2s}$, we consider the nonlocal Dirichlet class 
$$
\mathbb{X}^{s,2}_{v}(\Omega ,\ti{\Omega}) := \left\{w\in W^{s,2}(\ti{\Omega};\er^{N})\cap L^{1}_{2s}\, \colon \, \mbox{$w\equiv v$ \, a.e. in $\mathbb{R}^{n}\setminus \Omega$}\right\}.
$$
Basic properties of this class are described in \cite[Section 4]{parte1} and related references; we denote $\mathbb{X}^{s,2}_{v}(\Omega, \ti{\Omega})\equiv \mathbb{X}^{s,2}_{0}(\Omega, \ti{\Omega})$ when $v\equiv 0$. Finally, we have that $h \in \mathbb{X}^{s,2}_{v}(\Omega, \ti{\Omega})$ iff $h-v\in \mathbb{X}^{s,2}_{0}(\Omega, \ti{\Omega})$ when  $v \in W^{s,2}(\ti{\Omega}; \er^{N})$. Note that, using a standard extension scheme - see for instance \cite[Lemma 2.11]{bls} or \cite[Lemma 5.1]{dpv} - any $w\in \mathbb{X}_{0}^{s,2}(\Omega,\ti{\Omega})$ can be extended to $\er^n$ in order to have $w \in W^{s,2}(\mathbb{R}^{n};\er^{N})$ with 
\eqn{ht.10}
$$
[w]_{s,2;\mathbb{R}^{n}}^{2}\le [w]_{s,2;\ti{\Omega}}^{2}+\frac{c\nr{w}_{L^{2}(\Omega)}^{2}}{\dist(\Omega,\er^n\setminus \ti{\Omega})^{2s}}\,,
$$
where $c\equiv c(n,N,s)$. Therefore when dealing with maps $w\in \mathbb{X}_{0}^{s,2}(\Omega,\ti{\Omega})$ we shall always assume that they are defined on $\er^n$ and that \eqref{ht.10} holds. 

	\begin{definition} \label{eqweaksol22}
A map $h\in \mathbb{X}^{s,2}_{v}(\Omega, \ti{\Omega})$ is a solution  to \eqref{basicsolve} with $g\in L^{2n/(n+2s)}(\Omega;\er^{N})$, when 
	\eqref{solutio}
		holds whenever $\varphi \in \mathbb{X}^{s,2}_{0}(\Omega, \ti{\Omega})$. 
	\end{definition} 
By the extension procedure recalled before \eqref{ht.10}, this is equivalent to require that in Definition \ref{eqweaksol22} identity \eqref{solutio} holds whenever $\varphi \in W^{s,2}(\er^n;\er^{N})$ is such that $\varphi \equiv 0$ outside $\Omega$. Note also that  every solution to \eqref{basicsolve} is a (local) solution to $-\mathcal{L}_{\mathds{A}}h=g$ in the sense of Definition \ref{solutio}. We shall use Dirichlet problems as in \eqref{basicsolve} to obtain comparison solutions. Therefore we recall two existence and comparison results that can be obtained adapting the analogous ones given in  \cite[Section 4]{parte1} for the case in which  the matrix $\sA(\cdot)$ is constant. Indeed, such results only depend on linearity and ellipticity but not on the fact that the coefficients involved are variable. 
\begin{lemma}\label{esiste} If $v\in W^{s,2}(\ti{\Omega};\er^{N})\cap L^1_{2s}$ and $g\in L^{2n/(n+2s)}(\Omega;\er^{N})$, there exists a unique solution $h\in \mathbb{X}^{s,2}_{v}(\Omega, \ti{\Omega})$ to \eqref{basicsolve} in the sense of
Definition \ref{eqweaksol22}. 
\end{lemma}
 
\begin{lemma}\label{compg}
Let $u$ be a weak solution to $-\mathcal{L}_{\mathds{A}}u=g\in L^{2n/(n+2s)}(\Omega;\er^{N})$ in the sense of Definition \ref{def:weaksol2}. Take concentric balls $\BBB\Subset \tilde \BBB \Subset \Omega$ and define $v \in \mathbb{X}_{u}^{s,2}(\BBB,\tilde \BBB)$ as the solution to
$$
\begin{cases}
\ -\mathcal{L}_{\mathds{A}}v=0\quad &\mbox{in} \ \ \BBB\\
\ v=u\quad &\mbox{in} \ \ \mathbb{R}^{n}\setminus \BBB
\end{cases}
$$
in the sense of Definition \ref{eqweaksol22}. 
Then, whenever 
$
d >   2n/(n+2s),
$
 it holds that 
 \eqn{marri}
$$
\nra{u-v}_{L^2(\BBB)} \leq c\snr{\BBB}^{2s/n-1/d}\nr{g}_{\mathcal M^d(\BBB)}
$$
where $c\equiv c (n,N, \Lambda, s,d)$.
\end{lemma} 
In \eqref{marri} there appears the Marcinkiewicz space $\mathcal{M}^{d}(\Omega;\er^{N})$, which is defined via
\eqn{mardef}
$$
g \in \mathcal M^{d}(\Omega;\er^{N})
\Longleftrightarrow
\nr{g}_{\mathcal M^{d}(\Omega)}^d
:=\sup_{\lambda\geq 0} \, \lambda^d
\snr{\{|g|> \lambda\}\cap \Omega}< \infty,
\quad d \geq 1.
$$
Finally, a nonlocal fractional Caccioppoli type inequality, see for instance \cite[Lemma 3.7]{parte1}. 
\begin{lemma}\label{caccioppola}
Under assumptions \eqref{linearell} with $0<s<1$, let $u$ be a weak solution to 
$-\mathcal{L}_{\scalebox{0.6}{$\sA$}}u=g\in L^2(\Omega;\er^{N})$ in $\Omega$ in the sense of Definition \ref{def:weaksol2}. 
Let $\BBB\Subset \Omega$ be a ball with radius $\rrr$; for all $0<\gamma<1$, the Caccioppoli inequality
\begin{flalign}\label{cacc}
&\notag \gamma^{n}\snra{u}_{s,2;\gamma \BBB}^2
\\ & \quad \le \frac{c}{(1-\gamma)^{2(n+s)}\rrr^{2s}}\left(\nra{u-u_0}_{L^{2}(\BBB)}^2+\tail(u-u_0;\BBB)\, \nra{u-u_0}_{L^{1}(\BBB)}\right)+c \rrr^{2s}\nra{g}_{L^{2}(\BBB)}^2
\end{flalign}
holds  with $c\equiv c(\data)$ for every $u_0 \in \er^{N}$. In particular, when $g\equiv0$, inequality \eqref{cacc} implies
\eqn{caccad}
$$
\rrr^{s}[u]_{s,2;\BBB/2} \leq c \nr{u}_{L^{2}(\BBB)} + c \rrr^{n/2}\tail(u;\BBB)
$$
with $c\equiv c (\data)$.
\end{lemma}
\subsection{Difference operators and Besov spaces} Here we collect a few known results allowing us to characterize local properties of Besov spaces via finite difference operators 
\eqn{enfasi}
$$
\tau_{h}^{k} \equiv \tau_{h, \mathcal A}^{k}
$$
where $k \in \en_0$,  $h \in \er^n$ and $\mathcal A \subset \er^n$ is a non-empty set. Let us recall their standard definitions, together with that of Besov spaces $B^{\tet}_{p,\infty}$. After defining the sets  $$\mathcal A_{k,\xi}:= \{ x \in \mathcal A \colon x+i\xi \in \mathcal A \textnormal{ for every } i=1,\ldots,k\}\,,$$ with $w \colon \mathcal A \to \er^{N}$, for every $k\geq 1$,  we define,
inductively, the operators 
\eqn{domania}
$$
\begin{cases}
\tau_{\xi}^{k+1}w(x):= \tau_{\xi}(\tau_{\xi}^{k} w)(x) \quad \mbox{for every $x\in \mathcal A_{k+1,\xi}$}, \\[3pt]
\tau_{\xi}w(x)\equiv \tau_{\xi}^1w(x):= w(x+\xi) -w(x)\quad \mbox{for every $x\in \mathcal A_{1,\xi}$},\\[3pt]
\tau_{\xi}^0w(x):=w(x)\quad \mbox{for every $x\in  \mathcal A=:  \mathcal  A_{0,\xi}$}   .
\end{cases}
$$
Instead, we let $\tau_{\xi}^{k}w(x)=0$ if $x \not \in \mathcal A_{k,\xi}$ for $k\geq 1$. See \cite[(1.384)]{Triebel3}. 
\begin{remark}\emph{In the following we shall often consider the operators in \eqref{enfasi} on subsets $\mathcal C$ such that $\mathcal C\subset \mathcal A_{k,\xi}$, in this case we shall omit denoting the dependence on $\mathcal A$ exactly as in \eqref{enfasi}. }
\end{remark}
The Besov space $B^{\tet}_{p,\infty}(\mathcal A)$ is here considered only for bounded, Lipschitz domains $\mathcal A \subset \er^n$ and the range of parameters $p \in [1,\infty)$, $\tet>0$. As explained in \cite[Proposition 4.6]{parte1}, an equivalent norm in
$B^{\tet}_{p,\infty}(\mathcal A)$ is given by 
\eqn{besovo0}
$$\nr{v}_{B^{\tet}_{p,\infty}(\mathcal A)}:=\nr{v}_{L^{p}(\mathcal A)}+\sup_{0<|h|\leq 
	 	1}|h|^{-\tet}\nr{\tau_{h, \mathcal A}^{l} v}_{L^{p}(\mathcal A)}$$
where $l > \tet$ is an integer; note that different choices of $l$ give rise to equivalent norms; see also  \cite[Theorem 1.118]{Triebel3}. The constants involved in the equivalences obviously depend on $l$ and the Lipschitz constant of $\mathcal A$; the dependence on $\mathcal A$ can be made explicit when $\mathcal A$ is a ball, via a simple scaling argument. A standard manipulation yields
\eqn{besovo}
$$\nr{v}_{B^{\tet}_{p,\infty}(\mathcal A)}\lesssim_{l,n,\tet}  \mathcal{h}_{0}^{-\tet}\nr{v}_{L^p(\mathcal A)}+\sup_{0<|h|\leq 
		\mathcal{h}_{0}}|h|^{-\tet}\nr{\tau_{h,\mathcal A}^{l} v}_{L^p(\mathcal A)}$$
whenever $\mathcal{h}_{0}\in (0,1)$. Besov spaces $B^{\tet}_{p,\infty}$  are special cases of the whole family of Besov  spaces $B^{\tet}_{p,q}$, $p,q>0$, for which we have $B^{\tet}_{p,p}\equiv W^{\tet,p}$ provided $\tet$ is not an integer and $p\geq 1$. Such spaces are described, for instance, in Triebel's classical books such as \cite{Triebel3,Triebel4}. 

We now present a few results aimed at formulating in a local fashion, and in a ready-to-use way, a few known embedding theorems in Besov spaces. The first lemma locally quantifies the Besov embedding $W^{t,q} \hookrightarrow B^{t}_{q,\infty}$. For the proof we refer for instance to \cite[Proposition 2.6, (2.11)]{bl}.  
\begin{lemma} \label{prop:embedding0} 
Let $B_{\rr} \Subset B_{r}\subset \er^n$ be concentric balls with $r\leq 1$, and let $w \in W^{t,q}(B_{r})$, with $0<t<1$ and $q\geq 1$. Then, with $\tau_{h}\equiv \tau_{h,B_r}$
 $$
\sup_{0 < |h|< r-\rr}\, \left\|\frac{\tau_{h} w}{ |h|^{t}}\right\|_{L^{q}(B_{\rr})} \lesssim_{n,q} (1-t)^{1/q} [w]_{t,q;B_{r}} + \left[\left(\frac{r}{r-\rr}\right)\frac{1}{r^t}+ \frac{1}{t^{1/q}(r-\rr)^t}\right] \nr{w}_{L^{q}(B_{r})}\,.
 $$
\end{lemma}
Properties based on double difference operators are instead in the following lemma, for which we refer to  \cite[Lemma 2.17]{verena}, \cite[Proposition 2.4]{bl}, \cite[Lemma 2.4]{garain}, \cite[Lemma 3.2]{dm25} for \eqref{immersione2} and to \cite[Lemma 2.9]{DKLN} for \eqref{immersione2b}. 
\begin{lemma}\, \hspace{-2.5mm}\label{l4dd}  
Let $B_{\rr} \Subset B_{r}\subset \er^n$ be concentric balls with $r\leq 1$, let $w\in L^{q}(B_{r};\mathbb{R}^{k})$, $q  > 1$. With $M\geq 0$, $\tet \in (1,2)$ and $\mathcal{h}_0\in (0,1)$ such that $\rr+2\mathcal{h}_0<r$, if 
$$
\sup_{0< |h|<\mathcal{h}_{0}/2^4}\left\|\frac{\tau_{h}^2w}{|h|^{\tet}}\right\|_{L^{q}(B_{\rr+ \mathcal{h}_{0}})}\leq M,
$$
 with $\tau_{h}^2 \equiv \tau_{h,B_r}^2$, then 
\eqn{immersione2}
$$
\nr{Dw}_{L^{q}(B_{\rr+\mathcal{h}_{0}/2})}\lesssim_{n,q}  \frac{M}{(\tet-1)(2-\tet)}+\frac{\nr{w}_{L^{q}(B_{\rr+\mathcal{h}_{0}})}}{(\tet-1)(2-\tet)\mathcal{h}_{0}^{\tet}}
$$
holds together with 
\eqn{immersione2b}
$$
[Dw]_{\beta, q;B_{\rr}} \lesssim_{n,q}  
\frac{M}{(\tet-1)(2-\tet)(\tet-1-\beta)^{1/q}}
+\frac{\nr{Dw}_{L^{q}(B_{\rr+\mathcal{h}_{0}/2})}}
{(\tet-1)(2-\tet)(\tet-1-\beta)^{1/q}\beta^{1/q}\mathcal{h}_{0}^{\tet}}
$$
whenever $\beta \in (0,\tet-1)$. 
\end{lemma}
We conclude with a lemma connecting $\tail$s and translations. 
\begin{lemma} \label{sl1}  If $w \in L^{1}_{2s}$, then 
\eqn{sl1-1}
  $$ 
\tail(w(\cdot+h);B_{\rr}(z_0)) \lesssim_{n,s} \left(\frac{\rr}{\tau}\right)^{2s}\tail(w;B_{\tau}(x_0))
 $$
 holds provided $|z_0-x_0|\leq \rr/4$, $0\leq |h|\leq \rr/8$ and $0< \tau \leq \rr/2$.
\end{lemma}
\begin{proof}
By the very definition of $\tail$, it is sufficient to prove \eqref{sl1-1} when $\tau = \rr/2$. Note that 
\eqn{sl11}
$$ 
[\er^n \setminus B_{\rr}(z_0)]+h \subset \er^n \setminus B_{\rr/2}(x_0) \subset \er^n \setminus B_{\rr/4}(x_0).
$$  
Indeed, if $|x-z_0|\geq \rr$, then 
$$|x+h-x_0| \geq |x-z_0|- |z_0-x_0|-|h|> \rr-\rr/4-\rr/4 =\rr/2.$$ 
Moreover, 
note that 
 \eqn{sl12}
 $$
 |x-x_0|\geq \frac{\rr}{2} \Longrightarrow \frac{1}{|x-z_0-h|^{n+2s}} \leq  \frac{4^{n+2s}}{|x-x_0|^{n+2s}}\,.
 $$ 
Indeed, observe that
$$ 
|x-x_0|
   \leq  |x-z_0-h|+  |z_0-x_0|+  |h|\leq  |x-z_0-h| + \frac{3\rr}{8}\leq  |x-z_0-h| + \frac{3}{4}|x-x_0|
$$
so that $|x-x_0| \leq 4  |x-z_0-h| $ from which \eqref{sl12} follows. 
Using \eqref{sl11} and \eqref{sl12} then we have 
 \begin{flalign*}  
 \tail(w(\cdot+h);B_{\rr}(z_0)) & = \rr^{2s} \int_{\er^n\setminus B_{\rr}(z_0)} \frac{|w(x+h)|}{|x+h -(z_0+h)|^{n+2s}}\dx \\
 & = \rr^{2s} \int_{[\er^n \setminus B_{\rr}(z_0)]+h} \frac{|w(x)|}{|x -(z_0+h)|^{n+2s}}\dx
  \\
 & \leq  \rr^{2s} \int_{\er^n \setminus B_{\rr/2}(x_0)} \frac{|w(x)|}{|x -z_0-h|^{n+2s}}\dx\\
  & \leq  4^{n+2s}\rr^{2s} \int_{\er^n \setminus B_{\rr/2}(x_0)} \frac{|w(x)|}{|x -x_0|^{n+2s}}\dx\\ & \leq 4^{n+4}\tail(w;B_{\rr/2}(x_0)), 
 \end{flalign*}
 that is \eqref{sl1-1} with $\tau =\rr/2$ and the proof is complete. 
\end{proof} 
We finally gather more standard embedding theorems in Besov spaces. These are special cases of more general facts that we report in the form and the notation that we shall need later on. As in the rest of the paper, the exponent $s$ denotes the one associated to the operator in \eqref{theone}. 
\begin{lemma}[Miscellanea of embeddings] Let $0 < t_1< t_2$, with $t_1$ being non-integer, $q \geq 1$ and $\BBB \subset \er^n$ be a ball, then 
\eqn{immergi}
$$\nr{v}_{W^{t_1,q}(\BBB)} \approx \nr{v}_{B^{t_1}_{q,q}(\BBB)} \leq c\nr{v}_{B^{t_2}_{q,\infty}(\BBB)}, 
$$
where $c\equiv c (n,N,t_1,t_2,q,\snr{\BBB})$. 
 Moreover, 
\eqn{immergi2}
$$
\nr{v}_{B^{s}_{q,\infty}(\BBB)}\leq c\nr{v}_{W^{s,q}(\BBB)}
$$
holds for $c\equiv c (n,N,s,q,\snr{\BBB})$. For  $\qq\geq  q \geq 1$,  $t >0$ the embedding inequality
\eqn{immergi01}
$$ 
 \nr{v}_{W^{s,\qq}(\BBB)}\leq c 
\nr{v}_{W^{t,q}(\BBB)}
$$
with $c\equiv c (n,N,s,t,q,\qq, \snr{\BBB})$ 
holds provided 
\eqn{immercond}
$$
 t -n/q \geq  s -n/\qq.
$$
In the case $ 1+n/\qq < t< 2$
\eqn{immergesmooth}
$$
		\nr{Dv}_{C^{0, t-1-n/\qq}(\BBB)}   \leq c \nr{v}_{W^{t,\qq}(\BBB)}\,, 
$$
with $c\equiv c (n,N,t,\qq, \snr{\BBB})$. 
\end{lemma}
\begin{proof} All the inequalities from \rif{immergi} to \rif{immergesmooth} can be derived in the case $B\equiv \ttB_1$; for general balls $B$ they then follow by a standard scaling argument. Inequality \eqref{immergi} is \cite[Theorem 1.107]{Triebel3}, while inequality \eqref{immergi2} is \cite[Theorem 2.85]{Triebel4}. Moreover, a version on different balls of \eqref{immergi2}, which would be still fine for our purposes,  can be obtained using Lemma \ref{prop:embedding0}. For \eqref{immergi01} we refer to \cite[(1.300)-(1.301)]{Triebel3} and related comments and references. Finally, as in this case $t$ is non-integer, we have 
$W^{t,\qq}(\BBB) \equiv  B^t_{\qq, \qq}(\BBB) \hookrightarrow B^{t-n/\qq}_{\infty,\infty}(\BBB) =   C^{1,t-1-n/\qq}(\BBB)$ (see the comments in the proof of Proposition 4.6 from \cite{parte1}). Therefore \eqref{immergesmooth} follows. 
\end{proof}
\subsection{Doubled difference operators}\label{doppidiff}
In addition to the standard difference operators $\tau_{h}\equiv \tau_{h, \er^n}$ defined in \eqref{enfasi} and  \eqref{domania}, we shall use another finite difference operator $\tilde \tau_{h}$, this time acting on maps $\psi \colon  \er^n\times \er^n  \to \er^{N}$. This is defined by 
\eqn{ttau2}
$$ 
\tilde \tau_{h}\psi(x,y):=\psi(x+h, y+h)-\psi(x,y),\quad h \in \er^n.
$$
This implies
\eqn{matobo2}
$$
\tilde \tau_{h}[\psi_1(x)+\psi_2(y)] = \tau_{h} \psi_1(x) + \tau_{h}\psi_2(y).
$$
The usual finite difference integration-by-parts formula then reads 
\eqn{matobo1pre}
$$
\int_{\er^n}\int_{\er^n} \langle \mathds{G}(x,y),    \tilde \tau_{-h}\psi(x,y)\rangle\frac{\dxy}{|x-y|^n}  = \int_{\er^n}\int_{\er^n} 
 \langle \tilde  \tau_{h}\mathds{G}(x,y),   \psi(x,y)\rangle\frac{\dxy}{|x-y|^n}
$$
whenever $\mathds{G}, \psi \colon  \er^n\times \er^n  \to \er^{N}$ are measurable and 
provided that 
$$
\frac{ \langle \mathds{G}(x,y),\psi(x,y)\rangle}{|x-y|^n}, \frac{ \langle \mathds{G}(x+h,y+h),\psi(x,y)\rangle}{|x-y|^n} \in L^1(\er^{2n}).
$$  
In the special case of operators as in \eqref{weaksolLL}, under the assumptions \eqref{condib}, the identity in \eqref{matobo1pre} becomes
\begin{flalign}
& \notag \int_{\er^n}\int_{\er^n} \langle \mathds{B}(x-y)w(x,y),    \tilde \tau_{-h}\psi(x,y)\rangle \frac{\dxy}{|x-y|^n} \\ & \qquad 
 = \int_{\er^n}\int_{\er^n} 
 \langle \mathds{B}(x-y)\tilde \tau_{h}w(x,y),   \psi(x,y)\rangle\frac{\dxy}{|x-y|^n} \label{matobo1}
\end{flalign}
whenever $\psi, w\colon  \er^n\times \er^n  \to \er^{N}$ are measurable and 
provided that 
$$
\frac{ \langle \mathds{B}(x-y)w(x,y),\psi(x,y)\rangle}{|x-y|^n}, \frac{ \langle \mathds{B}(x-y)w(x+h,y+h),\psi(x,y)\rangle}{|x-y|^n} \in L^1(\er^{2n}).
$$

\subsection{$s$-harmonic approximation}\label{harmonica}
In this section we restate the $s$-harmonic approximation lemma stated and proved  in \cite{parte1}. The main difference is that here we need a version where the matrix coefficients of the operator are not necessarily constant. This causes no change in either the proof or the statement with respect to the version reported below.  
\begin{lemma}\label{shar}
Let $\sA\colon \er^{2n}\to \er^{N\times N}$ be a measurable matrix field satisfying \eqref{linearell}. Let $\delta_0$ be a number such that $0 <s <s+\delta_{0}<1$; let $B\Subset \ti{B}\Subset \mathbb{R}^{n}$ be concentric balls with radii belonging to  $(1/32,1)$ and denote  $\textnormal{\texttt{d}}:=\dist(B,\partial \ti{B})>0$. Let $v\in W^{s,2}(\er^n;\er^{N})$, $g\in L^{\infty}(\mathbb{R}^{n};\er^{N})$ be such that 
\eqn{sh.0}
$$
\nr{v}_{W^{s,2}(\mathbb{R}^{n})}+\nr{v}_{W^{s+\delta_{0},2}(\er^n)}+ \tail(v;\ti{B}) + \nr{g}_{L^{\infty}(\mathbb{R}^{n})}\le c_{0}
$$
for some $c_{0}>0$. There exist numbers $t,p$
\eqn{tpnumeri}
$$s < t \equiv t(\data, \delta_0)<1 , \quad  p\equiv p(\data, \delta_0)>2, $$ and a constant 
\eqn{chconst}
$$\cchh\equiv \cchh(\data,c_{0},\delta_{0}, \textnormal{\texttt{d}})\geq 1$$ such that, if for $\eps\in (0,1)$
\eqn{sh.1}
$$ 
\left|\int_{\mathbb{R}^{n}}\int_{\mathbb{R}^{n}}\langle \sA(x,y)\frac{v(x)-v(y)}{|x-y|^s},\frac{\varphi(x)-\varphi(y)}{|x-y|^s}\rangle\frac{\dxy}{|x-y|^{n}}-\int_{B}\langle g,\varphi\rangle\dx\right|\le\left(\frac{\eps}{\cchh}\right)^{\frac{2p}{p-2}} 
$$
holds for all $\varphi\in C^{0,t}(\er^n;\er^{N})$ vanishing outside $B$ such that $[\varphi]_{0,t;\er^n}\leq 1$, and if $h\in\mathbb{X}^{s,2}_{v}(B,\ti{B})$ solves
$$
\begin{cases}
\ -\mathcal{L}_{\mathds{A}}h=g\quad &\mbox{in} \ \ B\\
\ h=v\quad &\mbox{in} \ \ \mathbb{R}^{n}\setminus B\,,
\end{cases}
$$
in the sense of Definition \ref{eqweaksol22}, then 
\eqn{sh.5}
$$
\begin{cases}
\, \nr{h}_{W^{s,2}(\ti{B})}
+\tail(h-(h)_{B};B)\le c\equiv c(\data, c_{0})\\[4pt]
 \nr{h-v}_{L^2(B)}\leq \frac{c(n,s)}{\sqrt{\textnormal{\texttt{d}}}}\eps\,.
\end{cases}
$$
\end{lemma}

\section{Fine properties of fractional Sobolev functions}\label{thefines}
In this section we collect a few useful characterizations of the set where fractional Sobolev functions admit a precise representative and of where certain related convergence properties take place. As usual, in what follows, $\Omega \subset \er^n$ will denote a bounded open subset  and $n\geq 2$. 

In \cite{KMS1}  an inequality was proved allowing to connect the $\tail$s on two concentric balls $B_{\rr}\subset B_{\rrr}$,  i.e., 
\begin{flalign}\label{scatail}
\tail(w-(w)_{B_{\rr}};B_{\rr}) &\lesssim_{n,s} \left(\frac{\rr}{\rrr}\right)^{2s}\tail(w-(w)_{B_{\rrr}};B_{\rrr})+\left(\frac{\rr}{\rrr}\right)^{2s}\nra{w-(w)_{B_{\rrr}}}_{L^1(B_{\rrr})}\nonumber \\  
&\qquad +\int_{\rr}^{\rrr}\left(\frac{\rr}{\lambda}\right)^{2s}\nra{w-(w)_{B_{\lambda}}}_{L^1(B_{\lambda})}\frac{\dlam}{\lambda}, 
\end{flalign}
that holds for every $w\in L^{1}_{2s}$ with $s \in (0,1)$. This holds together with
\begin{flalign}\label{scatailancora}
	\tail(w;B_{\rr}) &\lesssim_{n,s} \left(\frac{\rr}{\rrr}\right)^{2s}\tail(w;B_{\rrr}) + \left(\frac{\rr}{\rrr}\right)^{2s} \nra{w}_{L^1(B_{\rrr})} +\int_{\rr}^{\rrr}\left(\frac{\rr}{\lambda}\right)^{2s}\nra{w}_{L^1(B_{\lambda})}\frac{\dlam}{\lambda}\,. 
\end{flalign}
See \cite[Section 3]{parte1}. 
As an easy consequence we have 
\eqn{scatailancora2}
$$
	\tail(w;B_{\rr}) \lesssim_{n,s} \left(\frac{\rr}{\rrr}\right)^{2s}\tail(w;B_{\rrr}) + \left(\frac{\rrr}{\rr}\right)^{n} \nra{w}_{L^1(B_{\rrr})}.
$$ 
Moreover, as in \cite[Lemma 3.2]{parte1}, if $B_{\rr}(x_{1})\subset B_{\rrr}(x_{2})$ are two not necessarily concentric balls, then
\begin{flalign}\label{scatail.1}
\tail(w-(w)_{B_{\rr}(x_{1})};B_{\rr}(x_{1}))&\lesssim_{n,s} \left(\frac{\rr}{\rrr}\right)^{2s}\left(\frac{\rrr}{\rrr-\snr{x_{1}-x_{2}}}\right)^{n+2s}\tail(w-(w)_{B_{\rrr}(x_{2})};B_{\rrr}(x_{2}))\nonumber \\
& \qquad +\left(\frac{\rrr}{\rr}\right)^{n}\nra{w-(w)_{B_{\rrr}(x_{2})}}_{L^{1}(B_{\rrr}(x_{2}))}\,. 
\end{flalign}
We now give, in Lemma \ref{dedicata} below, a higher order version of \eqref{scatail}, which, unlike \eqref{scatail}, requires the condition $s>1/2$. For this we need the following:
\begin{definition}[Affine approximation]
With $w\in W^{1,1}(B_{\rr}(x);\er^{N}) $, we define the affine map
\eqn{affinecanonica}
$$
\ell_{x, \rr}^w(y):=(Dw)_{B_{\rr}(x)}(y-x)+ (w)_{B_{\rr}(x)}, \qquad y \in \er^n.
$$ 
\end{definition}
\begin{lemma}\label{dedicata}
Let $w\in W^{1,1}(B_{\rrr}(x);\er^{N})\cap L^{1}_{2s}$, $s>1/2$. Then, for $0< \rr\leq \rrr$
\begin{flalign}
\notag & \tail(w-\ell_{x, \rr}^w;B_{\rr}(x)) \lesssim_{n,N,s} \left(\frac{\rr}{\rrr}\right)^{2s}\tail(w-\ell_{x, \rrr}^w;B_{\rrr}(x))
\\ & \qquad\qquad  +\rrr\left(\frac{\rr}{\rrr}\right)^{2s}\nra{Dw-(Dw)_{B_{\rrr}(x)}}_{L^1(B_{\rrr}(x))}+\left(\frac{\rr}{\rrr}\right)^{2s}\nra{w-(w)_{B_{\rrr}(x)}}_{L^1(B_{\rrr}(x))}\nonumber \\  
&\qquad \qquad  \qquad +\int_{\rr}^{\rrr}\left(\frac{\rr}{\lambda}\right)^{2s}\nra{Dw-(Dw)_{B_{\lambda}(x)}}_{L^1(B_{\lambda}(x))} \dlam\nonumber \\ & \qquad \qquad \qquad \quad  +\int_{\rr}^{\rrr}\left(\frac{\rr}{\lambda}\right)^{2s}\nra{w-(w)_{B_{\lambda}(x)}}_{L^1(B_{\lambda}(x))}\frac{\dlam}{\lambda}\,.\label{scatalunga}
\end{flalign}
\end{lemma}
\begin{proof} In the following, all the balls will be centred at $x$ and we denote by $\mathcal S$ the quantity in the right-hand side of \eqref{scatalunga}; $c$ will denote a generic constant depending on $n,N,s$. Decompose
\begin{flalign*}
\tail(w-\ell_{x, \rr}^w;B_{\rr}) & \leq  \rr^{2s} \int_{\er^n\setminus B_{\rrr/4}} \frac{\snr{w(y)-\ell_{x, \rr}^w(y)}}{\snr{y-x}^{n+2s}}\dy  + \rr^{2s} \int_{B_{\rrr}\setminus
 B_{\rr}} \frac{\snr{w(y)-\ell_{x, \rr}^w(y)}}{\snr{y-x}^{n+2s}}\dy 
 \\
 & =: T_1 +T_2\,.
\end{flalign*}
Therefore we need to prove that 
\eqn{siamo}
$$
T_1 +T_2  \leq c  \mathcal S.
$$
Observing that
\begin{flalign*}
 \int_{\er^n\setminus B_{\rrr/4}} \frac{\snr{[(Dw)_{B_\rrr}-(Dw)_{B_\rr}] (y-x)}}{\snr{y-x}^{n+2s}}\dy & \leq \snr{(Dw)_{B_\rrr}-(Dw)_{B_\rr}} \int_{\er^n\setminus B_{\rrr/4}} \frac{\dy}{\snr{y-x}^{n+2s-1}}  \\
 & \lesssim \frac{\snr{(Dw)_{B_\rrr}-(Dw)_{B_\rr}}}{(2s-1)\rrr^{2s-1}}
\end{flalign*}
we easily obtain
\begin{flalign}
T_1 
  & \leq  c \left(\frac{\rr}{\rrr}\right)^{2s} \tail(w-\ell_{x, \rrr}^w;B_{\rrr/4}) \notag \\
  & \qquad +c
\left(\frac{\rr}{\rrr}\right)^{2s} \left(\snr{(w)_{B_\rrr}-(w)_{B_\rr}} +\rrr\snr{(Dw)_{B_\rrr}-(Dw)_{B_\rr}}\right). \label{t11}
\end{flalign}
As for the first term in \eqref{t11}, using \eqref{scatailancora2} and Poincar\'e's inequality we have 
\begin{flalign}
 \notag \left(\frac{\rr}{\rrr}\right)^{2s}\tail(w-\ell_{x, \rrr}^w;B_{\rrr/4}) & \leq c  \left(\frac{\rr}{\rrr}\right)^{2s}\tail(w-\ell_{x, \rrr}^w;B_{\rrr})+c \left(\frac{\rr}{\rrr}\right)^{2s}\nra{w-\ell_{x, \rrr}^w}_{L^1(B_{\rrr})} \\  & \leq  c \left(\frac{\rr}{\rrr}\right)^{2s} \tail(w-\ell_{x, \rrr}^w;B_{\rrr})+c\rrr \left(\frac{\rr}{\rrr}\right)^{2s}\nra{Dw-(Dw)_{B_{\rrr}}}_{L^1(B_{\rrr})}. \label{addizionale}
\end{flalign}
When $\rr \geq \rrr/4$ this immediately implies 
\eqn{t11b}
$$T_1 \leq c  \mathcal S\,.$$
Indeed, the last terms  in \eqref{t11} can be estimated  observing that Jensen's inequality implies 
\eqn{addizionale2}
$$
\snr{(w)_{B_\rrr}-(w)_{B_\rr}} +\rrr\snr{(Dw)_{B_\rrr}-(Dw)_{B_\rr}} \lesssim_n \nra{w-(w)_{B_{\rrr}}}_{L^1(B_{\rrr})}+\rrr\nra{Dw-(Dw)_{B_{\rrr}}}_{L^1(B_{\rrr})}
$$ 
and again using that $\rr\approx  \rrr$. 
Connecting the information in \eqref{addizionale} and \eqref{addizionale2} to \eqref{t11} yields \eqref{t11b} in the case it is $\rr \geq \rrr/4$. We therefore reduce to the case  $\rr < \rrr/4$, where we take $\gamma \in (2,4]$ and integer $k\geq 2$ such that
$
\rr/\rrr = 1/\gamma^k
$. 
Then we have, again using Jensen's inequality and \eqref{faccia}
\begin{flalign*}
\snr{(Dw)_{B_\rrr}-(Dw)_{B_\rr}} & = \snr{(Dw)_{B_{\gamma^{k}\rr}}-(Dw)_{B_\rr}} \\
& \leq \sum_{i=0}^{k-1}  \snr{(Dw)_{B_{\gamma^{i+1}\rr}}-(Dw)_{B_{\gamma^{i}\rr}}} \\
& \leq \sum_{i=1}^{k}  \gamma^{n}\mint_{B_{\gamma^{i}\rr}}  |Dw-(Dw)_{B_{\gamma^{i}\rr}}| \dy \\
& \leq c  \sum_{i=1}^{k-1}  \int_{\gamma^{i}\rr}^{\gamma^{i+1}\rr} \mint_{B_{\gamma^{i}\rr}}  |Dw -(Dw)_{B_{\gamma^{i}\rr}}| \dy \frac{\dd\lambda}{\lambda}+  c\mint_{B_{\gamma^{k}\rr}}  |Dw -(Dw)_{B_{\gamma^{k}\rr}}| \dy\\
& \leq c  \sum_{i=1}^{k-1}  \int_{\gamma^{i}\rr}^{\gamma^{i+1}\rr} 
\nra{Dw-(Dw)_{B_{\lambda}}}_{L^1(B_{\lambda})}
 \frac{\dd \lambda}{\lambda}+c\nra{Dw-(Dw)_{B_{\rrr}}}_{L^1(B_{\rrr})}\\
 & \leq  c     \int_{\rr }^{\rrr} 
\nra{Dw-(Dw)_{B_{\lambda}}}_{L^1(B_{\lambda})}
 \frac{\dd \lambda}{\lambda}+c\nra{Dw-(Dw)_{B_{\rrr}}}_{L^1(B_{\rrr})}
\end{flalign*}
and therefore, as  $2s>1$, estimating $\rrr^{1-2s}\leq \lambda^{1-2s}$ when $\lambda \leq \rrr$, we have 
\begin{flalign*}
& \rrr\left(\frac{\rr}{\rrr}\right)^{2s}\snr{(Dw)_{B_\rrr}-(Dw)_{B_\rr}} \\
 & \quad \leq c  
 \int_{\rr }^{\rrr}   \left(\frac{\rr}{\lambda}\right)^{2s}
\nra{Dw-(Dw)_{B_{\lambda}}}_{L^1(B_{\lambda})}
\dd \lambda +c\rrr\left(\frac{\rr}{\rrr}\right)^{2s} \nra{Dw-(Dw)_{B_{\rrr}}}_{L^1(B_{\rrr})}\,.
\end{flalign*}
Similarly, we have 
\begin{flalign*}
&\left(\frac{\rr}{\rrr}\right)^{2s} \snr{(w)_{B_\rrr}-(w)_{B_\rr}} \\
 & \quad  \leq c 
 \int_{\rr }^{\rrr}   \left(\frac{\rr}{\lambda}\right)^{2s}
\nra{w-(w)_{B_{\lambda}}}_{L^1(B_{\lambda})}
 \frac{\dd \lambda}{\lambda}+c\left(\frac{\rr}{\rrr}\right)^{2s} \nra{w-(w)_{B_{\rrr}}}_{L^1(B_{\rrr})}\,.
\end{flalign*}
Using the inequalities in the last two displays together with \eqref{t11} and \eqref{addizionale} again yields \eqref{t11b} that now is established in the full range $\rr \in (0,\rrr]$.  We now prove that 
\eqn{t11c}
$$T_2 \leq c \mathcal S\,.$$
When $\rr \geq  \rrr/4$, we can simply estimate 
$$T_2\leq \rr^{2s} \int_{\er^n\setminus
 B_{\rrr/4}} \frac{\snr{w(y)-\ell_{x, \rr}^w(y)}}{\snr{y-x}^{n+2s}}\dy= T_1 \stackleq{t11b} c \mathcal S,$$
that is, \eqref{t11c}, and therefore we again reduce to the case $\rr <  \rrr/4$. Writing $\rr/\rrr = 1/\gamma^k$ for $\gamma \in (2,4]$ and $k\geq 2$ as done for $T_1$, we have 
\eqn{tt22}
$$
T_2 =\rr^{2s} \sum_{j=0}^{k-1} \int_{B_{\gamma^{j+1}\rr} \setminus B_{\gamma^{j}\rr}} \frac{|w(y)-\ell_{x, \rr}^w(y)|}{|y-x|^{n+2s}}\dy \leq c \sum_{j=1}^{k} \frac{1}{\gamma^{2sj}}
\mint_{B_{\gamma^{j}\rr}}  |w-\ell_{x, \rr}^w| \dy
$$
and again, for $j \in \{1, \ldots, k\}$
$$
\mint_{B_{\gamma^{j}\rr}}  |w-\ell_{x, \rr}^w| \dy \leq 
\mint_{B_{\gamma^{j}\rr}}  |w-\ell_{x, \gamma^{j}\rr}^w| \dy+\sum_{i=0}^{j-1}\mint_{B_{\gamma^{j}\rr}}  |\ell_{x, \gamma^{i+1}\rr}^w-\ell_{x, \gamma^{i}\rr}^w| \dy.
$$
Note that, for every index $i \in \{0,\dots, j-1\}$ we can further estimate 
\begin{flalign*}
\mint_{B_{\gamma^{j}\rr}}  |\ell_{x, \gamma^{i+1}\rr}^w-\ell_{x, \gamma^{i}\rr}^w| \dy & \leq c\gamma^{j}\rr
\snr{(Dw)_{B_{\gamma^{i+1}\rr}}-(Dw)_{B_{\gamma^{i}\rr}}}+
c\snr{(w)_{B_{\gamma^{i+1}\rr}}-(w)_{B_{\gamma^{i}\rr}}}\\
& \leq    c\gamma^{j}\rr\mint_{B_{\gamma^{i+1}\rr}}  |Dw-(Dw)_{B_{\gamma^{i+1}\rr}}| \dy+c  \mint_{B_{\gamma^{i+1}\rr}}  |w-(w)_{B_{\gamma^{i+1}\rr}}| \dy\,.
\end{flalign*}
Moreover, Poincaré's inequality gives
$$
\mint_{B_{\gamma^{j}\rr}}  |w-\ell_{x, \gamma^{j}\rr}^w| \dy \leq 
c\gamma^{j}\rr\mint_{B_{\gamma^{j}\rr}}  |Dw-(Dw)_{B_{\gamma^{j}\rr}}| \dy.
$$
Merging the content of the last three displays finally yields
$$
\mint_{B_{\gamma^{j}\rr}}  |w-\ell_{x, \rr}^w| \dy\leq 
c\gamma^{j}\rr \sum_{i=1}^{j}
 \mint_{B_{\gamma^{i}\rr}}  |Dw-(Dw)_{B_{\gamma^{i}\rr}}| \dy
+ c \sum_{i=1}^{j}
 \mint_{B_{\gamma^{i}\rr}}  |w-(w)_{B_{\gamma^{i}\rr}}| \dy.
$$ 
Using this last estimate in \eqref{tt22}, and using again that $2s>1$, $\gamma \geq 2$, and that $\rr/\rrr = 1/\gamma^k$, we find
\begin{flalign*}
T_2 & \leq c\sum_{j=1}^{k}  \sum_{i=1}^{j}
\frac{ \rr }{\gamma^{(2s-1)j}}  \mint_{B_{\gamma^{i}\rr}}  |Dw-(Dw)_{B_{\gamma^{i}\rr}}| \dy\\
& \qquad  + c \sum_{j=1}^{k}  \sum_{i=1}^{j}
\frac{1}{\gamma^{2sj}}  \mint_{B_{\gamma^{i}\rr}}  |w-(w)_{B_{\gamma^{i}\rr}}| \dy\\
& = c\sum_{i=1}^{k}  \frac{ \rr }{\gamma^{(2s-1)i}} \mint_{B_{\gamma^{i}\rr}}  |Dw-(Dw)_{B_{\gamma^{i}\rr}}| \dy\sum_{j=i}^{k}\frac 1{\gamma^{(2s-1)(j-i)}}\\
& \qquad  + c  \sum_{i=1}^{k}  \frac{1}{\gamma^{2si}} \mint_{B_{\gamma^{i}\rr}}  |w-(w)_{B_{\gamma^{i}\rr}}| \dy\sum_{j=i}^{k}\frac{1}{\gamma^{2s(j-i)}} \\
& \leq  c \sum_{i=1}^{k}  \frac{ \rr}{\gamma^{(2s-1)i}} \mint_{B_{\gamma^{i}\rr}}  |Dw-(Dw)_{B_{\gamma^{i}\rr}}| \dy\\ & \qquad  +c\sum_{i=1}^{k} \frac{1}{\gamma^{2si}} \mint_{B_{\gamma^{i}\rr}}  |w-(w)_{B_{\gamma^{i}\rr}}| \dy \\
& \leq \frac{c}{\log \gamma} \sum_{i=1}^{k-1}  \int_{\gamma^{i}\rr}^{\gamma^{i+1}\rr}  \frac{\rr}{\gamma^{(2s-1)i}}\mint_{B_{\gamma^{i}\rr}}  |Dw-(Dw)_{B_{\gamma^{i}\rr}}| \dy \frac{\dd\lambda}{\lambda}\\ & \qquad \qquad + \frac{c\rr}{\gamma^{(2s-1)k}}\mint_{B_{\gamma^{k}\rr}}  |Dw-(Dw)_{B_{\gamma^{k}\rr}}| \dy\\
& \qquad  + \frac{c}{\log \gamma}  \sum_{i=1}^{k-1}  \int_{\gamma^{i}\rr}^{\gamma^{i+1}\rr} \frac{1}{\gamma^{2si}}\mint_{B_{\gamma^{i}\rr}}  |w-(w)_{B_{\gamma^{i}\rr}}| \dy \frac{\dd\lambda}{\lambda}+ \frac{c}{\gamma^{2sk}}\mint_{B_{\gamma^{k}\rr}}  |w-(w)_{B_{\gamma^{k}\rr}}| \dy\\
& \leq c  \sum_{i=1}^{k-1}  \int_{\gamma^{i}\rr}^{\gamma^{i+1}\rr} \left(\frac{\rr}{\lambda}\right)^{2s}
\nra{Dw-(Dw)_{B_{\lambda}}}_{L^1(B_{\lambda})}
\dd \lambda +c\rrr\left(\frac{\rr}{\rrr}\right)^{2s}\nra{Dw-(Dw)_{B_{\rrr}}}_{L^1(B_{\rrr})}\\
& \qquad  + c  \sum_{i=1}^{k-1}  \int_{\gamma^{i}\rr}^{\gamma^{i+1}\rr} \left(\frac{\rr}{\lambda}\right)^{2s}
\nra{w-(w)_{B_{\lambda}}}_{L^1(B_{\lambda})}
 \frac{\dd \lambda}{\lambda}+c\left(\frac{\rr}{\rrr}\right)^{2s}\nra{w-(w)_{B_{\rrr}}}_{L^1(B_{\rrr})}\\
 & \leq  c     \int_{\rr }^{\rrr} \left(\frac{\rr}{\lambda}\right)^{2s}
\nra{Dw-(Dw)_{B_{\lambda}}}_{L^1(B_{\lambda})}
\dd \lambda+c\rrr\left(\frac{\rr}{\rrr}\right)^{2s}\nra{Dw-(Dw)_{B_{\rrr}}}_{L^1(B_{\rrr})}\\
 & \qquad + c     \int_{\rr }^{\rrr} \left(\frac{\rr}{\lambda}\right)^{2s}
\nra{w-(w)_{B_{\lambda}}}_{L^1(B_{\lambda})}
 \frac{\dd \lambda}{\lambda}+c\left(\frac{\rr}{\rrr}\right)^{2s}\nra{w-(w)_{B_{\rrr}}}_{L^1(B_{\rrr})}
\end{flalign*}
from which we deduce \eqref{t11c} in the full range $\rr \leq \rrr$. This, together with \eqref{t11b}, gives \eqref{siamo} thereby completing the proof of the lemma. Note that the constants appearing in the last line of the above display, in
\eqref{siamo}, and in \eqref{scatalunga} depend on $s$ and blow up as
$s\downarrow1/2$. 
\end{proof}
\begin{lemma}\label{prelirid} Let $w\in W^{1,2}_{\loc}(\Omega;\er^{N})\cap L^{1}_{2s}$, $s>1/2$. If $x\in \Omega$, then 
\eqn{implicona0}
$$
 \lim_{\rr\to 0}\,  \rr^{1-s}\nra{Dw}_{L^2(B_{\rr}(x))}=0 \Longrightarrow \lim_{\rr\to 0}\,  \rr^{-s}  \tx{E}_{w}(\ell_{x, \rr}^w;x,\rr) =0, 
$$
where $\ell_{x, \rr}^w$ is the affine map defined in \eqref{affinecanonica}. 
As a consequence, if 
$$
\Sigma(w):=\Big\{x \in \Omega \, \colon \, \limsup_{\rr\to 0}\,  \rr^{-s}  \tx{E}_{w}(\ell_{x, \rr}^w;x,\rr)  >0\Big\},
$$
then 
\eqn{thenthen}
$$
\begin{cases}
\displaystyle \Sigma(w) \subset \Sigma_0(w):= \Big\{x \in \Omega \, \colon \, \limsup_{\rr\to 0}\, \rr^{1-s}\nra{Dw}_{L^2(B_{\rr}(x))}>0\Big\}\\[4mm]
\ddim(\Sigma(w))\leq \ddim(\Sigma_0(w))\leq n-2+2s<n\,.
\end{cases}
$$
\end{lemma}
\begin{proof} In the following we fix a point $x \in \Omega$, all the balls are centred at $x$ and, as usual, we abbreviate $B_{\rr}\equiv B_{\rr}(x)$. We prove \eqref{implicona0} separately, i.e.,  
\eqn{implicona}
$$
 \lim_{\rr\to 0}\,  \rr^{1-s}\nra{Dw}_{L^2(B_{\rr})}=0 \Longrightarrow \begin{cases}
\displaystyle \limsup_{\rr\to 0}\,  \rr^{-s} \nra{w-\ell_{x, \rr}^w}_{L^2(B_{\rr})} =0\\
\displaystyle \limsup_{\rr\to 0}\,  \rr^{-s} \tail(w-\ell_{x, \rr}^w;B_{\rr})=0\,.
\end{cases}.
$$
We note 
that the first implication in \eqref{implicona} simply follows by  the standard Poincaré's inequality
$$
\rr^{-s}\nra{w-\ell_{x, \rr}^w}_{L^2(B_{\rr})}\leq c \rr^{1-s} \nra{Dw-(Dw)_{B_{\rr}}}_{L^2(B_{\rr})}\leq c \rr^{1-s}\nra{Dw}_{L^2(B_{\rr})}.
$$
For the second implication in \eqref{implicona}, consider a ball $B_{\rrr} \subset \Omega$ and $\rr \leq \rrr$. We first show that 
\eqn{leo1}
$$
\begin{cases}
\displaystyle \lim_{\rr\to 0}\frac{1}{\rr^{s}}\int_{\rr}^{\rrr}\left(\frac{\rr}{\lambda}\right)^{2s}\nra{Dw-(Dw)_{B_{\lambda}}}_{L^1(B_{\lambda})} \dlam =0\\[4mm]
\displaystyle
\lim_{\rr\to 0}\frac{1}{\rr^{s}}\int_{\rr}^{\rrr}\left(\frac{\rr}{\lambda}\right)^{2s}\nra{w-(w)_{B_{\lambda}}}_{L^1(B_{\lambda})}\frac{\dlam}{\lambda}=0\,.
\end{cases}
$$
For \eqref{leo1}$_1$, observe that the function 
$$
\rr \mapsto h(\rr):=\int_{\rr}^{\rrr} \frac{1}{\lambda^{2s}}\nra{Dw-(Dw)_{B_{\lambda}}}_{L^1(B_{\lambda})} \dlam
$$
is trivially $C^1$-regular in $(0,\rrr)$ (by absolute continuity of the integral) and non-increasing and therefore the limit 
$$
\lim_{\rr\to 0}\int_{\rr}^{\rrr} \frac{1}{\lambda^{2s}}\nra{Dw-(Dw)_{B_{\lambda}}}_{L^1(B_{\lambda})} \dlam
$$
always exists. We can always assume that such a limit is infinite otherwise \eqref{leo1}$_1$ is obvious. At this stage, applying l'Hôpital criterion
we obtain
\begin{flalign*}
\notag \lim_{\rr\to 0}\frac 1{\rr^{s}}\int_{\rr}^{\rrr}\left(\frac{\rr}{\lambda}\right)^{2s}\nra{Dw-(Dw)_{B_{\lambda}}}_{L^1(B_{\lambda})} \dlam &=  \lim_{\rr\to 0} \frac{\rr^{1-s}}{s}\nra{Dw-(Dw)_{B_{\rr}}}_{L^{1}(B_{\rr})}\\ \notag 
&\lesssim   \lim_{\rr\to 0} \rr^{1-s}\nra{Dw-(Dw)_{B_{\rr}}}_{L^{2}(B_{\rr})}\\
& \lesssim  \lim_{\rr\to 0}\,  \rr^{1-s}\nra{Dw}_{L^2(B_{\rr})}=0
\end{flalign*}
so that \eqref{leo1}$_1$ follows again. For \eqref{leo1}$_2$ we can argue similarly. Assuming that 
$$
\lim_{\rr\to 0}\int_{\rr}^{\rrr}\left(\frac{1}{\lambda}\right)^{2s}\nra{w-(w)_{B_{\lambda}}}_{L^1(B_{\lambda})}\frac{\dlam}{\lambda}=\infty, 
$$
otherwise \eqref{leo1}$_2$ follows, by l'Hôpital criterion and Poincaré's  inequality, we have 
\begin{flalign*}
 \lim_{\rr\to 0}\frac 1{\rr^{s}}\int_{\rr}^{\rrr}\left(\frac{\rr}{\lambda}\right)^{2s}\nra{w-(w)_{B_{\lambda}}}_{L^{1}(B_{\lambda})}\frac{\dlam}{\lambda} &=  \lim_{\rr\to 0} \frac{\rr^{-s}}{s}\nra{w-(w)_{B_{\rr}}}_{L^{1}(B_{\rr})}\\
  &\lesssim  \lim_{\rr\to 0} \rr^{-s}\nra{w-(w)_{B_{\rr}}}_{L^{2}(B_{\rr})}\\
& \lesssim  \lim_{\rr\to 0}\,  \rr^{1-s}\nra{Dw}_{L^2(B_{\rr})}=0\,.
\end{flalign*}
We can now use \eqref{scatalunga}. Multiplying both sides by $\rr^{-s}$ and letting $\rr\to 0$, thanks to \eqref{leo1} we obtain the second implication in  \eqref{implicona}.  Moreover, the dimension estimate in \eqref{thenthen}$_2$ follows by standard measure theoretic criteria  \cite[Proposition 2.7]{giusti} (``Giusti's lemma"). \end{proof}
\begin{lemma}\label{riducilemma}
 Let $w\in W^{t,p}_{\loc}(\Omega;\er^{k})$, $t \in (0,1)$ and $p\geq1$. If $\gamma> 0$ is such that $(\gamma+t)p\leq n$, then 
 $$
\Sigma_{\gamma}:=\Big\{x \in \Omega \, \colon \, \limsup_{\rr\to 0}\,  \rr^{\gamma}\nra{w}_{L^{p}(B_{\rr}(x))} >0\Big\}\Longrightarrow \ddim(\Sigma_\gamma)\leq n-(\gamma+t)p.
$$
If $(\gamma+t)p >n$, then $\Sigma_{\gamma}$ is empty. 
\end{lemma}
\begin{proof}  
By a standard covering argument, it is enough to prove the assertion locally;
therefore we may assume that $w\in W^{t,p}(\Omega;\er^k)$.
\\
{\em The case $(\gamma+t)p\leq n$}. Let us recall that for $\beta>0$ such that $p\beta \leq n$, we have 
\eqn{defir1}
$$
\mathcal R_\beta:=\Big\{x \in \Omega \, \colon \, \limsup_{\rr\to 0}\,  \rr^{\beta}\snra{w}_{t,p;B_{\rr}(x)} >0\Big\} \Longrightarrow  \ddim(\mathcal R_\beta)\leq n- p\beta.
$$
This is actually a fractional version of the already mentioned Giusti's lemma \cite[Proposition 2.7]{giusti}, that has been observed in \cite{min03} (in fact, $\mathcal R_\beta$ is obviously empty when $p\beta =n$). Note that $t$ does not come into play in \eqref{defir1}. To continue, we set
\eqn{sisis1}
 $$
\Sigma_{\gamma, 1}:=\Big\{x \in \Omega \, \colon \, \limsup_{\rr\to 0}\,  \rr^{\gamma}\snr{(w)_{B_{\rr}(x)}} >0 \Big\}.
$$
and prove that 
\eqn{riduci1}
$$\ddim(\Sigma_{\gamma, 1})\leq n-(\gamma+t)p\,.$$ For this we fix $\eps $ such that $0<\eps <\gamma + t $. Then we take any $x\in \Omega$ such that the following maximal type function is bounded
$$
M(x):= \sup_{B_{\rr}(x)\subset \Omega, 
\rr \leq 1}  \rr^{\gamma + t-\eps} \snra{w}_{t,p;B_{\rr}(x)} <\infty.
$$ 
Take $B_{\sigma}(x)\subset \Omega$ with $\sigma \in (0,1)$, and define $\sigma_k:= \sigma/2^k$, $k \in \en_0$. If $\rr$ is such that $\sigma_{k+1} \leq  \rr \leq \sigma_k$ for some $k\in \en_0$, via Jensen's inequality and \eqref{faccia} we find
\begin{flalign*}
\notag \rr^{\gamma}
\snr{(w)_{B_{\rr}(x)}-(w)_{B_{\sigma_{k}}(x)}}
\notag& \leq c \rr^{\gamma}  \nra{w-(w)_{B_{\sigma_{k}}(x)} }_{L^1(B_{\rr}(x))}\\
\notag& \leq c \sigma_{k}^{\gamma}  \nra{w-(w)_{B_{\sigma_{k}}(x)} }_{L^{p}(B_{\sigma_{k}}(x))} 
\\
 & \leq c \sigma_{k}^{\gamma + t} \snra{w}_{t,p;B_{\sigma_{k}}(x)} \\
 &  \leq c M(x)\sigma_{k}^{\eps}   .
\end{flalign*} 
Using also the above inequality we have 
\begin{flalign}
\notag \rr^{\gamma} 
\snr{(w)_{B_{\rr}(x)}} &\leq  \rr^{\gamma} 
\snr{(w)_{B_{\sigma_{k}}(x)}} + \rr^{\gamma} 
\snr{(w)_{B_{\rr}(x)}-(w)_{B_{\sigma_{k}}(x)}}  \\
& \leq  \left(\frac{\rr}{\sigma_{k}}\right)^\gamma\sigma_{k}^{\gamma} 
\snr{(w)_{B_{\sigma_{k}}(x)}} +cM(x)  \sigma_{k}^{\eps} \,.\label{alieno}
\end{flalign} 
Using $\rr\equiv \sigma_{k+1}$ in \eqref{alieno}, and summing up, for every $m\in \en_0$ we gain
$$
\sum_{k=1}^{m+1} \sigma_{k}^{\gamma} 
\snr{(w)_{B_{\sigma_{k}}(x)}} \leq  \frac{1}{2^{\gamma}}\sum_{k=0}^{m} \sigma_{k}^{\gamma} 
\snr{(w)_{B_{\sigma_{k}}(x)}} + c M(x)\sum_{k=0}^{\infty}  \sigma_{k}^{\eps}.
$$
Reabsorbing terms and eventually letting $m\to \infty$  gives
$$
\sum_{k=1}^{\infty} \sigma_{k}^{\gamma} 
\snr{(w)_{B_{\sigma_{k}}(x)}} \leq  \frac{\sigma^{\gamma}\snr{(w)_{B_{\sigma}(x)}} }{2^{\gamma}-1} 
 +  \frac{c M(x)}{2^{\gamma}-1}\sum_{k=0}^{\infty}  \sigma_{k}^{\eps}, 
$$
which implies the sequence $\{\sigma_{k}^{\gamma}|(w)_{B_{\sigma_{k}}(x)}|\}$ tends to zero. 
This and \rif{alieno} easily allow us to conclude that 
$
\lim_{\rr\to 0}\,  \rr^{\gamma}|(w)_{B_{\rr}(x)}|= 0
$. Recalling \rif{sisis1}, we have just proved that  
\eqn{bull3}
$$
M(x) < \infty \Longrightarrow \lim_{\rr\to 0}\,  \rr^{\gamma}\snr{(w)_{B_{\rr}(x)}} =0  \Longrightarrow x \in (\Sigma_{\gamma, 1})^{\texttt{c}}. 
$$
On the other hand the definition in \eqref{defir1} implies if $x\in (\mathcal R_{\gamma+t-\eps})^{\texttt{c}}$, then $M(x) < \infty$. Connecting the last two pieces of information implies that 
 $\Sigma_{\gamma, 1} \subset \mathcal R_{\gamma+t-\eps}.$ 
 Invoking \eqref{defir1} with the choice $\beta \equiv \gamma + t-\eps$ we conclude with
$\ddim(\Sigma_{\gamma, 1})\leq n-(\gamma+t)p+p\eps $ that holds for every $\eps$ as above, which obviously implies \eqref{riduci1}.  Next, consider $B_{\rr}(x)\subset \Omega$; again using \eqref{fp} we have 
\begin{flalign}
\rr^{\gamma}\nra{w}_{L^{p}(B_{\rr}(x))} & \leq \rr^{\gamma}\nra{w-(w)_{B_{\rr}(x)} }_{L^{p}(B_{\rr}(x))} +\rr^{\gamma}\snr{(w)_{B_{\rr}(x)} } \notag \\
& \leq  \rr^{\gamma+ t} \snra{w}_{t,p;B_{\rr}(x)}+\rr^{\gamma}\snr{(w)_{B_{\rr}(x)} } \,.
\label{bully0}
\end{flalign}
This implies that $\Sigma_\gamma \subset \mathcal R_{\gamma+ t}\cup\Sigma_{\gamma, 1} $ so that $\ddim(\Sigma_\gamma)\leq n-(\gamma+t)p$ 
 finally follows from \eqref{defir1}, this time applied with $\beta \equiv \gamma+ t$, and \eqref{riduci1}. 
 
{\em The case $(\gamma+t)p>n$}. 
This time the choice of $\eps>0$ is such that $(\gamma + t-\eps)p>n$, so that 
\eqn{bully}
$$
\lim_{\rr \to 0}  \rr^{\gamma + t} \snra{w}_{t,p;B_{\rr}(x)} =0 \quad \mbox{and} \quad M(x)<\infty
$$ 
hold for every point $x\in \Omega$. It follows that \eqref{bull3} holds for every $x\in \Omega$. Using this and \eqref{bully} in \eqref{bully0} implies that  $\Sigma_{\gamma}$ is empty.  
 \end{proof}
 
\section{Homogeneous linear systems with translation-invariant coefficients}\label{homsec1}
In this section we consider operators with translation-invariant coefficients
$-\mathcal{L}_{\mathds{B}}$ as in  \eqref{weaksolLL}, under the assumptions
\eqref{condib}, and derive a priori estimates for solutions to equations of the
form $-\mathcal{L}_{\mathds{B}}v=0$.
\subsection{Scaling: from $B_{\rr}(x_0)$ to $\ttB_1\equiv B_1(0)$}
\label{scalasezione}
We briefly recall a standard scaling procedure in the more general
non-homogeneous case. Let $\BBB\subset \er^{n}$ be a fixed ball, and let
$v \in W^{s,2}_{\loc}(\BBB;\er^{N}) \cap L^1_{2s}$ be a weak solution to
$-\mathcal{L}_{\mathds{B}} v=g$ in $\BBB$, that is,
\eqn{TIeq}
$$
\langle -\mathcal{L}_{\mathds{B}} v,\varphi\rangle
=\int_{\BBB} \langle g,  \varphi\rangle \dx,
\qquad
g \in L^{2}(\BBB;\er^{N})
$$
for every $\varphi \in W^{s,2}(\er^n;\er^{N})$ that is compactly supported in
$\BBB$. Given $B_{\rr}(x_0)\Subset \BBB$ and $v_0 \in \er^{N}$, we rescale \eqref{TIeq} as
\eqn{scala1}
$$
 v_{\rr}(x):= \frac{v(x_{0}+\rr x)-v_0}{V},\quad
 \mathds{B}_{\rr}(x-y):=\mathds{B}(\rr(x-y)),\quad
 g_{\rr}(x):=\frac{\rr^{2s}g(x_{0}+\rr x)}{V},
$$
for $x,y\in \er^n$ and $V>0$. It follows that
$
-\mathcal{L}_{\mathds{B}_{\rr}}v_{\rr}=g_{\rr}$
in $\ttB_1.$ Note that if $\mathds{B}$ satisfies \eqref{condib}, then so does $\mathds{B}_{\rr}$. 
As usual, a priori estimates for $v$ in $B_{\rr}(x_0)$ are obtained by scaling
back the corresponding estimates in $\ttB_1$. We shall use this fact repeatedly
in the sequel.
\subsection{Gradient H\"older continuity in the homogeneous case}\label{therestsec} Here we prove
\begin{theorem}[Local regularity]\label{maintra}
Let $v \in W^{s,2}(\BBB;\er^{N}) \cap L^1_{2s}$ be a weak  solution to  
\eqn{equazione}
$$-\mathcal{L}_{\mathds{B}}v =0  \qquad \mbox{in $\BBB$} $$  where $\BBB \subset \er^n$ is a ball, $1/2< s < 1$ and $\mathds{B}\colon \er^n\to \er^{N\times N}$ satisfies \eqref{condib}. 
	Then  $v \in C^{1,\alpha}_{\loc}(\BBB;\er^{N})$ for every  $\alpha \in (0,2s-1)$ and 
	\eqn{almostC2s}
$$
		\rr\nr{Dv}_{L^\infty(B_{\rr/2})} +
		\rr^{1+\alpha}[Dv]_{0, \alpha; B_{\rr/2}} \leq  c \nra{v-v_{0}}_{L^{2}(B_\rr)}+  c\tail(v-v_{0};B_\rr) 
$$
	holds whenever $B_\rr\subset \BBB$ is a ball which is not necessarily concentric with $\BBB$,  and for every $v_0\in \er^{N}$, where $c$ depends only on $\data$ and $\alpha$.	
\end{theorem}  
\begin{remark}{\em By a standard scaling and covering argument, with the local $\tail$s controlled as in
Lemma \ref{covlemma} below, estimate \eqref{almostC2s} improves to
$$
		\rr\nr{Dv}_{L^\infty(B_{\kappa \rr})} +
		\rr^{1+\alpha}[Dv]_{0, \alpha; B_{\kappa\rr}} \leq  c \nra{v-v_{0}}_{L^{2}(B_\rr)}+  c\tail(v-v_{0};B_\rr) 
$$
which holds whenever $\kappa \in (0,1)$, 
with 
$c \equiv c (\data,\alpha,\kappa)$. 
}
\end{remark}
\subsubsection{Reduction to the case $B_{\rr}\equiv \ttB_1$.}\label{riduzione}
By the scaling procedure described in Section \ref{scalasezione}, we may reduce
the proof of Theorem \ref{maintra} to the case
$B_{\rr} \equiv B_{\rr}(x_0) \equiv \ttB_1$ (take $V=1$).
Moreover, in \eqref{almostC2s} we may assume $v_0 \equiv 0$.
Indeed, for Theorem \ref{maintra} it suffices to show that for every
$\alpha \in (0,2s-1)$ there exists a constant $c \equiv c(\data,\alpha)$ such
that
\eqn{startrek}
$$
		\nr{Dv}_{C^{0, \alpha}(\ttB_{1/2})}   \leq  c \nra{v}_{L^{2}(\ttB_{1})}+  c\tail(v;\ttB_{1})
$$
from which \eqref{almostC2s} follows after rescaling. Consequently, the proof of
Theorem \ref{maintra}  will focus on \eqref{startrek}. 
\subsection{Iteration lemmas for Theorem \ref{maintra}}
The proof of Theorem \ref{maintra} requires a sequence of
inequalities and intermediate steps. In particular, for each $q \geq 2$ we
consider the following assumption concerning \emph{all} solutions to \rif{equazione} in \emph{all} balls and for \emph{all} kernels $\mathds{B}$ satisfying \eqref{condib}:
$$
(\texttt{cacc})_{q}:=
\begin{cases}
\mbox{Every solution $w$ of $-\mathcal{L}_{\mathds{B}}w =0$ in a ball $\BBB\subset \er^n$
belongs to $W^{s,q}_{\loc}(\BBB;\er^{N})$,} \\[2pt]
\mbox{and the Caccioppoli inequality} \\
\rr^{s}[w]_{s,q;B_{\rr/2}}
\le \cac{q} \nr{w}_{L^{q}(B_{\rr})}
+ \cac{q} \rr^{n/q}\tail(w;B_{\rr}) \\[2pt]
\mbox{holds for every ball $B_{\rr} \Subset \BBB$, $\rr \le 1$,}\\[2pt]
\mbox{whenever $\mathds{B}$ satisfies \eqref{condib}, where $\cac{q}\equiv c_{q}(\data,q)\geq 1$.}
\end{cases}
$$
\begin{remark}\emph{
Note that $(\texttt{cacc})_{2}$ is automatically satisfied thanks to
Lemma \ref{caccioppola}, and it will serve as the starting point of the
bootstrap procedures adopted in the proof of Theorem \ref{maintra}. The constant $\cac{q}$ is an important  part of assumption $(\texttt{cacc})_{q}$ and will enter the a priori estimates. The constant $\cac{2}$ is determined via Lemma \ref{caccioppola} as a function of $\data$, while the constants $\cac{q}$ appearing along the aforementioned bootstrap procedure will be kept under control along the iterations.}
\end{remark}

\begin{lemma}[Fractional differentiability improvement I]\label{step1lemma}
Under assumptions \eqref{condib} and $(\textnormal{\texttt{cacc}})_{q}$ for some
$q \geq 2$, let $v \in W^{s,2}(\ttB_1;\er^{N}) \cap L^1_{2s}$ be a solution to
$-\mathcal{L}_{\mathds{B}}v =0$ in $\ttB_1$, and fix parameters
$\sss$, $\kappa$, and $\eps$ such that
$s \leq \sss < 2s$, $\kappa \in (0,1)$, and $0<\eps < s+\sss/2$.
Then
\eqn{combi4}
$$
\nr{v}_{B^{s+\sss/2-\eps}_{q,\infty}(\ttB_{\kappa/2^{10}})}
\leq
c\,\nr{v}_{B^{\sss}_{q,\infty}(\ttB_{\kappa})}
+ c\,\tail(v;\ttB_{\kappa})
$$
holds with
$c = c(\data,q,\cac{q},\eps,\kappa)$, where $\cac{q}\geq 1$ is the constant appearing in
$(\textnormal{\texttt{cacc}})_{q}$.
\end{lemma}
\begin{proof} We can assume without loss of generality that $v \in B^{\sss}_{q,\infty}(\ttB_{\kappa})$, that is, the right-hand side of \eqref{combi4} is finite, otherwise \eqref{combi4} is trivial. In particular, for the initial step of the bootstrap, $q=2$ and $\sss=s$, this finiteness follows from the embedding $W^{s,2}\hookrightarrow B^{s}_{2,\infty}$; moreover, $(\texttt{cacc})_2$ follows from Lemma \ref{caccioppola}. First, just a preliminary observation. In this proof we shall use the finite difference operators defined in \eqref{domania}, and, unless specified differently, they will refer to $\er^n$, i.e., $\mathcal A\equiv \er^n$ in \eqref{enfasi}.  In the following we rely on a sort of nonlinear atomic decomposition technique,
first introduced in \cite{krm05}. To proceed, we fix $\beta \in (0,1)$ and $h \in \er^{n} \setminus \{0\}$ with 
\eqn{sceltah}	 
	$$|h|<\left(\frac{\kappa}{2^{12}}\right)^{1/\beta}\leq \frac{\kappa}{2^{12}}$$ 
and take a lattice of open, disjoint hypercubes $\{Q_{|h|^{\beta}/\sqrt{n}}(y)\}_{y\in (2|h|^{\beta}/\sqrt{n})\mathbb Z^n}$. Recalling Lemma \ref{lemmacubi}, we select those hypercubes centred at points $\tx{I}:=\{z_{j}\}\subset (2|h|^{\beta}/\sqrt{n})\mathbb Z^n$ such that 
\eqn{limitaz}
$$ \snr{z_{j}} \leq  \frac{\kappa}{2^{10}}+ 2|h|^{\beta}\stackleq{sceltah}	  \frac{\kappa}{2^{10}}+ \frac{\kappa}{2^{11}} < \frac{\kappa}{2^{8}}$$ 
so that 
	\eqn{overlap1}
$$
\# \tx{I} \lesssim_{n} \frac{\kappa^n}{|h|^{n\beta}} + 2^n \lesssim \frac 1{|h|^{n\beta}} .
$$	
We determine the corresponding family $\{Q_{j}\equiv Q_{|h|^{\beta}/\sqrt{n}}(z_{j})\}_{z_j\in \tx{I} }$, so that
\begin{flalign}\label{11.1}
\ttB_{\kappa/2^{10}}\subset  \bigcup \overline{Q_{j}},\quad Q_{j_{1}}\cap Q_{j_{2}}=\emptyset \ \Longleftrightarrow \ j_{1}\neq j_{2}\,.
\end{flalign}  
The family $\{Q_{j}\}$ can be realized as the family of inner cubes of the family of  balls
\eqn{palloni}  
$$\{B_{j} \}\equiv \{B_{|h|^{\beta}}(z_j)\},$$ 
i.e., $Q_{j}\equiv Q_{\textnormal{inn}}(B_{j})$. Finally, considering the corresponding family of outer hypercubes $\{Q_{|h|^{\beta}}(z_{j})\}_{z_j\in \tx{I} }$ to count the overlaps, and recalling \rif{11.1}, we conclude with  
	\begin{equation} \label{overlap}
	\ttB_{\kappa/2^{10}} \subset \bigcup_{z_j \in \tx{I}} \overline{B_{j}} , \qquad \sup_{x \in \er^{n}} \sum_{z_j \in \tx{I}} \chi_{2^kB_{j}}(x) \lesssim_n 2^{nk}
	\end{equation}
for every $k \in \mathbb{N}_0$; see also \cite[Lemma 2.11]{DKLN}. More precisely, we have that for every $k\in \en_0$ the family $\{2^kB_{j}\}$ can be split into $\mathfrak{n} \approx \mathfrak{n}(k) \approx 2^{nk}$ subfamilies of mutually disjoint balls $\{B_{i}^{l}\}_i$, $l \leq \mathfrak{n}$, i.e. 
$$ 
 \bigcup_{z_j \in \tx{I}}  2^kB_{j}= \bigcup_{l \leq \mathfrak{n}}  \bigsqcup_{i} B_{i}^{l} , \quad   B_{i}^{l}\equiv B_{i,k}^{l}  .
$$
Note that, for such a family, we have that 
\eqn{sommakappa}
$$
\sum_{z_{j} \in \tx{I}}\mu(2^kB_{j}) \lesssim_{n} 2^{nk} \mu\left( \bigcup_{z_j \in \tx{I}}  2^kB_{j}\right)
$$
and this takes place whenever $\mu$ is a non-decreasing set function defined on Borel sets  such that 
\eqn{verificasomma}
$$
\sum_{j}\mu(X_j) \leq \mu\left(\bigcup_j X_j\right)
$$ 
holds whenever $\{X_j\}$ is a disjoint, finite family of balls contained in $\er^n$.
The upper bound on $|h|$ given by \eqref{sceltah} gives that if $ \varphi \in W^{s,2}(\er^n;\er^{N})$ is compactly supported in $\ttB_{1/2}$, then $ \tau_{-h}^2\varphi$ is compactly supported in $\ttB_1$ and therefore $\langle - \mathcal{L}_{\mathds{B}}  v,\tau_{-h}^2\varphi \rangle =0 $. By using \eqref{matobo2} and \eqref{matobo1} we deduce that 
\eqn{solutiodd}
$$\langle - \mathcal{L}_{\mathds{B}} \tau_h^2 v,\varphi \rangle =0 $$ and therefore we conclude that $w:=\tau_h^2 v$ solves $- \mathcal{L}_{\mathds{B}}w=0$ in $\ttB_{1/2}$. Notice that using \rif{sceltah}	  it is easy to see that  $\tau_h^2 v \in W^{s,2}_{\loc}(\ttB_{1/2};\R^N)\cap L^1_{2s}$ and therefore $w$ is a genuine solution in the sense of Definition \ref{def:weaksol2}. Next, define, for every $z_j \in \tx{I}$ the positive integer
\eqn{defimj}
	$$m_{j}\equiv m(|h|^{\beta}, z_{j}):= \max \{k \in \mathbb{N} \, \colon\,  2^kB_{j}  \subset \ttB_{\kappa/2^{7}} \}$$ 
so that 
\eqn{defimj2}
	 $$
	2^{m_{j}-11}\kappa \stackgeq{sceltah}   2^{m_{j}+1} |h|^{\beta} \stackgeq{defimj} \frac{\kappa}{2^{7}} - \snr{z_j} \stackrel{\eqref{limitaz}}{>}  \frac{\kappa}{2^{7}} -   \frac{\kappa}{2^{8}} = \frac{\kappa}{2^{8}}
	 $$
from which we deduce 
	\eqn{limitazionem}
	$$m_{j}  > 3\, \quad  \mbox{and therefore}  \quad  4B_{j} \Subset  \ttB_{\kappa/2^{7}}\Subset \ttB_{1/2}.$$
Moreover, again from the definition of $m_j$ and \eqref{defimj2}, we have 
\eqn{sl3}	 
	$$\frac {\kappa}{2^{8}}\leq  2^{m_{j}+1} |h|^{\beta} \leq  \frac {\kappa}{2^{6}} $$	
	and therefore 
	\eqn{sl33}	 
	$$ m_{j} \lesssim \log  \frac{\kappa}{|h|^{\beta}}  \leq \log \, \frac1{|h|}\lesssim_{\eps} \frac{1}{|h|^{\frac{q\eps}{q-1}}}, \quad \mbox{holds for every $\eps >0$}\,.
	$$
	Here we recall that $q$ is the exponent considered in $(\textnormal{\texttt{cacc}})_{q}$. From now on we shall always take $0< \eps<s+\sss/2$. 
We apply Lemma \ref{prop:embedding0} to $ \tau_h^2 v$ thereby getting (recall that $|h| \leq |h|^{\beta}$)
$$\nr{\tau^3_h v}_{L^q(B_{j})}^q  \leq  c |h|^{qs} [\tau_h^2 v]_{s,q;2B_{j}}^q +c |h|^{qs(1-\beta)} \nr{\tau_h^2 v}_{L^q(2B_{j})}^q .
$$ 
Thanks to \eqref{solutiodd} we now use assumption $(\texttt{cacc})_{q}$ applied to $w\equiv \tau_h^2 v$ (recall the inclusion in \eqref{limitazionem}), therefore we have 
\eqn{combi1}
$$
		\nr{\tau^3_h v}_{L^q(B_{j})}^q 
		\leq  c |h|^{qs(1-\beta)} \nr{\tau_h^2 v}_{L^q(4B_{j})}^q+ c |h|^{qs(1-\beta)+n \beta} \tail(\tau_h^2 v;4B_{j})^q  $$
	where $c\equiv c(\data, q,\cac{q})$, and $\cac{q}\geq 1$ is the constant appearing in $(\texttt{cacc})_{q}$. Next, we deal with the last term in the above display. Splitting in annuli and using H\"older's inequality we find
	\begin{flalign*}
		\tail(\tau_h^2 v;4B_{j})& = (4|h|^{\beta})^{2s} \sum_{k=2}^{m_j+3} \int_{2^{k+1}B_{j}\setminus 2^{k}B_{j}}
\frac{\snr{\tau_h^2 v}}{\snr{x-z_j}^{n+2s}}\dx 		
		 \\ 
		 & \qquad +
(4|h|^{\beta})^{2s}\int_{\er^n\setminus 2^{m_{j}+4}B_{j}}
\frac{\snr{\tau_h^2 v}}{\snr{x-z_j}^{n+2s}}\dx 	\\	 
		 & \leq  c  \sum_{k=3}^{m_{j}+4} \frac{|h|^{-n\beta}}{2^{(n+2s)k}} \nr{\tau_h^2 v}_{L^1(2^{k}B_{j})} +\frac{c}{2^{2s(m_{j}+4)}} \tail(\tau_h^2 v;2^{m_{j}+4}B_{j})\nonumber \\ \ & \leq  c  \sum_{k=3}^{m_{j}+4}  \frac{|h|^{-n\beta/q}}{2^{(n/q+2s)k}} \nr{\tau_h^2 v}_{L^q(2^{k}B_{j})} + \frac{c}{4^{sm_{j}}} \tail(\tau_h^2 v;2^{m_{j}+4}B_{j})\,.
		\end{flalign*}	 
By further using the discrete H\"older's inequality, i.e., 
$$ 
\sum_{k=3}^{m_{j}+4}  \frac{|h|^{-n\beta/q}}{2^{(n/q+2s)k}} \nr{\tau_h^2 v}_{L^q(2^{k}B_{j})}\lesssim m_{j}^{1-1/q}
 \left(  \sum_{k=3}^{m_{j}+4}  \frac{|h|^{-n\beta}}{2^{(n+2qs)k}}   \nr{\tau_h^2 v}_{L^q(2^{k}B_{j})}^q\right)^{1/q}
$$
we arrive at 	
	\begin{flalign} 
		\notag \tail(\tau_h^2 v;4B_{j})^q &  \leq  c  m_{j}^{q-1}\sum_{k=3}^{m_{j}+4} \frac{ |h|^{-n\beta}}{2^{(n+2qs)k}} \nr{\tau_h^2 v}_{L^q(2^{k}B_{j})}^q \\
		& \quad +\frac{c}{4^{qsm_{j}}} \tail(\tau_h^2 v;2^{m_{j}+4}B_{j})^q \label{combi2}
	\end{flalign}
	where $c\equiv c (n,N,s,q)$. 
The definitions of $\tail$ and $\tau_h^2 v$ yield
	\begin{flalign}
\notag \tail(\tau_h^2 v;2^{m_{j}+4}B_{j}) & \leq \tail(v(\cdot+2h);2^{m_{j}+4}B_{j}) \\
& \quad + 2 
\tail(v(\cdot+h);2^{m_{j}+4}B_{j}) + \tail(v;2^{m_{j}+4}B_{j})\,.\label{combi10}
	\end{flalign}	
The estimates for the three terms appearing in the right-hand side of the above inequality are similar and are based on the application of Lemma \ref{sl1}. Let's see how to estimate, for instance, the first one  (recall now that from \eqref{palloni} $z_j$ are the centers of $B_j$). In Lemma \ref{sl1} with $z_0\equiv z_{j}$, $x_0\equiv 0$, $\rr\equiv 2^{m_{j}+4}|h|^{\beta}$, $\tau = \kappa/2^{6}$. 
For this observe that 
$$
\snr{z_{j}} \stackleq{limitaz} \frac{\kappa}{2^{8}} \stackleq{sl3} 2^{m_{j}+2}|h|^{\beta} \equiv \frac{\rr}{4}, \qquad 
2 |h|\leq 2 |h|^{\beta} \stackleq{limitazionem} 2^{m_{j}+1}|h|^{\beta} \equiv \frac{\rr}{8} 
$$
and
$$
\tau \equiv \frac{\kappa}{2^6} \stackleq{sl3}  2^{m_{j}+3 } |h|^{\beta}  \equiv \frac{\rr}{2}.
$$
Lemma \ref{sl1} now yields 
$$
 \tail(v(\cdot+2h);2^{m_{j}+4}B_{j})  \lesssim_{n,s}  \tail(v;\ttB_{\kappa/2^{6}}).
$$
All in all, estimating similarly the remaining two terms in the right-hand side of \eqref{combi10}, we find 
$$
\tail(\tau_h^2 v;2^{m_{j}+4}B_{j})  \lesssim_{n,s}		\tail(v;\ttB_{\kappa/2^{6}}).
$$ 
This and  \eqref{combi2} provide
$$ 
		\tail(\tau_h^2 v;4B_{j})^q    \leq  c m_{j}^{q-1}\sum_{k=3}^{m_{j}+4} \frac{ |h|^{-n\beta}}{2^{(n+2qs)k}} \nr{\tau_h^2 v}_{L^q(2^{k}B_{j})}^q  +\frac{c}{ 4^{qsm_{j}}} \tail(v;\ttB_{\kappa/2^6})^q
$$
that, together with \eqref{combi1} gives 
\begin{flalign}
\notag		\nr{\tau^3_h v}_{L^q(B_{j})}^q 
		&\leq  c |h|^{qs(1-\beta)} m_{j}^{q-1}\sum_{k=2}^{m_{j}+4} \frac{1}{2^{(n+2qs)k}} \nr{\tau_h^2 v}_{L^q(2^{k}B_{j})}^q\\
		& \quad + \frac{c |h|^{qs(1-\beta)+n \beta}} {4^{qsm_j}}\tail(v;\ttB_{\kappa/2^6})^q  \,.\label{provide}
\end{flalign}	
Now our aim is to recover a global estimate patching up the ones in the above display, thereby summing over $z_j \in \tx{I}$, i.e., 
 \begin{flalign}
\nr{\tau_h^3 v}_{L^q(\ttB_{\kappa/2^{10}})}^q
&\stackleq{overlap}
\sum_{z_j\in \tx I}\nr{\tau_h^3 v}_{L^q(B_j)}^q \notag\\
&\stackrel{\eqref{provide}}{\le}
c\,|h|^{qs(1-\beta)}
\sum_{k=2}^{\infty}\frac{1}{2^{(n+2qs)k}}
\sum_{z_j\in \tx I_k} m_j^{q-1}\nr{\tau_h^2 v}_{L^q(2^kB_j)}^q \notag\\
&\qquad
+c\,|h|^{qs(1-\beta)+n\beta}\tail(v;\ttB_{\kappa/2^6})^q
\sum_{z_j\in \tx I}\frac{1}{4^{qsm_j}}, \label{federer}
\end{flalign}
where we are denoting $\tx I_k:=\{z_j\in \tx I\, \colon \, k\leq m_j+4\}$. 
To proceed,  a few remarks. First, we observe that if $z_j\in \tx I_k$, then by definition of $\tx I_k$ we have
\eqn{includi}
$$
2^kB_j\subset 2^{m_{j}+4}B_{j} \subset \ttB_{\kappa/4}.
$$
Indeed, using \eqref{limitaz} and \eqref{sl3} we have 
$
|z_j| + 2^{m_{j}+4}|h|^{\beta} \leq \kappa/2^{8} + \kappa/2^{3} < \kappa/4
$, that substantiates \eqref{includi}. 
Then, applying \eqref{sommakappa} to the set function  $\mathcal A\mapsto \mu(\mathcal A)\equiv  \nr{\tau_h^2 v}_{L^q(\mathcal A)}^q$,  we have that
$$
 \sum_{z_j \in \tx{I}_k} \nr{\tau_h^2 v}_{L^q(2^{k}B_{j})}^q
  \lesssim_{n} 2^{nk}  \nr{\tau_h^2 v}_{L^q(\ttB_{\kappa/4})}^q
$$
Second, note that 
$$
\sum_{z_j \in \tx{I}}\frac{1}{4^{qsm_{j}}}  \stackrel{\eqref{sl3}}{\lesssim_{n,s,q}} \frac{1}{\kappa^{2qs}} \sum_{z_j \in \tx{I}}|h|^{2qs\beta}  \lesssim_{n,\kappa} \# \tx{I}|h|^{2qs\beta}\stackrel{\eqref{overlap1}}{\lesssim} |h|^{2qs\beta-n\beta} .
$$
Using the content of the last two displays and of \eqref{sl33} in \eqref{federer} we obtain
\eqn{pedantic1}
		 $$ 
 \nr{\tau^3_h v}_{L^q(\ttB_{\kappa/2^{10}})}^q  \leq  c  |h|^{qs(1-\beta)-q\eps}\nr{\tau_h^2 v}_{L^q(\ttB_{\kappa/4})}^q+ c |h|^{qs(1+\beta)}\tail(v;\ttB_{\kappa/2^6})^q
$$
where $c = c(\data,q,\cac{q},\eps,\kappa)$. Observe that \rif{sceltah} implies that 
$\tau_h^2 v\equiv \tau_{h,\er^n}^2 v= \tau_{h,\ttB_{\kappa/2}}^2 v $ on $\ttB_{\kappa/4}$. 
Therefore, recalling that $\sss < 2s$ and  \eqref{besovo0},  and observing that we can further estimate 
$$
\nr{\tau_h^2 v}_{L^q(\ttB_{\kappa/4})}
=
\nr{\tau_{h,\ttB_{\kappa/2}}^2 v}_{L^q(\ttB_{\kappa/4})}
\le
\nr{\tau_{h,\ttB_{\kappa}}^2 v}_{L^q(\ttB_{\kappa})}
\le
|h|^{\sss}\nr{v}_{B^{\sss}_{q,\infty}(\ttB_{\kappa})}.
$$
Inserting this last estimate in \rif{pedantic1} yields
$$
 \nr{\tau^3_h v}_{L^q(\ttB_{\kappa/2^{10}})}^q  \leq  c  |h|^{qs(1-\beta)+q\sss-q\eps} \nr{v}_{B^{\sss}_{q,\infty}(\ttB_{\kappa})}^q+ c |h|^{qs(1+\beta)}\tail(v;\ttB_{\kappa/2^6})^q.
$$
This estimate holds whenever $h \in \er^n$ satisfies \eqref{sceltah}. We now choose $\beta$ as
$$
\beta := \frac{\sss}{2s} \in (0,1)\quad \Longleftrightarrow  \quad s(1-\beta)+\sss= s(1+\beta)
$$
thereby obtaining
$$
 \nr{\tau^3_{h, \ttB_{\kappa/2^{10}}} v}_{L^q(\ttB_{\kappa/2^{10}})}\leq  \nr{\tau^3_h v}_{L^q(\ttB_{\kappa/2^{10}})}  \leq  c  |h|^{s+\sss/2-\eps} \left(\nr{v}_{B^{\sss}_{q,\infty}(\ttB_{\kappa})}+ \tail(v;\ttB_{\kappa/2^6})\right).
$$
Recalling \eqref{besovo} and \eqref{sceltah}, we conclude with 
$$
		\nr{v}_{B^{s+\sss/2-\eps}_{q,\infty}(\ttB_{\kappa/2^{10}})} \leq c  \nr{v}_{B^{\sss}_{q,\infty}(\ttB_{\kappa})} + c\,\tail(v;\ttB_{\kappa/2^6}) 
$$
with $c = c(\data,q,\cac{q},\eps,\kappa)$, 
from which 
\eqref{combi4} follows from a further application of \eqref{scatailancora2} and H\"older's inequality. Note that here we used \eqref{besovo0} with $l=2$ for the norm
$\nr{v}_{B^\sss_{q,\infty}(\ttB_\kappa)}$, and with $l=3$ for the norm
$\nr{v}_{B^{s+\sss/2-\eps}_{q,\infty}(\ttB_{\kappa/2^{10}})}$. \end{proof}

\begin{lemma}[Fractional differentiability improvement II]\label{step2lemma} Under the assumptions of Lemma \ref{step1lemma}, and in particular under assumption $\textnormal{(\texttt{cacc})}_{q}$, for every $t$ such that $s <t< 2s$, there exists an  integer 
\eqn{interot}
$$ m(t)\equiv m(t,s)\geq 10$$ 
depending only on $t$ and $s$, such that 
\eqn{applicancora}
$$ 
	\nr{v}_{W^{t,q}(\ttB_{\kappa/2^{m(t)}})} \leq c  \nr{v}_{L^{q}(\ttB_{\kappa})} + c\,\tail(v;\ttB_{\kappa}) 
$$ 
holds whenever $\kappa \in (0,1)$, where 
$c\equiv c(\data,t,q,\cac{q}, \kappa)$. 
\end{lemma}
\begin{proof} We inductively define, for every integer $k\geq 0$
$$  
 \sss_{k+1} := s+\frac{\sss_{k}}{2}-\frac{1}{k+2}\left(s-\frac{\sss_{k}}{2}\right)=:s+\frac{\sss_{k}}{2}-\eps_k ,\qquad \sss_0:=s.
$$
By induction $\{\sss_k\}$ is an increasing sequence such that $s\leq \sss_k<2s$ for every $k\in \en_0$ and $\sss_k \nearrow 2s$.  Fix $\kappa\in (0,1)$. Applying \eqref{combi4} (with $\kappa$ replaced by $\kappa/2^{10k}$ and $\eps\equiv \eps_k$), we have
$$ 
		\nr{v}_{B^{\sss_{k+1}}_{q,\infty}(\ttB_{\kappa/2^{10(k+1)}})} \leq c  \nr{v}_{B^{\sss_k}_{q,\infty}(\ttB_{\kappa/2^{10k}})} + c\,\tail(v;\ttB_{\kappa/2^{10k}}) 
$$
for every integer $k\geq 0$.  
By using \eqref{scatailancora2} we conclude with 
$$
		\nr{v}_{B^{\sss_{k+1}}_{q,\infty}(\ttB_{\kappa/2^{10(k+1)}})} \leq c  \nr{v}_{B^{\sss_k}_{q,\infty}(\ttB_{\kappa/2^{10k}})} +c \nr{v}_{L^1(\ttB_{\kappa})} + c\,\tail(v;\ttB_{\kappa}), 
$$ 
where $c\equiv c(\data,k,q,\cac{q}, \kappa)$.
Iterating yields 
\eqn{iterata}
$$
		\nr{v}_{B^{\sss_{k}}_{q,\infty}(\ttB_{\kappa/2^{10k}})}  \leq c  \nr{v}_{B^{s}_{q,\infty}(\ttB_{\kappa})} + c\,\tail(v;\ttB_{\kappa})
$$ for every integer $k\geq 1$ and $\kappa \in (0,1)$, where $c\equiv c(\data,k,q,\cac{q}, \kappa)$. 
Using this together with \eqref{immergi}-\eqref{immergi2} and \eqref{scatailancora2}, it is easy to obtain that for every $t \in (s,2s)$ (without loss of generality we may assume $t$ is non-integer) there exists an integer $m(t)\equiv m(t,s)$ depending only on $t$ and $s$, such that 
\eqn{arriva2}
$$
		\nr{v}_{W^{t,q}(\ttB_{\kappa/2^{m(t)}})} \leq c  \nr{v}_{W^{s, q}(\ttB_{\kappa/2})}  + c\,\tail(v;\ttB_{\kappa/2})
$$
holds whenever $\kappa \in (0,1)$, with $c\equiv c (\data, t,q,\cac{q},\kappa)$ (note that here we have used \rif{iterata} with the change of notation $\kappa \leftrightarrow \kappa/2$ as $\kappa$ in \rif{iterata} is arbitrary).    
Using the Caccioppoli inequality in the assumption $
(\texttt{cacc})_{q}$ to bound the first term in the right-hand side of \eqref{arriva2}, and yet \rif{scatailancora2} to bound the second, we complete the proof of \eqref{applicancora}. Note that the constant $c$ in \rif{applicancora} is uniformly bounded when $\kappa \uparrow 1$.
\end{proof}
\begin{lemma}[Integrability improvement 1]\label{step3lemma} Under the assumptions of Lemma \ref{step1lemma}, and in particular assuming $\textnormal{(\texttt{cacc})}_{q}$, let $\qq$ be a finite number such that 
\eqn{migliora}
$$
q\leq \qq<\begin{cases}\displaystyle \, 
 \frac{nq}{n-sq} \  \ \mbox{if $qs < n$},\\[7pt]
\displaystyle \,  \infty  \ \ \mbox{if $qs \geq  n$}\,.
 \end{cases}
$$
\begin{itemize}
\item
Then $v \in W^{s,\qq}_{\loc}(\ttB_1;\er^{N})$ and 
\eqn{asscaccdopo}
$$
\rr^{s-n/\qq}[v]_{s,\qq;B_{\rr/2}} +\rr^{-n/\qq}\nr{v}_{L^{\qq}(B_{\rr/2})}
\le c\rr^{-n/q} \nr{v}_{L^{q}(B_{\rr})}+c \tail(v;B_{\rr})
$$
holds  whenever $B_{\rr} \Subset \ttB_1$ is a ball not necessarily centred at the origin. In both cases the constant $c$ depends on $\data,q,  \cac{q},\qq$. 
\item In particular, 
\eqn{migliora2}
$$\textnormal{(\texttt{cacc})}_{q} \Longrightarrow \textnormal{(\texttt{cacc})}_{\qq} \quad \mbox{and}\quad \cac{\qq}\equiv  \cac{\qq}(\data,q, \cac{q}, \qq)\,.$$
\end{itemize}
\end{lemma}
\begin{proof} We first give the proof of \eqref{asscaccdopo}, which implies  $v \in W^{s,\qq}_{\loc}(\ttB_1;\er^{N})$. Fix $\qq$ as in \eqref{migliora}. In the case $qs<n$ we take $t \equiv t(\qq) \in (s,2s)$ close enough to $2s$ in order to have $
\qq <nq/[n-(t-s)q]$, which is equivalent to \rif{immercond}. In the case $qs\geq n$ we take $t \equiv t(\qq)>s$ such that $2s-n/\qq<t<2s$ so that condition \eqref{immercond} is again fulfilled. We use such a choice of $t\equiv t(\qq)$ in Lemma \ref{step2lemma} and determine the integer $\mf{m} (\qq)$ via \eqref{interot} as follows 
$
\mf{m}\equiv \mf{m} (\qq) = m(t(\qq)).
$
Therefore \eqref{applicancora} translates into $
		\nr{v}_{W^{t(\qq),q}(\ttB_{\kappa/2^{\mf{m} }})} \leq c  \nr{v}_{L^{q}(\ttB_{\kappa})} + c\,\tail(v;\ttB_{\kappa})
$, 
where
$
 t(\qq) > s -n/\qq+ n/q,
$ 
that, combined with \eqref{immergi01}, gives 
\eqn{applicancora2} 
$$
\nr{v}_{W^{s
		, \qq}(\ttB_{\kappa/2^{\mf{m}  }})}\leq c  \nr{v}_{L^{q}(\ttB_{\kappa})} + c\,\tail(v;\ttB_{\kappa})  \,,
$$ 
where $c\equiv c (\data,q,\qq, \cac{q}, \kappa)$ remains bounded as $\kappa \to 1$. This is the basic inequality from which, via a combination of standard covering and scaling arguments, we recover \eqref{asscaccdopo}. Indeed, take an arbitrary ball $B_{\rr}\equiv B_{\rr}(x_0)\Subset \ttB_1$; we scale $v$ from $B_\rr$ to $\ttB_1$ as in Section \ref{scalasezione}, apply \eqref{applicancora2} with any $\kappa \in (1/2,1)$  to $v_{\rr}$, let $\kappa\uparrow1$ and finally scale back. This yields 
\eqn{scalata}
$$
\rr^{s-n/\qq}[v]_{s,\qq;B_{\rr/2^{\mf{m}}}}+  \rr^{-n/\qq} \nr{v}_{L^{\qq}(B_{\rr/2^{\mf{m}}})}\leq 	
 c \rr^{-n/q} \nr{v}_{L^{q}(B_{\rr})} + c\tail(v;B_{\rr}) \,,
$$
where $c\equiv c (\data,q,\qq, \cac{q})$. Using a covering argument, i.e., Lemma \ref{covlemma} below, we obtain 
$$
\nr{v}_{W^{s, \qq}(\ttB_{\kappa/2})}
\le c \nr{v}_{L^q(\ttB_\kappa)}+c  \tail(v;\ttB_\kappa),  \quad \kappa \in (1/2,1)
$$
with $c\equiv c (\data,q,\qq, \cac{q})$, 
so that \eqref{asscaccdopo} follows again scaling as just done in order to get \eqref{scalata}. Moreover, applying H\"older's inequality in the right-hand side of \eqref{asscaccdopo} and the same scaling argument described above allows us to conclude with \eqref{migliora2} (in this respect note that all the arguments employed are independent of the specific kernel $\mathds{B}$; the only important thing is \eqref{condib}). \end{proof} 
\begin{lemma}[Integrability improvement 2]\label{step4lemma} Under the assumptions of Theorem \ref{maintra} (with $B_{\rr} \equiv \ttB_1$ as in the reduction described in Section \ref{riduzione}), there exists an increasing sequence $\{q_k\}_{k\geq 0}$ with $q_0= 2$ and $q_k \nearrow  \infty$, such that \begin{itemize}
\item $v \in W^{s, q_k}_{\loc}(\ttB_1;\er^{N})$ and
\eqn{asscaccdopo2ss}
$$
		\nr{v}_{W^{s, q_k}(\ttB_{1/2})} \leq c  \nr{v}_{L^{2}(\ttB_1)} + c\,\tail(v;\ttB_1)
$$  
holds for every integer $k\geq 0$, where $c \equiv c (\data, k)\footnote{To be precise, the constant $c$ depends on $k$ in the sense that it depends on $\{q_1, \ldots, q_k\}$. In turn, via \eqref{exit0} these are inductively defined using parameters $n,s$.}$. 
\item Condition $\textnormal{(\texttt{cacc})}_{q_k}$ holds for every $k\geq 0$. 
\end{itemize}
\end{lemma}
\begin{proof} The idea is to iterate Lemma \ref{step3lemma}. This leads to define the sequence $\{q_{k}\}_{k\in \en_0}$, which is actually built in two steps, as follows. First we define a finite sequence $\{\ti{q}_{k}\}_{k \leq \bar k+1}$, inductively, as 
\eqn{exit0}
$$
\begin{cases}
\displaystyle  \ti{q}_{k+1} := \frac{n\ti{q}_k}{n-s\ti{q}_k}-b_k  \quad  \mbox{provided $k\geq0$ and $s\ti{q}_k < n$}, \\[7pt] \ti{q}_0=2,\\[7pt]
\displaystyle b_k:=\frac 12\min\left\{ \frac{n\ti{q}_k}{n-s\ti{q}_k}-\ti{q}_k, \frac{1}{k+1}\right\} \,.
\end{cases}
$$ 
This sequence is increasing and stops at  $\bar k+1$, which is such that 
\eqn{exit}
$$s \ti{q}_{\bar k+1} \geq n, \quad \bar k := \max \{k \in \en_0 \, \colon \, s\ti{q}_k < n\}\,.$$
Note that this integer always exists, otherwise $\ti{q}_k< n/s$ would follow for every $k$, implying that $\{\ti{q}_k\}$ has a finite limit, which is impossible by a standard fixed point argument based on \eqref{exit0}. We finally define the sequence $\{q_k\}$ as 
$$
q_k := \begin{cases}
\ti{q}_k & \mbox{if $k\leq \bar k +1$}\\
k\ti{q}_{\bar k+1}  & \mbox{if $k > \bar k +1$}\,.
\end{cases}
$$ 
Recalling the content of Lemma \ref{step3lemma}, we have
\eqn{iterate} 
$$ 
\begin{cases}
\mbox{$(\texttt{cacc})_{q_0}=(\texttt{cacc})_{2}$ holds by \eqref{caccad}}\\
\mbox{$(\texttt{cacc})_{q}$} \Longrightarrow \mbox{$(\texttt{cacc})_{\qq}$ holds for any $\qq$ as in \eqref{migliora}, with $ \cac{\qq}\equiv  \cac{\qq}(\data,q, \cac{q}, \qq)$}\,.
\end{cases}
$$
Inductively using  Lemma \ref{step3lemma}, and in particular \eqref{asscaccdopo} and \eqref{iterate}, gives that 
\eqn{asscaccdopokk}
$$
\begin{cases}
 \rr^{s-n/q_{k+1}}[v]_{s,q_{k+1};B_{\rr/2}}+\rr^{-n/q_{k+1}}\nr{v}_{L^{q_{k+1}}(B_{\rr/2})} \le c\rr^{-n/q_k} \nr{v}_{L^{q_k}(B_{\rr})}+c \tail(v;B_{\rr})\\[5 pt]
 \mbox{$(\texttt{cacc})_{q_{k+1}}$ holds}
 \end{cases}
$$
for every $k$ such that $0\leq k\leq \bar k$, 
with \eqref{asscaccdopokk}$_1$ that holds for every $B_{\rr} \Subset \ttB_1$ with $c\equiv c(\data,k)$. By \eqref{exit}, we finally apply Lemma \ref{step3lemma} once more, this time with the choice $q\equiv q_{\bar k+1}$ and with any $\qq\geq q_{\bar k+1}$, that is
$$
\rr^{s-n/\qq}[v]_{s,\qq;B_{\rr/2}} +\rr^{-n/\qq}\nr{v}_{L^{\qq}(B_{\rr/2})}
\le c\rr^{-n/q_{\bar k+1}} \nr{v}_{L^{q_{\bar k+1}}(B_{\rr})}+c \tail(v;B_{\rr})
$$
holds whenever $\qq \geq q_{\bar k+1}$ and $B_{\rr} \Subset \ttB_1$. Using H\"older's inequality on the right-hand side it also follows that $(\texttt{cacc})_{\qq}$ holds whenever $\qq\geq q_{\bar k+1}$ thanks to \eqref{exit}. It follows that \eqref{asscaccdopokk} continues to hold for all indices $k \geq  \bar k+1$ and, ultimately, for every $k\geq 0$. Note that no inductive argument is needed here after index $\bar k+1$ as a consequence of \eqref{migliora}. 
By \eqref{asscaccdopokk} we find 
\eqn{asscaccdopokk22}
$$
\nr{v}_{W^{s, q_{k+1}}(\ttB_{\kappa/2^{k+1}})}
\le c \nr{v}_{L^{q_k}(\ttB_{\kappa/2^{k}})}+c \tail(v;\ttB_{\kappa/2^{k}})
$$
for every $k\geq 0$ and $\kappa \in (0,1)$.  
Applying \eqref{scatailancora2} to estimate 
$$
\tail(v;\ttB_{\kappa/2^{k}}) \lesssim_{n,s,k, \kappa} \nr{v}_{L^2(\ttB_{\kappa})}+ \tail(v;\ttB_{\kappa}) 
,$$ using this last inequality in \eqref{asscaccdopokk22} and iterating yields
$$
		\nr{v}_{W^{s, q_k}(\ttB_{\kappa/2^{k}})} \leq c  \nr{v}_{L^{2}(\ttB_{\kappa})} + c\,\tail(v;\ttB_{\kappa})
$$
for every integer $k \geq 1$, $\kappa \in (0,1]$, where  $c\equiv c (\data,k, \kappa)$. This last estimate is similar to \eqref{applicancora2}; using the same scaling and covering argument explained after \eqref{applicancora2} we arrive at \eqref{asscaccdopo2ss} for every $k\geq 1$. The same argument also proves that $v \in W^{s, q_k}_{\loc}(\ttB_1;\er^{N})$. Note that the case $k=0$ of  \eqref{asscaccdopo2ss} trivially holds by $(\texttt{cacc})_{2}$, recalling that $q_0=2$. \end{proof}
\subsection{Proof of Theorem \ref{maintra}}
We need to prove \eqref{startrek}. 
In turn, this follows by the standard covering and scaling arguments explained after \eqref{applicancora2}, once we prove that
\eqn{finale}
$$
		\nr{Dv}_{C^{0, \alpha}(\ttB_{1/2^{m}})}   \leq  c \nra{v}_{L^{2}(\ttB_{1})}+  c\tail(v;\ttB_{1})
$$
holds with $c\equiv c (\data, \alpha)$, where $m \equiv m (n,s,\alpha)$ is a positive integer. The proof of \eqref{finale} goes along the scheme of Table \ref{tabella} and we begin by applying Lemma \ref{step4lemma} obtaining that \eqref{asscaccdopo2ss}
and $\textnormal{(\texttt{cacc})}_{q_k}$ hold for every $k\geq 0$. This last fact allows us to apply Lemma  \ref{step2lemma} with $q\equiv q_k$ for every $k\geq 0$, and use \eqref{applicancora}, that is 
$
		\nr{v}_{W^{t,q_k}(\ttB_{1/2^{m}})} \leq c  \nr{v}_{L^{q_k}(\ttB_{1/2})} + c\,\tail(v;\ttB_{1/2}) 
,$ which 
holds for every $t \in (s,2s)$, with $m\equiv m(t,s)$ being a positive integer and  $c\equiv c (\data, t, k)$. This can be combined with \eqref{asscaccdopo2ss} which gives 
\eqn{buti1}
$$
		\nr{v}_{W^{t,q_k}(\ttB_{1/2^{m}})} \leq c  \nr{v}_{L^{2}(\ttB_{1})} + c\,\tail(v;\ttB_{1}) 
$$
with $c \equiv c (\data, t,k)$; note that we have used \rif{scatailancora2} and  H\"older's inequality to control
$\tail(v;\ttB_{1/2})$ by
$\tail(v;\ttB_1)+\nra{v}_{L^2(\ttB_1)}$.  As we are assuming $s>1/2$, we can choose $t\equiv t (n,s,\alpha)\in (s,2s)$ sufficiently close to $2s$ and $q_k$ large enough to have 
$
\alpha < t-1-n/q_k,
$ so that \eqref{immergesmooth} implies 
$$
	\nr{Dv}_{C^{0, \alpha}(\ttB_{1/2^{m}})} \leq  c	\nr{Dv}_{C^{0, t-1-n/q_k}(\ttB_{1/2^{m}})}
\leq   c\nr{v}_{W^{t,q_k}(\ttB_{1/2^{m}})},
$$	
where  $c$ depends on $\data$ and $ \alpha$, and $m$ ultimately depends on $n,s,\alpha$. Using this last inequality in conjunction to \eqref{buti1}  leads to \eqref{finale}. With the above choice of $t$ and $q_k$, the constant in \eqref{buti1} depends only on $\data$ and $\alpha$.
\begin{table}[h!]
\centering
\caption{Interaction of Lemmas  \ref{step1lemma}- \ref{step4lemma} towards the proof of Theorem \ref{maintra}.}
\label{tabella}
\begin{tabular}{|c|c|c|}
\hline
\makebox[4.5cm][c]{Lemma and assumption  \rule{0pt}{.6cm}} & 
\makebox[4.5cm][c]{Main outcome  \rule{0pt}{.6cm}} & 
\makebox[4.5cm][c]{Strategy  \rule{0pt}{.6cm}} \\
\hline
\makebox[4.5cm][c]{Lemma \ref{step1lemma} -- $\textnormal{(\texttt{cacc})}_{q}$ \rule{0pt}{.6cm}} & 
\makebox[4.5cm][c]{$v \in B^{\sss}_{q,\infty} \Longrightarrow v \in B^{s+\sss/2-\eps}_{q,\infty}$  \rule{0pt}{.6cm}} & 
\makebox[4.5cm][c]{Dyadic decomposition  \rule{0pt}{.6cm}} \\
\hline
\makebox[4.5cm][c]{Lemma \ref{step2lemma} -- $\textnormal{(\texttt{cacc})}_{q}$  \rule{0pt}{.6cm}} & 
\makebox[4.5cm][c]{$v \in W^{t,q}_{\loc}$ for every $t<2s$  \rule{0pt}{.6cm}} & 
\makebox[4.5cm][c]{Iterates Lemma \ref{step1lemma}  \rule{0pt}{.6cm}} \\
\hline
\makebox[4.5cm][c]{Lemma \ref{step3lemma} -- $\textnormal{(\texttt{cacc})}_{q}$  \rule{0pt}{.6cm}} & 
\makebox[4.5cm][c]{$\textnormal{(\texttt{cacc})}_{q} \Longrightarrow \textnormal{(\texttt{cacc})}_{\qq}$  \rule{0pt}{.6cm}} & 
\makebox[4.5cm][c]{Lemma \ref{step2lemma}, embedding \rule{0pt}{.6cm}} \\
\hline
\makebox[4.5cm][c]{Lemma \ref{step4lemma}   \rule{0pt}{.6cm}} & 
\makebox[4.5cm][c]{$\textnormal{(\texttt{cacc})}_{q}$ holds for large $q$  \rule{0pt}{.6cm}} & 
\makebox[4.5cm][c]{Iterates Lemma \ref{step3lemma}  \rule{0pt}{.6cm}} \\
\hline
\end{tabular}
\end{table}
\subsection{A covering lemma}
We end with a covering result already employed several times in the proof of Theorem \ref{maintra}. We report it in some detail as, although incorporating various arguments already scattered in the literature, it does not appear elsewhere in such a general form. We shall often use the lemma with the choices $v=w$ and $\sigma=s$. 
\begin{lemma} \label{covlemma}
Let $v \in L^{\qq}(\BBB;\er^{k})$, $k\geq 1$, where $\BBB\subset \er^n$ is a ball,  $w \in L^1_{2s}$, be such that 
\eqn{that}
$$
[v]_{\sigma,\qq;B_{\rr/2^m}}^\qq\leq 	
 c_0  \rr^{\aaa}[\mu(B_{\rr})]^{\bbb} + c_0\rr^{\ccc}\,\tail(w;B_{\rr})^{\qq} \,,
$$
holds whenever $B_{\rr}\Subset \BBB$, $\rr\leq \rr_0\rad(\BBB)$, where $\sigma \in (0,1)$, $m \geq 1$ is an integer, $\rr_0\in (0,1)$, $\qq,\bbb\geq 1$, $c_0\ge 1$, $\aaa,\ccc\in \er$, and where finally $\mu$ is a non-decreasing, non-negative set function verifying \eqref{verificasomma}. Then
\begin{flalign}
\notag 
[v]_{\sigma,\qq;\gamma\BBB}^\qq
&\leq c\,2^{mn+(m+1)\aaa}d^\aaa[\mu(\BBB)]^\bbb \\
&\quad +\frac{c\,2^{mn+(m+1)\ccc}d^{\ccc}}{\min\{(1-\gamma), \rr_0\}^{n(\qq+1)}}
\left(\tail(w;\BBB)^\qq + \nra{w}_{L^1(\BBB)}^\qq\right)
+\frac{c|\gamma \BBB|}{d^{n+\qq \sigma}} \|v\|_{L^\qq(\gamma \BBB)}^\qq \label{parteee}
\end{flalign}
holds whenever $0<\gamma <1$, where
\eqn{tortora}
$$
d:=2^{-(m+4)}\min\{(1-\gamma), \rr_0\}\rad(\BBB)\,,
$$
and $c\equiv c(n,s,\sigma,\qq,c_0)$. The same conclusion holds if \eqref{that} is replaced by
\[
\lambda(B_{\rr/2^m})
\le
c_0\rr^\aaa[\mu(B_\rr)]^\bbb
+c_0\rr^\ccc\tail(w;B_\rr)^\qq,
\]
where $\lambda$ is a non-negative Borel measure finite on bounded sets. In this
case $\lambda(\gamma\BBB)$ is bounded by the right-hand side of \eqref{parteee},
without the last term.
\end{lemma}
\begin{proof}

We can obviously assume that $\BBB$ is centred at the origin, that is, $\BBB\equiv \ttB_\kappa$ for some $\kappa >0$. Recalling Lemmas \ref{lemmacubi}, we consider the lattice of hypercubes $\{Q_{d/\sqrt{n}}(y)\}_{y\in (2d/\sqrt{n})\mathbb Z^n}$ with $d$ as in \rif{tortora}, and select those hypercubes centred at points $I:=\{z_{j}\}\subset (2d/\sqrt{n})\mathbb Z^n$ such that $ \snr{z_{j}} \leq  \gamma \kappa+ 2d $.	
We set  $\{Q_{j}\equiv Q_{d/\sqrt{n}}(z_{j})\}_{z_j\in I }$, so that, this time 
$$
\ttB_{\gamma \kappa}\subset  \bigcup \overline{Q_{j}},\quad Q_{j_{1}}\cap Q_{j_{2}}=\emptyset \ \Longleftrightarrow \ j_{1}\not =j_{2}\,, \quad \# I \lesssim_n \frac{2^{mn}}{\min\{(1-\gamma)^n, \rr_0^n\}}.
$$
Using the family of  balls  $\{B_{d}(z_{j}) \}$ we have
\eqn{that2}
$$ 
	\ttB_{\gamma\kappa} \subset \bigcup_{j \in I} \overline{B_{d}(z_{j})}   \subset \bigcup_{j \in I} \overline{B_{2^{m+1}d}(z_{j})}\Subset \ttB_{\kappa}, \quad \{ B_{2^{m+1}d}(z_{j})\} = \bigcup_{l \leq \mathfrak{n}}  \bigsqcup_{j} B_{j}^{l}, \quad \mathfrak{n}\approx 2^{mn}.
$$
Let $\mathcal D:= \{(x,y) \in \ttB_{\gamma\kappa} \times \ttB_{\gamma\kappa}\, \colon \, |x-y| < d/4\}$ and note that
$$
\mathcal D \subset \bigcup_{j} B_{2d}(z_{j})\times B_{2d}(z_{j})
$$
therefore, using \eqref{that} we obtain
\begin{flalign}
\notag [v]_{\sigma,\qq;\ttB_{\gamma\kappa}}^\qq & \leq  \int_{\mathcal D} \frac{\snr{v(x)-v(y)}^\qq}{|x-y|^{n+\qq \sigma}} \dd \mathcal L^{2n}(x,y) + \int_{(\ttB_{\gamma\kappa}\times\ttB_{\gamma\kappa})\setminus \mathcal D} \frac{\snr{v(x)-v(y)}^\qq}{|x-y|^{n+\qq \sigma}}\dd \mathcal L^{2n}(x,y)  \\
\notag& \leq \sum_{j} [v]_{\sigma,\qq;B_{2d}(z_{j})}^\qq  + \frac{c\gamma^n\kappa^n}{d^{n+\qq \sigma}} \|v\|_{L^\qq(\ttB_{\gamma\kappa})}^\qq\\ 
\notag& \leq c 2^{(m+1)\aaa}d^{\aaa}[\mu(\ttB_\kappa)]^{\bbb-1} \sum_{j}  \mu(B_{2^{m+1}d}(z_{j}))\notag \\ &
\qquad + \frac{c\gamma^n\kappa^n}{d^{n+\qq \sigma}} \|v\|_{L^\qq(\ttB_{\gamma\kappa})}^\qq + c2^{(m+1)\ccc}d^{\ccc}\sum_{j}\tail(w;B_{2^{m+1}d}(z_{j}))^{\qq}.
\label{trtr}
\end{flalign}
Using \eqref{verificasomma} and \eqref{that2} we have
\eqn{trtr22}
$$
\sum_{j}  \mu(B_{2^{m+1}d}(z_{j})) \lesssim 2^{mn} \ \mu(\ttB_{\kappa}).
$$
Recalling the definition of $d$ and that  $\snr{z_{j}} \leq  \gamma \kappa+ 2d $, the terms appearing in the last sum can be easily estimated via \eqref{scatail.1} as follows:
$$
\tail(w;B_{2^{m+1}d}(z_{j}))^\qq\leq  \frac{c}{\min\{(1-\gamma), \rr_0\}^{n\qq}} \left(
\tail(w;\ttB_{\kappa})^\qq + \nra{w}_{L^1(\ttB_{\kappa})}^\qq
\right)
$$
and therefore
\eqn{trtr2}
$$
\sum_{j}\tail(w;B_{2^{m+1}d}(z_{j}))^\qq\leq \frac{c2^{mn}}{\min\{(1-\gamma), \rr_0\}^{n(\qq+1)}}\left(
\tail(w;\ttB_{\kappa})^\qq+  \nra{w}_{L^1(\ttB_{\kappa})}^\qq
\right).
$$ 
Connecting the content of \eqref{trtr22} and \eqref{trtr2} to the one of \eqref{trtr} we get \eqref{parteee}. The part of the lemma addressing the case of the general Borel measure  $\lambda(\cdot)$ is analogous, and, actually, simpler. 
\end{proof}

\section{Nonhomogeneous linear systems}\label{homsec2}
Using Theorem \ref{maintra}, and recalling \rif{mardef}, we prove\footnote{Formally taking $s=1$, Theorem \ref{cdg} parallels the classical result
$
\Delta w \in \mathcal{M}^{\frac{n}{1-\beta}} \Longrightarrow w \in C^{1,\beta}
$
which holds whenever $\beta\in (0,1)$ and which is sharp, according to Morrey's embedding and Calderón-Zygmund theory.  Indeed, a less known variant of Morrey's embedding claims that a function $v$ such that $Dv \in \mathcal M^{\frac{n}{1-\beta}}$ locally belongs to $C^{0,\beta}$. On the other hand, classical Calder\'on-Zygmund theory and off-diagonal interpolation give that  $\Delta w \in \mathcal{M}^{\frac{n}{1-\beta}} $ implies $D^2w \in \mathcal{M}^{\frac{n}{1-\beta}}$. At this point $w \in C^{1,\beta}$ follows by these two facts. See also \cite[Corollary 1, (C12)]{KM0}.}
\begin{theorem}\label{cdg}
Let $u\in W^{s,2}(\BBB;\er^{N})\cap L^{1}_{2s}$ be a weak solution to
\eqn{2.3.0}
$$
-\mathcal{L}_{\mathds{B}}u=g\qquad \mbox{in} \ \ \BBB,
$$
where $\BBB \subset \er^n$ is a ball, $-\mathcal{L}_{\mathds{B}}$ is as in \eqref{weaksolLL} under assumptions \eqref{condib} with $s>1/2$, and  
$$g\in \mathcal{M}^{\frac{n}{2s-1-\beta}}(\BBB;\er^{N})\,, \quad 0< \beta < 2s-1  \,.$$
Then $Du\in C^{0,\beta}_{\loc}(\BBB;\er^{N\times n})$ and 
\begin{flalign}
 & \snr{\BBB}^{\frac 1n}\nr{Du}_{L^{\infty}(\BBB/2)} +\snr{\BBB}^{\frac{1+\beta}{n}}[Du]_{0,\beta; \BBB/2} \nonumber\\
 & \qquad  \leq c \nra{u-u_{0}}_{L^{2}(\BBB)}+  c\tail(u-u_{0};\BBB)+c\snr{\BBB}^{\frac{1+\beta}{n}}\nr{g}_{\mathcal M^{\frac{n}{2s-1-\beta}}(\BBB)} \label{dg.2bis}
\end{flalign}
holds with $c\equiv c (\data,\beta )$, whenever $u_0\in \er^{N}$. 
\end{theorem}

\begin{proof} By a standard covering argument, with the local $\tail$ terms controlled as in
Lemma \ref{covlemma}, it is enough to prove the estimate in \rif{dg.2bis}. We rescale equation \rif{2.3.0} in $\ttB_1$ as described in \rif{scala1}, where in this situation we apply \rif{scala1} to $\BBB= B_{\rr}(x_0)$, $v_0\equiv u_0$ and with 
$$
V= \nra{u-u_0}_{L^{2}(\BBB)}+  \tail(u-u_0;\BBB)+\rr^{ 1+\beta}\nr{g}_{\mathcal M^{\frac{n}{2s-1-\beta}}(\BBB)}. 
$$
Note that we can assume without loss of generality that $V>0$ otherwise \eqref{dg.2bis} is trivial. 
We are therefore reducing to the situation where we have a solution $u \in W^{s,2}(\ttB_1;\er^N)\cap L^1_{2s}$ to \rif{2.3.0} with $B\equiv \ttB_1$ and such that 
\eqn{riduzione1}
$$
\nra{u}_{L^{2}(\ttB_1)}+  \tail(u;\ttB_1)+\nr{g}_{\mathcal M^{\frac{n}{2s-1-\beta}}(\ttB_1)} \leq  1\,.
$$
Our goal is to prove that there exists a constant $c\equiv c (\data,\beta )\geq 1$ such that 
\eqn{riduzione2} 
$$
\nr{Du}_{L^{\infty}(\ttB_{1/2})} +[Du]_{0,\beta; \ttB_{1/2}}\leq c \,.
$$
We recall that if  $\elly(\ti{x}):= \bbt(\ti{x}-x)+\aaat$, $\bbt \in \mathbb{R}^{N \times n}$, $\aaat \in \mathbb{R}^N$, $\ti{x}, x\in \er^n$, is an affine map, we have 
\eqn{tritri}
$$
\begin{cases}
\displaystyle \nra{\elly}_{L^2(B_{r}(x))} \approx_{n}  \snr{\bbt} r + \snr{\aaat} \\[5pt]
\displaystyle 
\tail(\elly;B_{r}(x)) \lesssim_{n} \frac{\snr{\bbt} r}{2s-1} + \frac{\snr{\aaat}}{2s}  \lesssim_{n,s} \snr{\bbt} r + \snr{\aaat}\,.
\end{cases}
$$
Note that \rif{tritri}$_2$ is guaranteed by $s>1/2$. 
Fix an arbitrary affine map $\elly$, define 
$u_{\elly}:=u-\elly$ 
and consider two balls $B_{r}(x)\Subset B_{r_{\star}}(x)\Subset \ttB_1$. By Lemma \ref{esiste} we take  $v \in \mathbb{X}_{u_{\elly}}^{s,2}(B_{r},B_{r_*})$ such that 
$$
\begin{cases}
\displaystyle
\ -\mathcal{L}_{\mathds{B}}v=0\qquad &\mbox{in} \ \ B_{r}\\ \displaystyle 
\ v=u_{\elly}\qquad &\mbox{in} \ \ \mathbb{R}^{n}\setminus B_{r}\,.
\end{cases}
$$
Since $-\mathcal{L}_{\mathds{B}}\elly =0$ by Lemma \ref{remarkino34} below, we have 
$$
w:=u_{\elly}-v \Longrightarrow \begin{cases}
\displaystyle
\ -\mathcal{L}_{\mathds{B}}w=g\qquad &\mbox{in} \ \ B_{r}\\ \displaystyle 
\ w=0\qquad &\mbox{in} \ \ \mathbb{R}^{n}\setminus B_{r}\,.
\end{cases}
$$
Using Lemma \ref{compg}, estimate \eqref{marri}, we obtain 
\eqn{2.3.2}
$$
\nra{w}_{L^{2}(B_{r})} \leq c r^{1+\beta} \nr{g}_{\mathcal{M}^{\frac{n}{2s-1-\beta}}(B_{r})}\,.
$$
Recalling that by Theorem \ref{maintra} $Dv$ is H\"older continuous in $B_{r}\equiv B_{r}(x)$, we consider the affine map 
$\ell_{x}(\ti{x}):=D v(x)(\ti{x}-x) + v(x)$, $\ti{x}\in \er^n$
and \rif{tritri} implies 
$$
 \nra{\ell_x}_{L^2(B_{r}(x))} \approx  \snr{D v(x)} r + \snr{v(x)}. 
$$ 
With 
$$
 \ti{\beta} :=\frac{\beta +2s-1}{2} \Longrightarrow  \ti{\beta}-\beta= \frac{2s-1-\beta}{2}  >0, 
$$
using \rif{almostC2s} and Mean Value Theorem we obtain
\eqn{eq:higherreg}
$$
\begin{cases}
\,\displaystyle   \nra{v-\ell_{x}}_{L^2(B_{\lambda}) }\leq c  \left(\frac{\lambda}{r}\right)^{1+\ti{\beta}}  \left(\nra{v}_{L^{2}(B_{r})}+\tail(v;B_{r})\right), \quad 
 0 < \lambda \leq \frac{r}{2} \\[10pt]
\,  \displaystyle    \snr{D  v(x)} r + \snr{ v(x)}\leq c \nra{v}_{L^{2}(B_{r})}+c\, \tail(v;B_{r})
 \end{cases}
$$
where $c\equiv c (n,N,s,\Lambda, \ti{\beta})$. Noting that 
$
\nra{v}_{L^{2}(B_{r})}  + \tail(v;B_{r}) \lesssim \tx{E}_{u}(\ell;x,r)+ \nra{w}_{L^2(B_{r})}
$, by \rif{2.3.2} it follows  
 \eqn{eq:higherreg22}
$$
\nra{v}_{L^{2}(B_{r})}  + \tail(v;B_{r})\leq c \tx{E}_{u}(\ell;x,r) + c r^{1+\beta} \nr{g}_{\mathcal{M}^{\frac{n}{2s-1-\beta}}(B_{r})}. 
$$
We shall also use 
 \eqn{eq:higherreg2222}
$$
 \tail(u_{\elly};B_{r/2})\lesssim_{n,s}
\tail(u_{\elly};B_{r})+\nra{u_{\elly}}_{L^2(B_{r})}. 
$$ 
We now fix a real number $\texttt{t}\in (0, 1/16)$ 
and estimate
\begin{eqnarray*} 
	\tx{E}_{u}(\elly + \ell_{x};x,\texttt{t} r)
	&\leq & \nra{u_{\elly}- \ell_{x}}_{L^{2}(B_{\texttt{t} r})} + \tail(u_{\elly}- \ell_{x};B_{\texttt{t} r})\nonumber \\ 
	&\stackrel{\eqref{scatailancora}}{\le}&c \nra{w}_{L^2(B_{\texttt{t}r}) } +c \nra{v-\ell_{x}}_{L^2(B_{\texttt{t} r}) } \nonumber \\
	&& +c\texttt{t}^{2s}\tail(u_{\elly};B_{r/2}) + c\texttt{t}^{2s}\tail(\ell_{x};B_{r/2}) \nonumber \\
	&&+c \texttt{t}^{2s} \nra{u_{\elly}}_{L^{2}(B_{r/2})} + c \texttt{t}^{2s} \nra{\ell_{x}}_{L^{2}(B_{r/2})} \nonumber \\
	&& +c\int_{\texttt{t} r}^{r/2}\left(\frac{\texttt{t} r}{\lambda}\right)^{2s}\nra{u_{\elly}- \ell_{x}}_{L^{1}(B_{\lambda} )}\frac{\dlam}{\lambda} \nonumber \\
	&\stackrel{\eqref{tritri}, \eqref{eq:higherreg2222}}{\le}&c \nra{w}_{L^2(B_{\texttt{t}r}) } +c \nra{v-\ell_{x}}_{L^2(B_{\texttt{t} r}) } \nonumber \\
	&&+c \texttt{t}^{2s} \nra{u_{\elly}}_{L^{2}(B_{r})} +c\texttt{t}^{2s}\tail(u_{\elly};B_{r}) + c\texttt{t}^{2s}\left(\snr{D  v(x)} r + \snr{ v(x)}\right)  \nonumber \\
		&& +c\int_{\texttt{t} r}^{r/2}\left(\frac{\texttt{t} r}{\lambda}\right)^{2s}\nra{w}_{L^{2}(B_{\lambda} )}\frac{\dlam}{\lambda}  \nonumber \\
	&& +c\int_{\texttt{t} r}^{r/2}\left(\frac{\texttt{t} r}{\lambda}\right)^{2s}\nra{v- \ell_{x}}_{L^{2}(B_{\lambda} )}\frac{\dlam}{\lambda}  \nonumber \\
	&\stackrel{\eqref{eq:higherreg}}{\le}&c\texttt{t}^{-n/2} \nra{w}_{L^2(B_{r})} +c\texttt{t}^{1+\ti{\beta}} \left ( \nra{v}_{L^{2}(B_{r})} + \tail(v;B_{r})\right ) \nonumber \\
	&&+ c\texttt{t}^{2s}\tx{E}_{u}(\elly;x,r) +c\texttt{t}^{2s}\left ( \nra{v}_{L^{2}(B_{r})} + \tail(v;B_{r})\right ) \nonumber \\ &&+c\texttt{t}^{2s}r^{2s}  \int_{\texttt{t} r}^{r/2} \lambda^{-2s}\nra{w}_{L^2(B_{\lambda}) } \frac{\dlam}{\lambda}\nonumber \\
	&&+c\texttt{t}^{2s} r^{2s-1-\ti{\beta}} \left (\nra{v}_{L^{2}(B_{r})}  + \tail(v;B_{r}) \right ) \int_{\texttt{t} r}^{r/2} \lambda^{1+\ti{\beta}-2s} \frac{\dlam}{\lambda}  \nonumber \\
	&\stackrel{\eqref{2.3.2},\eqref{eq:higherreg}}{\le} &c \texttt{t}^{-n/2} r^{1+\beta} \nr{g}_{\mathcal{M}^{\frac{n}{2s-1-\beta}}(B_{r})} +c \texttt{t}^{1+\ti{\beta}} \left ( \nra{v}_{L^{2}(B_{r})} + \tail(v;B_{r}) \right ) \nonumber \\ 
	&&+ c\texttt{t}^{1+\ti{\beta}}\tx{E}_{u}(\elly;x,r)  \\ &&
	 +c\texttt{t}^{2s} r^{n/2+2s+1+\beta}  \int_{\texttt{t}r}^{r/2}\lambda^{-2s-n/2}\frac{\dlam}{\lambda}\nr{g}_{\mathcal{M}^{\frac{n}{2s-1-\beta}}(B_{r})} \,.
\end{eqnarray*}
Using \eqref{eq:higherreg22} we conclude with 
$$
\tx{E}_{u}(\elly + \ell_{x};x,\texttt{t} r) \leq c_1\texttt{t}^{1+\ti{\beta}}\tx{E}_{u}(\elly;x,r) +c_2 \texttt{t}^{-n/2} r^{1+\beta} \nr{g}_{\mathcal{M}^{\frac{n}{2s-1-\beta}}(B_{r})}\,,
$$
where $c_1, c_2\equiv c_1, c_2(\data,\beta)>0$. We choose $\texttt t\in(0,1/16)$, depending only on $\data$ and $\beta$, such that
$
c_1\texttt{t}^{\ti{\beta}-\beta} \leq 1/2
$
and this yields, after some manipulations also involving \rif{riduzione1}
\eqn{firstorderex}
$$(\texttt{t}r)^{-1-\beta}\tx{E}_{u}(\elly + \ell_{x};x,\texttt{t}r) \leq \frac{r^{-1-\beta}}{2}\tx{E}_{u}(\elly;x,r) +c_3 $$
with $c_3:= c_2\texttt{t}^{-n/2-1-\beta}$, and therefore it is $c_3\equiv c_3(\data, \beta)$. 
We now define
$$ \mathcal{E}_{1+\beta}(x,r):= r^{-1-\beta} \inf_{\elly \textnormal{ affine}} \tx{E}_{u}(\elly;x,r)$$ so that \rif{firstorderex} translates into 
$$ \mathcal{E}_{1+\beta}(x,\texttt{t}r)\leq \frac{r^{-1-\beta}}{2}\tx{E}_{u}(\elly;x,r) +c_3.$$ 
Recalling that the affine map $\elly$ was arbitrary we conclude with 
\eqn{montal}
$$ \mathcal{E}_{1+\beta}(x,\texttt{t}r)\leq \frac{1}{2}\mathcal{E}_{1+\beta}(x,r) +c_3 \,.$$ 
We now define the sequence $\{r_j:=\texttt{t}^jr\}_{j\geq 0}$ and the following maximal type operators 
$$
 \,   \MMM^{\#}_{k}(x):= \sup_{0\leq j \leq k}\, \mathcal{E}_{1+\beta}(x,r_{j}),  \qquad 
  \,   \MMM^{\#}_{\infty}(x):=  \sup_{j \geq 0}\,  \mathcal{E}_{1+\beta}(x,r_j)
$$
for any $k \in \en_0$. 
With the above definitions from \rif{montal} it easily follows that 
$$
\MMM^{\#}_{k+1}(x) \leq \frac 12 \MMM^{\#}_{k+1}(x) + \mathcal{E}_{1+\beta}(x,r) + c_3
$$
holds for every non-negative integer $k$, where $c_3\equiv c_3(\data, \beta)$. 
This last estimate implies that $\MMM^{\#}_{k+1}(x) \leq   2\mathcal{E}_{1+\beta}(x,r) + 2c_3 $ for every $k$ so that, letting $k\to \infty$ we find $\MMM^{\#}_{\infty}(x) \leq  2\mathcal{E}_{1+\beta}(x,r) + 2c_3, $ 
and, ultimately 
\eqn{ulti}
$$
 \inf_{\elly \textnormal{ affine}}  \nra{u-\elly}_{L^2(B_{r_j}(x))}  \leq    \left(2\mathcal{E}_{1+\beta}(x,r) + 2c_3\right)r_j ^{1+\beta}
$$
for every integer $j \geq 0$. 
With $x\in \ttB_{1/2}$ we have $B_{1/2}(x)\subset \ttB_1$, and therefore, using \eqref{scatail.1} and \eqref{riduzione1}
\begin{flalign*}
\mathcal E_{1+\beta}(x,1/2)
& \leq 4
\tx E_u(0;x,1/2)
\leq 
4\nra{u}_{L^2(B_{1/2}(x))}+4\tail(u;B_{1/2}(x))
\\ & \lesssim
\nra{u}_{L^2(\ttB_1)}+\tail(u;\ttB_1)
\lesssim 1.
\end{flalign*}
We now use \rif{ulti} with $x \in \ttB_{1/2}$ and $r=1/2$ and obtain 
$$
 \inf_{\elly \textnormal{ affine}}  \nra{u-\elly}_{L^2(B_{r_j}(x))}  \leq c     r_j ^{1+\beta}\,.
$$
From this last inequality a simple interpolation argument through different scales yields
$$
\sup_{x\in \ttB_{1/2}}\,   \inf_{\elly \textnormal{ affine}}  \nra{u-\elly}_{L^2(B_{\sigma}(x))} \leq   c \sigma^{1+\beta} 
$$ 
for every $\sigma \leq 1/2$, where $c$ depends only on $\data, \beta$ (at this stage we are also using that $\texttt{t}$ depends on $\data, \beta$). Thanks to this last inequality, the standard integral characterization of H\"older continuity due to Campanato and Meyers (see for instance \cite{cam00}) applies  and leads to  \rif{riduzione2}. 
\end{proof}
The following standard lemma has already been used in the previous proof. It can be found for instance in \cite[Remark 3.4]{kns}.  
\begin{lemma}\label{remarkino34} Let $\mathds{B} \colon  \er^{n} \to \er^{N\times N}$ be a tensor field satisfying the second condition in \eqref{condib} and let $s \in (1/2,1)$. Then $-\mathcal{L}_{\mathds{B}} \elly=0$ holds for every affine map $\elly$.   
  \end{lemma}

\section{Higher differentiability}\label{highsec}
In this section we derive higher differentiability for weak solutions to \eqref{nonlocaleqn}. In particular, we prove \eqref{higherss}.
\begin{theorem}\label{maggiore}
Under assumptions \eqref{bs.1} with 
$1/2\leq s <1$, 
let $u$ be a weak solution to \eqref{nonlocaleqn}. There exists $\delta_{0}\equiv \delta_{0}(\data)\in (0,1-s)$ such that $u\in W^{1+ \mathfrak s,2}_{\loc}(\Omega;\er^{N})$ for all $ \mathfrak s\in (0,s+\delta_{0})$. Moreover, if $0<\gamma<1$, then
\eqn{wdown}
$$ 
\rrr^{1+ \mathfrak s}\snra{Du}_{ \mathfrak s,2;\gamma \BBB}+\rrr\nra{Du}_{L^2(\gamma \BBB)}\lesssim_{\data,  \mathfrak s,\gamma}  \rrr^{s}\snra{u}_{s,2;\BBB}+\nra{u}_{L^{2}(\BBB)}+\tail(u;\BBB) 
$$
holds whenever $\BBB\equiv \BBB_{\rrr}\Subset \Omega$ is a ball. 
\end{theorem}
We first need  a $W^{1,2}$-estimate. 
\begin{proposition}\label{gradpre}
Under assumptions \eqref{bs.1} with 
$1/2\leq s <1$, let $u$ be a weak solution to \eqref{nonlocaleqn}. Then $u \in W^{1,2}_{\loc}(\Omega;\er^{N})$. Moreover, if $0<\gamma<1$, then 
\eqn{wdown0} 
$$
\rrr\nra{Du}_{L^2(\gamma \BBB)}\lesssim_{\data,\gamma} \rrr^s\snra{u}_{s,2;\BBB}+\nra{u}_{L^{2}(\BBB)}+\tail(u;\BBB)
$$
holds whenever $\BBB\equiv \BBB_{\rrr}\Subset \Omega$. 
\end{proposition}
\begin{proof} 
{\em Step 1: Higher fractional differentiability for difference quotients.} In this section we shall use the finite difference operators $\tau_{h}\equiv \tau_{h, \er^n}$ defined in \eqref{enfasi} and  \eqref{domania}. We start by considering the special case in which 
\eqn{reduce}
$$
u\in W^{s,2}(\ttB_5;\er^{N})\cap L^{1}_{2s}\quad \mbox{solves} \quad -\mathcal{N}_{\sAAA}u=0\quad  \mbox{in $\ttB_5$}.
$$ We can reduce to this case by a  standard scaling argument\footnote{Consider a ball $B_{5\rr}\equiv B_{5\rr}(x_0)\Subset \Omega$ and define $u_\rr(x):=\rr^{-s}u(x_0+\rr x)$, which satisfies \eqref{reduce}. By writing down \eqref{wdown0} for $u_\rr$, and scaling back we obtain 
$
\rr\nra{Du}_{L^2(B_{\rr/16})}\le c\rr^s\snra{u}_{s,2;B_{5\rr}}+c\nra{u}_{L^{2}(B_{5\rr})}+c\tail(u;B_{5\rr}), 
$
whenever $B_{5\rr}\Subset \Omega$. By using the covering argument in Lemma \ref{covlemma} we finally conclude with \eqref{wdown0}.}. We preliminarily use the procedure of Lemma \ref{el : lem.loc} below. This yields $w \in W^{s,2}(\er^{n};\er^{N})$, which is defined in \eqref{localizzata}, and solves  
$-\mathcal{N}_{\sAAA}w=f$ in $\ttB_{2}$, that is 
\eqn{theoned}
$$
\int_{\mathbb{R}^{n}}\int_{\mathbb{R}^{n}}\langle\sAA \left(\frac{w(x)-w(y)}{|x-y|^s}\right),\frac{\varphi(x)-\varphi(y)}{|x-y|^s}\rangle\frac{\dxy}{|x-y|^{n}}= \int_{\ttB_{2}} 
\langle f , \varphi 
\rangle\dx 
$$
holds for every $\varphi \in W^{s,2}(\er^n;\er^{N})$ with compact support in $\ttB_2$. Moreover 
\eqn{same}
$$ \mbox{$u\equiv w$ \ \ \ on $\ttB_3$ \ \ \ \ and \ \ \ \ $w\equiv 0$\ \ \  outside $\ttB_5$}\,.$$
Here  
$f$
 satisfies \eqref{eq : lem.loc}-\eqref{el : lem.loc.est} with $B_{3\rr}(x_0)\equiv \ttB_3$, for every $x \in \ttB_2$ and $ |h|< 1/4$. We now prove that there exist numbers $\delta_{0}\equiv \delta_{0}(\data)\in (0,1-s)$ and $c\equiv c(\data)$ such that 
 \eqn{hid.20}
$$ 
\left[\frac{\tau_{h}w}{|h|^{\beta}}\right]_{s+\delta_{0},2;\ttB_{1/4}}\le c\left\|\frac{\tau_{h}w}{|h|^\beta}\right\|_{L^{2}(\ttB_{1})}+c\tail\left(\frac{\tau_{h}w}{|h|^{\beta}};\ttB_{1}\right) +c\left\|\frac{\tau_{h}f}{|h|^\beta}\right\|_{L^{2}(\ttB_{1})}.
$$
 holds whenever 
 $h \in \er^n$, and $\beta \in \er$ such that 
 \eqn{limitazioni} 
$$0< \beta\leq 1\,, \qquad 0< |h|< \frac{1}{2^4}. $$ 
We stress that $\delta_0$ and $c$ appearing in \eqref{hid.20} are independent of $h$ and $\beta$. 
To this aim we fix $\beta,h$ as in \eqref{limitazioni}, recall the meaning of $\tilde{\tau}_{h}$, which has been introduced in \eqref{ttau2}, and define $\sA_{h}\colon \er^{2n} \to \er^{N\times N}$ as
\eqn{hid.0} 
$$
\sA_{h}(x,y):=\int_{0}^{1}\partial\sAA \left(\frac{w(x)-w(y)+\lambda(\tilde{\tau}_{h}(w(x)-w(y)))}{|x-y|^{s}}\right)\dd\lambda
$$
whenever $x, y \in  \er^n$. Also recalling \eqref{ttau2}-\eqref{matobo2}, we have 
\eqn{hid.02} 
$$
\tilde{\tau}_{h}\sAA \left(\frac{w(x)-w(y)}{|x-y|^{s}}\right)= \sA_{h}(x,y)\frac{\tilde{\tau}_{h}(w(x)-w(y))}{|x-y|^s}= \sA_{h}(x,y)\frac{\tau_{h}w(x)-\tau_{h}w(y)}{|x-y|^s}
$$ 
and, by \eqref{bs.1}
\eqn{hid.00} 
$$
\Lambda^{-1}\snr{\xi}^2 \leq \langle\sA_{h}(x,y) \xi, \xi \rangle, \quad  \snr{\sA_{h}(x,y)} \leq \Lambda
$$
holds for every $\xi\in \er^{N}$. Moreover, using \eqref{sim} it follows that 
\eqn{hid.000} 
$$
\sA_{h}(x,y)=\sA_{h}(y,x).
$$
With $\psi\in W^{s,2}(\er^n;\er^{N})$ having compact support in $\ttB_{1}$, we test \eqref{theoned} with $\varphi:=\tau_{-h}\psi$, which is an admissible test function. We want to use the material of Section \ref{doppidiff}. In particular, we use \eqref{matobo2}  and integrate by parts according to  \eqref{matobo1}, so that 
\begin{flalign*}
& \int_{\er^n}\int_{\er^n}\langle\tilde\tau_{h}\sAA \left(\frac{w(x)-w(y)}{|x-y|^s}\right),\frac{\psi(x)-\psi(y)}{|x-y|^{s}}\rangle \frac{\dxy}{|x-y|^{n}}\\
& \quad = \int_{\er^n}\int_{\er^n}\langle\sAA \left(\frac{w(x)-w(y)}{|x-y|^s}\right),\frac{\tilde\tau_{-h}(\psi(x)-\psi(y))}{|x-y|^{s}}\rangle \frac{\dxy}{|x-y|^{n}}\\
& \quad =\int_{\ttB_2} \langle f ,\tau_{-h}\psi\rangle\dx=\int_{\ttB_1} \langle   \tau_{h} f ,\psi\rangle\dx\,.
\end{flalign*}
Recalling \eqref{hid.0}-\eqref{hid.02}, dividing the above identity by $|h|^\beta$ and denoting
\eqn{diffwf} 
$$
\texttt{w}_h:=\frac{\tau_{h}w}{|h|^\beta}, \quad \texttt{f}_h:=\frac{\tau_{h}f}{|h|^\beta}
$$
we  arrive at 
$$
\int_{\er^n}\int_{\er^n}\langle\sA_{h}(x,y)\frac{\texttt{w}_h(x)-\texttt{w}_h(y)}{|x-y|^s}
,\frac{\psi(x)-\psi(y)}{|x-y|^s}\rangle \frac{\dxy}{|x-y|^{n}}=\int_{\ttB_{1}}\langle \texttt{f}_h,\psi\rangle\dx.
$$
As specified above, this holds whenever $\psi\in W^{s,2}(\er^n;\er^{N})$ has compact support in $\ttB_{1}$, that is, 
\eqn{linni}
$$
 -\mathcal{L}_{\mathds{A}_{h}}\texttt{w}_h =\texttt{f}_h\ \ \mbox{holds in $\ttB_{1}$}
$$
in the sense of Definition \ref{def:weaksol2}. 
In fact, thanks to \eqref{hid.00}-\eqref{hid.000}, the one in \eqref{linni} is exactly a system of the type considered in \eqref{solutio}. Therefore we can apply 
the fractional version of Gehring's lemma developed in \cite[Theorem 5.5]{parte1}, see also \cite[Theorem 1.1]{KMS} and \cite[Theorem 1.1]{byunnon}. Specifically, there exists a universal higher differentiability constant $\delta_0$ such that $
0 < \delta_0\equiv \delta_0 (\data) <1-s$, $\texttt{w}_h \in W^{s+\delta_0,2}_{\loc}(\ttB_{1/2};\er^{N})$
and 
\eqn{altina}
$$
[\texttt{w}_h  ]_{s+\delta_{0},2;\ttB_{1/4}}\le c [\texttt{w}_h  ]_{s,2;\ttB_{1/2}}+c\tail(\texttt{w}_h ;\ttB_{1/2}) +c\nr{\texttt{f}_h }_{L^{2}(\ttB_{1/2})}
$$
holds with  $c\equiv c(\data)$. 
Again by \eqref{linni} we can apply Lemma \ref{caccioppola} to $\texttt{w}_h$ and this yields 
$$
[\texttt{w}_h]_{s,2;\ttB_{1/2}} \le c\nr{\texttt{w}_h}_{L^{2}(\ttB_{1})}+c\tail\left(\texttt{w}_h;\ttB_{1}\right) +c\nr{\texttt{f}_h}_{L^{2}(\ttB_{1})},  
$$
that holds with $c\equiv c(\data)$. Combining the content of the last two displays, using \eqref{scatailancora2} and recalling the definitions in \eqref{diffwf}  finally leads to \eqref{hid.20}.

\noindent {\em Step 2: A $W^{1,2}$-estimate.}  Here we prove  that
\eqn{hid.4}
$$
\nr{Du}_{L^2(\ttB_{1/16})}\le c[u]_{s,2;\ttB_5}+c\nr{u}_{L^{2}(\ttB_{5})}+c\tail(u;\ttB_{5}) 
$$
holds with $c\equiv c(\data)$. For the proof, we take  $\beta=s$ in \eqref{hid.20} and estimate the resulting right-hand side terms. 
We start by using Lemma \ref{prop:embedding0}, that gives
\eqn{meola}
$$
\left\|\frac{\tau_{h}w}{|h|^s}\right\|_{L^{2}(\ttB_{1})}\stackrel{\eqref{same}}{=} \left\|\frac{\tau_{h}u}{|h|^s}\right\|_{L^{2}(\ttB_{1})}\le c[u]_{s,2;\ttB_5}+c\nr{u}_{L^{2}(\ttB_{5})}
$$
with $c\equiv c(n,N,s)$. 
Recalling that $|h|<1/2^4$, again via Lemma \ref{prop:embedding0} we find
\begin{flalign}
\notag \tail\left(\frac{\tau_{h}w}{|h|^{s}};\ttB_{1}\right)&\le  \int_{\ttB_{6}\setminus \ttB_{1}}\left|\frac{\tau_{h}w}{|h|^{s}}\right|\frac{\dx}{\snr{x}^{n+2s}} \le   c \left\|\frac{\tau_{h}w}{|h|^s}\right\|_{L^{2}(\ttB_{6})}\nonumber \\
&\le  c[w]_{s,2;\er^n}+c\nr{w}_{L^{2}(\er^n)} \stackrel{\eqref{estendi}}{\le} c[u]_{s,2;\ttB_5}+c\nr{u}_{L^{2}(\ttB_{5})}\label{asino}
\end{flalign}
for $c\equiv c(n,N,s)$.
Finally, using first \eqref{el : lem.loc.est}  and then \eqref{meola} yields
\begin{flalign*}
\left\|\frac{\tau_{h}f}{|h|^s}\right\|_{L^{2}(\ttB_{1})} &\leq c\left\|\frac{\tau_{h}u}{|h|^s}\right\|_{L^{2}(\ttB_{1})}+c\nr{u}_{L^{2}(\ttB_{2})} +c \tail(u;\ttB_{5})\\
&  \leq c[u]_{s,2;\ttB_5}+c\nr{u}_{L^{2}(\ttB_{5})}+c \tail(u;\ttB_{5})
\end{flalign*}
with $c\equiv c(n,N,s)$. Note that we are using also \eqref{scatailancora2}, and we shall use it all the time. Plugging the content of the three above displays in \eqref{hid.20} (with $\beta=s$) we obtain
$$
\left[\frac{\tau_{h}u}{|h|^{s}}\right]_{s+\delta_{0},2;\ttB_{1/4}}\le c[u]_{s,2;\ttB_5}+c\nr{u}_{L^{2}(\ttB_{5})}+c\tail(u;\ttB_{5}),
$$
where $c\equiv c(\data)$. Next, by Lemma \ref{prop:embedding0} applied this time to $|h|^{-s}\tau_{h}u$, we have 
\begin{eqnarray*}
\left[\frac{\tau_{h}u}{|h|^{s}}\right]_{s+\delta_{0},2;\ttB_{1/4}}&\ge&\frac{1}{c}\left\|\frac{\tau_{h}^{2}u}{|h|^{2s+\delta_{0}}}\right\|_{L^{2}(\ttB_{1/8})}-c\left\|\frac{\tau_{h}u}{|h|^{s}}\right\|_{L^{2}(\ttB_{1/4})}\nonumber \\
&\stackrel{\eqref{meola}}{\ge}&\frac{1}{c}\left\|\frac{\tau_{h}^{2}u}{|h|^{2s+\delta_{0}}}\right\|_{L^{2}(\ttB_{1/8})}-c[u]_{s,2;\ttB_5}-c\nr{u}_{L^{2}(\ttB_{5})},
\end{eqnarray*}
with $c\equiv c(n,N,s)$. Merging the inequalities in the two previous displays, we obtain
\eqn{bico}
$$
\sup_{0<|h|\leq 1/2^4}\, \left\|\frac{\tau_{h}^{2}u}{|h|^{2s+\delta_{0}}}\right\|_{L^{2}(\ttB_{1/8})}\le c[u]_{s,2;\ttB_5}+c\nr{u}_{L^{2}(\ttB_{5})}+c\tail(u;\ttB_{5}) ,
$$
for $c\equiv c(\data)$. As $2s\geq 1$ and $0<\delta_0<1-s$ as in \eqref{altina}, we have $1<2s+\delta_0<2$ and using Lemma \ref{l4dd}, estimate \eqref{immersione2}, we get 
\eqref{hid.4}.  
\end{proof}
\begin{remark}{\em Unlike the other main results of the paper, in Theorem \ref{maggiore} the borderline case $s=1/2$ is allowed. This is ultimately a consequence of the higher fractional differentiability in \eqref{altina}. This serves to ensure that the exponent $2s+\delta_{0}$ appearing in \eqref{bico} is larger than one also when $s=1/2$. This allows us to apply Lemma \ref{l4dd}.}
\end{remark}
\begin{proof}[Proof of Theorem \ref{maggiore}] We capitalize on Proposition \ref{gradpre} and its proof. More precisely, we restart reducing the problem as in \eqref{reduce}, and perform the localization in  Lemma \ref{el : lem.loc}  but now with upgraded regularity given by Proposition \ref{gradpre}, that gives 
\eqn{hid.44}
$$
\nr{Du}_{L^2(\ttB_{\gamma})}\lesssim_{\data,\gamma} [u]_{s,2;\ttB_5}+\nr{u}_{L^{2}(\ttB_5)}+\tail(u;\ttB_5)
$$
whenever $0<\gamma <5$. 
In \eqref{hid.20} we choose $\beta=1$ 
and estimate the resulting right-hand side terms. By means of \eqref{hid.44} we have 
$$
\left\|\frac{\tau_{h}w}{|h|}\right\|_{L^{2}(\ttB_{1})}=\left\|\frac{\tau_{h}u}{|h|}\right\|_{L^{2}(\ttB_{1})}\le c\nr{Du}_{L^2(\ttB_{2})}\leq c[u]_{s,2;\ttB_5}+c\nr{u}_{L^{2}(\ttB_5)}+c\tail(u;\ttB_5).
$$
As done in \eqref{asino}
\begin{flalign*}
\notag & \tail\left(\frac{\tau_{h}w}{|h|};\ttB_{1}\right) \le   c \left\|\frac{\tau_{h}w}{|h|}\right\|_{L^{2}(\ttB_{5+1/2^4})} \leq  c\nr{Dw}_{L^{2}(\er^n)}\nonumber \\
&\quad  \stackrel{\eqref{estendi}_2}{\le}c\nr{Du}_{L^2(\ttB_{4})}+c\nr{u}_{L^{2}(\ttB_{4})}\stackleq{hid.44}  c[u]_{s,2;\ttB_5}+c\nr{u}_{L^{2}(\ttB_5)}+c\tail(u;\ttB_5).
\end{flalign*}
Finally, we obtain
\begin{eqnarray*}
\left\|\frac{\tau_{h}f}{|h|}\right\|_{L^{2}(\ttB_{1})}&\stackrel{\eqref{el : lem.loc.est}_2}{\le}&c\left\|\frac{\tau_{h}u}{|h|}\right\|_{L^{2}(\ttB_{1})} +c\nr{u}_{L^{2}(\ttB_{5})}+c\tail(u;\ttB_{5})\nonumber \\
&\leq &c\nr{Du}_{L^{2}(\ttB_{2})} +c\nr{u}_{L^{2}(\ttB_{5})}+c\tail(u;\ttB_{5})\nonumber \\
&\stackrel{\eqref{hid.44}}{\le}&c[u]_{s,2;\ttB_5}+c\nr{u}_{L^{2}(\ttB_{5})}+c\tail(u;\ttB_{5}).
\end{eqnarray*}
In all the above estimates $c$ depends only on $\data$. Using the content of the last three estimates in \eqref{hid.20}, used with $\beta=1$, we conclude with
$$
\left[\frac{\tau_{h}u}{|h|}\right]_{s+\delta_{0},2;\ttB_{1/4}} \le c[u]_{s,2;\ttB_5}+c\nr{u}_{L^{2}(\ttB_{5})}+c\tail(u;\ttB_{5}).
$$
Finally, using Lemma \ref{prop:embedding0} as for  \eqref{bico}, this time we bound
$$
\sup_{0<|h|\leq 1/2^4}\, \left\|\frac{\tau_{h}^{2}u}{|h|^{1+s+\delta_{0}}}\right\|_{L^{2}(\ttB_{1/8})}\le c[u]_{s,2;\ttB_5}+c\nr{u}_{L^{2}(\ttB_{5})}+c\tail(u;\ttB_{5}). 
$$
Then, Lemma \ref{l4dd}, estimate \eqref{immersione2b}, together with \rif{hid.44}, now give
$$
[Du]_{\mathfrak s, 2;\ttB_{1/16}}\le c[u]_{s,2;\ttB_5}+c\nr{u}_{L^{2}(\ttB_{5})}+c\tail(u;\ttB_{5})
$$ 
where $c\equiv c(\data,\mathfrak s)$, for every $\mathfrak s \in (0, s+\delta_0)$. In order to conclude with \eqref{wdown}  it is now sufficient to use the usual scaling and covering argument of Lemma \ref{covlemma}. 
\end{proof}
We end with the localization argument that has been used in the proof of Proposition \ref{gradpre}. It is essentially a straightforward modification of \cite[Lemma 3.2]{DKLN}. We report a sketch of the proof for completeness, emphasising the differences between this version and the original, both in the statement and in the proof. 
\begin{lemma}[Localization]
	\label{el : lem.loc}
	Let $u\in W^{s,2}(B_{5\rr};\er^{N})\cap L^{1}_{2s}$ be a weak solution to \eqref{nonlocaleqn} in  $\Omega\equiv B_{5\rr}\equiv B_{5\rr}(x_{0})\Subset\er^n$. Fix a cutoff function $\eta \in C_{0}^{\infty}\left(B_{4\rr}\right)$ with $0\le \eta\le1$, $\eta\equiv 1$ on $B_{3\rr}$ and $|D \eta| \lesssim 1/\rr$. Then 
	\eqn{localizzata}
	$$w:= \eta u\in W^{s,2}(\er^n;\er^{N})$$ 
	is a weak solution to
	\begin{equation}\label{eq : lem.loc}
		-\mathcal{N}_{\sAAA}w=f \quad\text{in }B_{2\rr}
	\end{equation}
	for some $f \in L^{\infty}\left(B_{2\rr}\right)$.  Moreover, for any  $h \in \er^n$ with $|h|\leq \rr/4$, it holds that 
\eqn{el : lem.loc.est}
$$
\begin{cases}					
|f(x)| \leq c  \rr^{-2s} {\tail(u;B_{3\rr})},\  \mbox{a.e.} \ x \in B_{2\rr} \\[1mm] 
			|\tau_hf(x)| \leq c\rr^{-2s} |\tau_hu(x)|\\[1mm] 
			\qquad \qquad \      +c\rr^{-2s-1}|h|\left[|u(x+h)|+|u(x)|+{\tail(u;B_{3\rr})}\right]
			,\mbox{a.e.} \ x \in B_{\rr} 
\end{cases}
$$
where  $c\equiv c(\data)$. Finally
\eqn{estendi}
$$
\begin{cases}
[w]_{s,2;\er^n}
\leq c[u]_{s,2;B_{5\rr}}+c\rr^{-s}\nr{u}_{L^{2}(B_{5\rr})}\\[3pt]
\nr{Dw}_{L^2(\er^n)}
\leq c\nr{Du}_{L^2(B_{4\rr})}+c\rr^{-1}\nr{u}_{L^{2}(B_{4\rr})}\,.
\end{cases}
$$
Note that the last inequality is meaningful whenever $u\in W^{1,2}(B_{4\rr};\er^{N})$. 
\end{lemma}

\begin{proof} The proof follows from that of \cite[Lemma 3.2]{DKLN}, where $\sAA $ is denoted by $\Phi$ and the solution is scalar-valued, i.e., $N=1$. The proof of \eqref{eq : lem.loc} with $f \in L^\infty(B_{2\rr})$ and \eqref{el : lem.loc.est}$_1$ follows verbatim as in \cite{DKLN}. In this respect note that \eqref{bs.1}$_2$ also implies that 
\eqn{lippinono}
$$\snr{\sAA (w_{1})- \sAA (w_{2})}\le \Lambda \snr{w_{1}-w_{2}},$$
for every choice of $w_1, w_2 
\in \er^{N}$, which is used crucially everywhere in \cite[Lemma 3.2]{DKLN}. In order to prove \eqref{el : lem.loc.est}$_2$, we again follow \cite{DKLN}, where a similar result is provided using the fact that $u$ is locally H\"older continuous in that setting. We report the necessary modifications; in the following all the balls will be centred at $x_0$. Let us fix $x \in B_{\rr}$, $h \in \ttB_{\rr/4}$ so that 
\begin{equation}
		\label{el : ineq3.lem.loc}
		6 \min\{|x-y|,|x+h-y|\}\geq  |y-x_0|\quad\text{for any }y\in \mathbb{R}^{n}\setminus B_{3\rr}. 
	\end{equation}
	 Write \cite[(3.7)]{DKLN} with $x_1=x+h$, $x_2=x$ to estimate 
	 $
	 |\tau_h f(x)| \lesssim \snr{J_{1}}+ \snr{J_{2}}
	 $, 
	 and get the corresponding expressions $J_{1}= J_{1,1}+J_{1,2}+J_{1,3}$ as in  \cite[(3.8)]{DKLN} and subsequent identity.  
As for $J_{1,1}$, using \eqref{lippinono} and \eqref{el : ineq3.lem.loc}, we have
	\begin{align*}
		|J_{1,1}| &\leq c \int_{\mathbb{R}^{n}\setminus B_{3\rr}}\frac{|\tau_h w(x)|}{|x+h-y|^{n+2s}} \,dy\leq c |\tau_h w(x)| \int_{\mathbb{R}^{n}\setminus B_{3\rr}}\frac{1}{|y-x_{0}|^{n+2s}}\,dy\\
		&\leq c \rr^{-2s} |\tau_h w(x)|  =  c \rr^{-2s} |\tau_hu(x)|,
	\end{align*}
	with $c\equiv c(\data)$. The remaining terms are estimated using the elementary inequality
	$$
	\left|\frac{1}{\snr{x+h-y}^a}-\frac{1}{|x-y|^a}\right|\lesssim_{a} \frac{ |h|^\alpha
}{\snr{y-x_0}^{a+\alpha}}, \quad a>0, \  0< \alpha \leq 1.$$
Note that this uses \eqref{el : ineq3.lem.loc}.	
For $J_{1,2}$, we use \eqref{el : ineq3.lem.loc} and the above inequality 
	\begin{align*}
		|J_{1,2}| &\leq c \int_{\mathbb{R}^{n}\setminus B_{3\rr}}|w(x)-w(y)| \left|\frac{1}{|x+h-y|^{s}}-\frac{1}{|x-y|^{s}}\right|\frac{\,dy}{|x+h-y|^{n+s}} \\
		&\leq c|h|   \int_{\mathbb{R}^{n}\setminus B_{3\rr}} \frac{|w(x)| + |w(y)|}{|y-x_{0}|^{n+2s+1}}\,dy \leq c\rr^{-2s-1} |h| \left[|w(x)|+\tail(w;B_{3\rr})\right].
	\end{align*}
Similarly to   $J_{1,2}$, for $J_{1,3}$ we have
	\begin{align*}
		|J_{1,3}| &\leq \int_{\mathbb{R}^{n}\setminus B_{3\rr}} \frac{|w(x)|+|w(y)|}{|y-x_{0}|^{s}} \left|\frac{1}{|x+h-y|^{n+s}}-\frac{1}{|x-y|^{n+s}}\right|\,dy\\
		&\leq c\rr^{-2s-1} |h| \left[|w(x)|+{\tail(w;B_{3\rr})}\right], 
	\end{align*}
	for some $c\equiv c(\data)$. Therefore, recalling the definition of $w$ in \eqref{localizzata}, we finally gain
	\begin{align*}
		|J_{1}| \leq  c\rr^{-2s} |\tau_hu(x)| +c\rr^{-2s-1} |h|\left[|u(x+h)|+|u(x)|+{\tail(u;B_{3\rr})}\right]
	\end{align*}
	for some constant $c\equiv c(\data)$. Again as in \cite{DKLN} we can define the term $J_2$ and this can be estimated exactly with the modifications done here for $J_1$. The final outcome is again
$$
		|J_{2}| \leq  c \rr^{-2s} |\tau_hu(x)| +c\rr^{-2s-1} |h|\left[|u(x)|+{\tail(u;B_{3\rr})}\right]
$$
	for some constant $c\equiv c(n,s,\Lambda)$. The estimates for $J_1$ and $J_2$ imply \eqref{el : lem.loc.est}$_2$.  
\end{proof}

\section{Nonlocal linearization and flatness improvement}\label{coresec}
The main aim of this section is to prove
\begin{proposition}[Flatness improvement]\label{cor.1}
Under assumptions \eqref{bs.1} with $s>1/2$, let $u$ be a weak solution to \eqref{nonlocaleqn}, let $B_{\rr}(x_{0})\Subset \Omega$ be a ball and  $\alpha \in (0, 2s-1)$. There exists a positive threshold $\epsb{b}\equiv \epsb{b}(\data,\omega(\cdot), \alpha)\in (0,1)$,  such that if there exists an affine map $\ell$ such that 
\eqn{small}
$$
\rr^{-s}\tx{E}_{u}(\ell;x_{0},\rr)<\epsb{b}\,,
$$
then 
\eqn{exx.5}
$$
 \inf_{\elly \textnormal{ affine}}  \nra{u-\elly}_{L^2(B_{\sigma}(x_0))} \leq c 
\left(
 \nra{u}_{L^{2}(B_{\rr}(x_0))}+ \tail(u;B_{\rr}(x_0))\right)
\left(\frac{\sigma}{\rr}\right) ^{1+\alpha}
$$ 
holds for every $\sigma \leq \rr$, where $c\equiv c(\data,\omega(\cdot), \alpha)$. 
\end{proposition}
The proof of Proposition \ref{cor.1} will be carried out in Sections \ref{blowsec}--\ref{completa}. \subsection{Nonlocal linearization and blow-up} \label{blowsec} 
Let $u\in W^{s,2}_{\loc}(\Omega;\er^{N})\cap L^1_{2s}$ be a weak solution to \eqref{nonlocaleqn} and let $B_{\rr}(x_{0})\Subset \Omega$ be a ball. Recalling Definition \ref{leccesso} we further introduce 
$$ 
\sE_{w}(\ell;x_{0},\rr):=\rr^{-s} \tx{E}_{w}(\ell;x_{0},\rr)\,.
$$
Moreover, with $\aaat\in \er^{N}$, $\bbt \in \mathbb{R}^{N\times n}$, we fix the generic affine map $\ell(x):=\bbt (x-x_{0})+\aaat$, $x \in \er^n$. In the following we shall assume without loss of generality that
$u\not\equiv \ell$ which implies that $\sE_{u}(\ell;x_{0},\rr)>0$ (otherwise $u$ would be an affine map making \eqref{exx.5} obvious). Then we let 
\eqn{matrix}
$$
u_{\rr}(x):=\frac{u(x_{0}+\rr x)}{\sE_{u}(\ell;x_{0},\rr)\rr^s},\quad \quad \ell_{\rr}(x):=\frac{\rr\bbt x+\aaat}{\sE_{u}(\ell;x_{0},\rr)\rr^s}\,,\quad x\in \mathbb{R}^{n}, 
$$
and
\eqn{matrici01}
$$
\sAA_{\rr}(w):=\frac{\sAA \left( \sE_{u}(\ell;x_{0},\rr)w\right)}{ \sE_{u}(\ell;x_{0},\rr)}\,,\quad w\in \er^{N}, 
$$
so that, by \eqref{sim}, we have 
 \eqn{basino}
 $$
 \sAA_{\rr}(0)= 0\,, \quad 
 \partial \sAA_{\rr}(w) =  \partial \sAA (\sE_{u}(\ell;x_{0},\rr)w)\Longrightarrow \partial \sAA_{\rr}(-w)=\partial \sAA_{\rr}(w)\,.
 $$
 We further define, for $ x,y\in \mathbb{R}^{n}$
\eqn{matrici02}
$$
 \sA_{\rr,*}(x,y)\equiv \sA_{\rr,*}(x-y):= \partial \sAA_{\rr}\left(\frac{\ell_{\rr}(x)-\ell_{\rr}(y)}{|x-y|^{s}}\right) \stackrel{\eqref{matrix},\eqref{basino}}{=}
 \partial \sAA \left(\frac{\rr\bbt( x-y)}{\rr^s|x-y|^{s}}\right)
$$
and \eqref{basino} implies that
\eqn{sim0}
$$
 \sA_{\rr,*}(x-y)=  \sA_{\rr,*}(y-x)
$$
holds whenever $x,y\in \er^n$. 
Note that, with 
$$ \Omega_{x_0;1/\rr}:=\left\{x\in \mathbb{R}^{n}\colon x_{0}+\rr x\in \Omega\right\}= \frac{\Omega-x_0}{\rr},$$ the map $u_{\rr}$ solves 
$
 -\mathcal{N}_{\scalebox{0.65}{$\sAA_\rr$}}u_{\rr}=0$ in $\Omega_{x_0;1/\rr}$, that is 
\eqn{riscalata}
$$
 \int_{\mathbb{R}^{n}}\int_{\mathbb{R}^{n}}\langle \sAA_{\rr}\left(\frac{u_{\rr}(x)-u_{\rr}(y)}{|x-y|^{s}}\right),\frac{\varphi(x)-\varphi(y)}{|x-y|^{s}}\rangle\frac{\dxy}{|x-y|^{n}}=0
$$
holds whenever $\varphi \in W^{s,2}(\er^n;\er^{N})$
 has compact support in $\Omega_{x_0;1/\rr}$. 
 \begin{lemma} The matrix $\hat{\sA}_{\rr} \colon\er^{n}\to \er^{N\times N}$, defined by 
 \eqn{nucleo}
 $$
\hat{\sA}_{\rr}(x-y):=\int_{0}^{1}\partial\sAA \left(\frac{\lambda \rr \bbt(x-y)}{\rr^s|x-y|^{s}}\right)\dd\lambda, \quad x,y \in \er^n, 
$$
is  translation-invariant and moreover it satisfies 
\eqn{nunu}
$$
 \Lambda^{-1}\snr{\xi}^{2} \leq \langle \hat{\sA}_{\rr}(x-y)\xi,\xi\rangle,\quad \snr{\hat{\sA}_{\rr}(x-y)}\le \Lambda, \quad 
\hat{\sA}_{\rr}(x-y)=\hat{\sA}_{\rr}(y-x)
$$
whenever $x, y \in \er^n$, $\xi \in \er^N$. Moreover 
\eqn{0.1.}
$$
\int_{\mathbb{R}^{n}}\int_{\mathbb{R}^{n}}\langle\hat{\sA}_{\rr}(x-y)\frac{\ell_{\rr}(x)-\ell_{\rr}(y)}{|x-y|^s},\frac{\varphi(x)-\varphi(y)}{|x-y|^s}\rangle\frac{\dxy}{|x-y|^{n}}=0
$$
holds for all $\varphi\in W^{s,2}(\mathbb{R}^{n};\er^{N})$ with compact support in $\Omega_{x_0;1/\rr}$.
 \end{lemma}
 \begin{proof}
The translation invariance and \eqref{nunu} are a direct consequence of the definitions given  together with \eqref{bs.1} and \eqref{sim}. Finally \eqref{0.1.} follows by Lemma \ref{remarkino34}. 
 \end{proof}
\begin{proposition}[Nonlocal linearization]\label{lineaprop}The map
\eqn{defivrr} 
$$
v_{\rr}:=u_{\rr}-\ell_{\rr}
$$
belongs to $W^{s,2}_{\loc}(\Omega_{x_0;1/\rr};\er^{N})\cap L^1_{2s}$ and is a weak solution to the linearized equation 
\eqn{linearizzata}
$$
 -\mathcal{L}_{\scalebox{0.6}{$\sA_\rr$}}v_{\rr}=0\qquad \mbox{in} \ \ \Omega_{x_0;1/\rr}
 $$ 
 where
 $$
 \sA_{\rr}(x,y):=\int_{0}^{1}\partial \sAA_{\rr}\left(\frac{\ell_{\rr}(x)-\ell_{\rr}(y)}{|x-y|^{s}}+\lambda \frac{v_{\rr}(x)-v_{\rr}(y)}{|x-y|^{s}}\right)\dd\lambda, \quad x,y \in \er^n, 
 $$
that is
\eqn{elr}
$$
     \int_{\mathbb{R}^{n}}\int_{\mathbb{R}^{n}}\langle\sA_{\rr}(x,y)\frac{v_{\rr}(x)-v_{\rr}(y)}{|x-y|^s},\frac{\varphi(x)-\varphi(y)}{|x-y|^s}\rangle\frac{\dxy}{|x-y|^{n}}=0
$$
holds for every  $\varphi\in W^{s,2}(\er^n;\er^{N})$ with compact support in $\Omega_{x_0;1/\rr}$. Moreover, 
\eqn{ass.1}
$$ 
\left\{
\begin{array}{c}
\displaystyle 
 \sA_{\rr}(x,y)=\sA_{\rr}(y,x) \\[8pt]
\displaystyle 
\Lambda^{-1}\snr{\xi}^{2} \leq \langle \sA_{\rr}(x,y)\xi,\xi\rangle,\quad \snr{\sA_{\rr}(x,y)}\le \Lambda \\[8pt]\displaystyle
\ \snr{\sA_{\rr}(x,y)-\sA_{\rr,*}(x-y)}\le \Lambda \omega_{\rr}\left(\frac{\snr{v_{\rr}(x)-v_{\rr}(y)}}{|x-y|^{s}}\right)
\end{array}
\right.
$$ 
hold whenever $x,y\in \mathbb{R}^{n}$, $\xi\in \er^{N}$, where 
\eqn{modulo}
$$\omega_{\rr}(t):=\omega\left(\sE_{u}(\ell;x_{0},\rr)t \right)$$ for every $t \geq 0$. Moreover, $\tx E_{v_{\rr}}(0;0,1)=1$.  
\end{proposition}
\begin{proof} It is easy to see that 
$v_{\rr} \in W^{s,2}_{\loc}(\Omega_{x_0;1/\rr};\er^{N})\cap L^1_{2s}$; this follows by the definition of $u_\rr$ and from the fact that every affine map belongs to $W^{s,2}_{\loc}(\Omega_{x_0;1/\rr};\er^{N}) \cap L^{1}_{2s}$ provided $2s>1$, that is the basic assumption considered in this paper. The properties in \eqref{ass.1} are an easy consequence of \eqref{bs.1}, \eqref{sim} and \eqref{matrici01}-\eqref{matrici02}. To proceed with the proof, we write
\eqn{linearizza}
$$ 
  \sAA_{\rr}\left(\frac{u_{\rr}(x)-u_{\rr}(y)}{|x-y|^{s}}\right)-\sAA_{\rr}\left(\frac{\ell_{\rr}(x)-\ell_{\rr}(y)}{|x-y|^{s}}\right)
  =\sA_{\rr}(x,y)\frac{v_{\rr}(x)-v_{\rr}(y)}{|x-y|^{s}}.
$$ Using the first information in \eqref{sim}, we find
\begin{eqnarray}\label{0.1}
\sAA_{\rr}\left(\frac{\ell_{\rr}(x)-\ell_{\rr}(y)}{|x-y|^{s}}\right)&\stackrel{\eqref{basino}}{=}&\sAA_{\rr}\left(\frac{\ell_{\rr}(x)-\ell_{\rr}(y)}{|x-y|^{s}}\right)-\sAA_{\rr}(0)\nonumber \\ 
&=& \int_{0}^{1}\partial\sAA_{\rr}\left(\frac{\lambda(\ell_{\rr}(x)-\ell_{\rr}(y))}{|x-y|^{s}}\right)\dd\lambda\, \frac{\ell_{\rr}(x)-\ell_{\rr}(y)}{|x-y|^{s}} \nonumber \\
&\stackrel{\eqref{basino}}{=}&\int_{0}^{1}\partial\sAA \left(\frac{\lambda \rr\bbt(x-y)}{\rr^s|x-y|^{s}}\right)\dd\lambda\,  \frac{\ell_{\rr}(x)-\ell_{\rr}(y)}{|x-y|^{s}} \nonumber \\ 
&\stackrel{\eqref{nucleo}}{=}&\hat{\sA}_{\rr}(x-y) \frac{\ell_{\rr}(x)-\ell_{\rr}(y)}{|x-y|^{s}} \,.
\end{eqnarray} For every  $\varphi\in W^{s,2}(\er^n;\er^{N})$ with compact support in $\Omega_{x_0;1/\rr}$ we now have
\begin{eqnarray*}
0
&\stackrel{\eqref{riscalata}}{=}&\int_{\mathbb{R}^{n}}\int_{\mathbb{R}^{n}} \langle \sAA_{\rr}\left(\frac{u_{\rr}(x)-u_{\rr}(y)}{|x-y|^{s}}\right)-
  \sAA_{\rr}\left(\frac{\ell_{\rr}(x)-\ell_{\rr}(y)}{|x-y|^{s}}\right),\frac{\varphi(x)-\varphi(y)}{|x-y|^s}\rangle\frac{\dxy}{|x-y|^{n}}\nonumber \\
&&+ \int_{\mathbb{R}^{n}}\int_{\mathbb{R}^{n}}\langle \sAA_{\rr}\left(\frac{\ell_{\rr}(x)-\ell_{\rr}(y)}{|x-y|^{s}}\right),\frac{\varphi(x)-\varphi(y)}{|x-y|^s}\rangle\frac{\dxy}{|x-y|^{n}}\nonumber \\ 
&\stackrel{\eqref{linearizza},\eqref{0.1}}{=}&\int_{\mathbb{R}^{n}}\int_{\mathbb{R}^{n}} \langle \sA_{\rr}(x,y)\frac{v_{\rr}(x)-v_{\rr}(y)}{|x-y|^s},\frac{\varphi(x)-\varphi(y)}{|x-y|^s}\rangle\frac{\dxy}{|x-y|^{n}}\nonumber \\
&&+\int_{\mathbb{R}^{n}}\int_{\mathbb{R}^{n}}\langle \hat{\sA}_{\rr}(x-y)\frac{\ell_{\rr}(x)-\ell_{\rr}(y)}{|x-y|^s},\frac{\varphi(x)-\varphi(y)}{|x-y|^s}\rangle\frac{\dxy}{|x-y|^{n}} 
\end{eqnarray*}
recalling \eqref{0.1.} the last integral vanishes, which implies  \eqref{linearizzata} and the proof is complete. 
\end{proof}
\subsection{Basic properties of the linearized equation}
Continuing from Section \ref{blowsec}, here we derive  a few consequences of Proposition \ref{lineaprop}. These can mostly be inferred from the material in \cite{parte1} provided some modifications are made, and therefore we shall in the following describe how to adapt to   \eqref{elr} the methods developed in \cite[Sections 7.1.1, 7.1.2]{parte1}. We first state a couple of basic results. Lemma \ref{caccioppola} applies in our setting and gives that 
\eqn{2cacc}
$$
\gamma^{n}\snra{v_{\rr}}_{s,2;\gamma\BBB}^2\le \frac{c}{(1-\gamma)^{2(n+s)}\snr{\BBB}^{2s/n}}\nra{v_{\rr}}_{L^{2}(\BBB)}^2+\frac{c}{(1-\gamma)^{2(n+s)}\snr{\BBB}^{2s/n}}\tail(v_{\rr};\BBB) \nra{v_{\rr}}_{L^{1}(\BBB)} 
$$
holds with $c\equiv c(\data)$, whenever $\BBB \Subset \Omega_{x_0;1/\rr}$ and $0 < \gamma <1$, $\rad(\BBB)\leq 1$. A second result, which has actually already been used in the proof of Proposition \ref{gradpre}, is concerned with the higher differentiability results derived in \cite{KMS}, and claims that
\eqn{stimapp}
$$
\begin{cases}
\, v_{\rr}\in W^{s+\delta_0,2}_{\loc}(\Omega_{x_0;1/\rr};\er^{N}) \  \ \mbox{for some $\delta_0\equiv \delta_0(\data) \in (0, 1-s)$}\\[4pt]
\,  |\BBB|^{\frac{s+\delta_{0}}{n}} \snra{v_{\rr}}_{s+\delta_{0}, 2;\BBB/2} \leq c 
|\BBB|^{\frac{s}{n}} \snra{v_{\rr}}_{s, 2;\BBB}+c\tail(v_{\rr};2\BBB)
\end{cases}
$$
holds with $c \equiv c (\data)$ and for every ball $\BBB \Subset \Omega_{x_0;1/\rr}$; for this precise statement see \cite[Theorem 5.5]{parte1}. Observing that $\ttB_{1}\Subset \Omega_{x_0;1/\rr}$, a consequence of \eqref{matrix} and \eqref{2cacc}-\eqref{stimapp} is 
\eqn{richiama}
$$
\begin{cases}
\displaystyle
\nra{v_{\rr}}_{L^{2}(\ttB_{1})}^2+\tail(v_{\rr};\ttB_{1})^2=1\\[4pt]
[v_{\rr}]_{s,2;\ttB_{4/5} }\le c\equiv c(\data)\\[4pt]  [v_{\rr}]_{s+\delta_0,2;\ttB_{3/4} } \le c\equiv c(\data)\,. 
\end{cases}
$$
Indeed, \eqref{richiama}$_1$ follows by the very definition of $v_{\rr}$ in \eqref{defivrr} and in turn this implies \eqref{richiama}$_2$ via \eqref{2cacc}.  Once \eqref{richiama}$_2$  is achieved, \eqref{richiama}$_3$ follows by \eqref{stimapp}$_2$ and a standard covering argument as in Lemma \ref{covlemma}.  
\begin{lemma} 
Let 
\eqn{cutty}
$$
\ti{v}_{\rr}:=\eta v_{\rr}, \ 
\mbox{where $\eta \in C^{\infty}(\er^n)$, \  $\mathds{1}_{\ttB_{1/4} }\le \eta \le \mathds{1}_{\ttB_{1/2} }$ \ and \ $\nr{D\eta }_{L^{\infty}(\ttB_{1/2})}\lesssim 1$}.
$$
Then  
\begin{itemize} 
\item
$\ti{v}_{\rr}\in W^{s,2}(\er^n; \er^{N})\cap W^{s+\delta_0,2}(\er^n; \er^{N})$ with 
\eqn{cf.2}
$$
\begin{cases}
\, \nr{\ti{v}_{\rr}}_{W^{s,2}(\mathbb{R}^{n})}+\nr{\ti{v}_{\rr}}_{W^{s+\delta_0,2}(\mathbb{R}^{n})}\le c\equiv c(\data)\\
\, \tail(\ti{v}_{\rr};\ttB_{1})\leq 1
\end{cases}
$$
where $\delta_0\equiv \delta_0(\data)\in (0, 1-s)$ appears in \eqref{stimapp}.  The constant $c$ appearing in \eqref{cf.2} depends directly on the constant appearing in \eqref{richiama}. 
\item There exist $g_{\rr}\in L^{\infty}(\mathbb{R}^{n};\er^{N})$ such that $\ti{v}_{\rr}$ solves
$$
-\mathcal{L}_{\scalebox{0.6}{$\sA_\rr$}}\ti{v}_{\rr}=g_{\rr}\qquad \mbox{in} \ \ \ttB_{1/8}\,,
$$
that is 
\eqn{cutoffata}
$$
\int_{\er^{n}} \int_{\er^{n}} \langle \sA_{\rr}(x,y)\frac{\ti{v}_{\rr}(x)-\ti{v}_{\rr}(y)}{{|x-y|^s}},\frac{\varphi(x)-\varphi(y)}{|x-y|^s} \rangle \frac{\dxy}{|x-y|^{n}}= \int_{\ttB_{1/8}} \langle g_{\rr}, \varphi \rangle\dx
$$
holds for every $\varphi \in W^{s,2}(\er^n;\er^{N})$ with compact support in $\ttB_{1/8}$. 
\item Specifically, 
\eqn{nuovag}
$$
 g_{\rr}(x) =\mathds{1}_{\ttB_{1/8}}(x)2\int_{\er^n\setminus \ttB_{1/4}}(1-\eta(y))\sA_{\rr}(x, y)v_{\rr}(y)\frac{\dy}{|x-y|^{n+2s}}
$$
so that 
\eqn{cf.1}
$$ 
\nr{g_{\rr}}_{L^{\infty}(\mathbb{R}^{n})}\le c\nr{v_{\rr}}_{L^1(\ttB_{1/2})} + c \tail(v_{\rr}; \ttB_{1/4})\leq c \equiv c(\data)
$$
and, taking \eqref{cf.2} into account, all in all it is 
\eqn{lastella}
$$
\nr{\ti{v}_\rr}_{W^{s,2}(\mathbb{R}^{n})}+\nr{\ti{v}_\rr}_{W^{s+\delta_{0},2}(\er^n)}+ \tail(\ti{v}_\rr;\ttB_1) + \nr{g_\rr}_{L^{\infty}(\mathbb{R}^{n})}\le c_*\equiv c_*(\data).
$$
\end{itemize}
\end{lemma}
\begin{proof} This is essentially \cite[Lemma 7.2]{parte1} when applied to the present situation. We report the proof in some detail as it entails some simplification, which in fact also applies to the original \cite[Lemma 7.2]{parte1}.  Estimate \rif{cf.2} follows exactly as \cite[Lemma 7.2]{parte1}. As for \rif{cutoffata}-\eqref{nuovag}, using that $\varphi$ is supported in $\ttB_{1/8}\subset\ttB_{1/4}$ and that
$\sA_\rr(x,y)=\sA_\rr(y,x)$, we decompose
\begin{flalign*} 
&\hspace{-4mm} \int_{\mathbb{R}^{n}}\int_{\mathbb{R}^{n}}\langle \sA_{\rr}(x,y)\frac{\ti{v}_{\rr}(x)-\ti{v}_{\rr}(y)}{\snr{x-y}^{s}},\frac{\varphi(x)-\varphi(y)}{\snr{x-y}^{s}}\rangle\frac{\dxy}{\snr{x-y}^{n}}\nonumber \\
&=\int_{\ttB_{1/4}}\int_{\ttB_{1/4}}\langle \sA_{\rr}(x,y)\frac{\ti{v}_{\rr}(x)-\ti{v}_{\rr}(y)}{\snr{x-y}^{s}},\frac{\varphi(x)-\varphi(y)}{\snr{x-y}^{s}}\rangle\frac{\dxy}{\snr{x-y}^{n}}\nonumber \\
&\quad +2\int_{\mathbb{R}^{n}\setminus \ttB_{1/4}}\int_{\ttB_{1/4}}\langle \sA_{\rr}(x,y)\frac{\ti{v}_{\rr}(x)-\ti{v}_{\rr}(y)}{\snr{x-y}^{s}},\frac{\varphi(x)-\varphi(y)}{\snr{x-y}^{s}}\rangle\frac{\dxy}{\snr{x-y}^{n}}\nonumber \\
&=\int_{\ttB_{1/4}}\int_{\ttB_{1/4}}\langle \sA_{\rr}(x,y)\frac{v_{\rr}(x)-v_{\rr}(y)}{\snr{x-y}^{s}},\frac{\varphi(x)-\varphi(y)}{\snr{x-y}^{s}}\rangle\frac{\dxy}{\snr{x-y}^{n}}\nonumber \\
&\quad +2\int_{\mathbb{R}^{n}\setminus \ttB_{1/4}}\int_{\ttB_{1/4}}\langle \sA_{\rr}(x,y)\frac{v_{\rr}(x)-v_{\rr}(y)}{\snr{x-y}^{s}},\frac{\varphi(x)-\varphi(y)}{\snr{x-y}^{s}}\rangle\frac{\dxy}{\snr{x-y}^{n}}\nonumber \\
&\quad +2\int_{\mathbb{R}^{n}\setminus \ttB_{1/4}}\int_{\ttB_{1/4}}\langle \sA_{\rr}(x,y)((\ti{v}_{\rr}(x)-v_{\rr}(x))-(\ti{v}_{\rr}(y)-v_{\rr}(y)),\varphi(x)\rangle\frac{\dxy}{\snr{x-y}^{n+2s}} \\
&=2\int_{\ttB_{1/4}}\int_{\er^n\setminus \ttB_{1/4}}\langle \sA_{\rr}(x,y)(v_{\rr}(y)-\ti{v}_{\rr}(y)),\varphi(x)\rangle\frac{\dyx}{\snr{x-y}^{n+2s}}\\
&=  \int_{\ttB_{1/8}} \langle g_{\rr}, \varphi \rangle\dx
\end{flalign*}
Note that the adaptation of \cite[Lemma 7.2]{parte1} relies on the fact that \eqref{stimapp} and \eqref{richiama} are now available (these are the analogs of \cite[(7.13)]{parte1}). 
\end{proof}

\subsection{Flatness improvement}

The proof of Proposition \ref{cor.1} uses a number of preparatory results. The inequalities in \eqref{richiama}$_{2,3}$, and in particular the determination of the exponent $\delta_0\equiv \delta_0(\data)$ made in \eqref{stimapp}, allow to determine the new higher differentiability and integrability exponents, $t$ and $p$ respectively, coming from \eqref{tpnumeri} of Lemma \ref{shar}
\eqn{tep}
$$
s< t \equiv t(\data, \delta_0)\equiv t(\data)< 1\, \quad \mbox{and} \quad p\equiv  p(\data, \delta_0)\equiv p(\data)>2
$$
as quantities globally depending only on $\data$. These quantities appear in the statement and in the proof of the next 
\begin{lemma}\label{notte}
Assume that 
\eqn{small0}
$$
\sE_{u}(\ell;x_{0},\rr)<\eps_{0}
$$
holds for some $\eps_{0} \in (0,1)$. Then 
\begin{flalign}
\notag & \left|\int_{\mathbb{R}^{n}}\int_{\mathbb{R}^{n}}\langle\sA_{\rr,*}(x-y)\frac{\ti{v}_{\rr}(x)-\ti{v}_{\rr}(y)}{|x-y|^s},\frac{\varphi(x)-\varphi(y)}{|x-y|^s}\rangle\frac{\dxy}{|x-y|^{n}}-\int_{\ttB_{1/16}}\langle g_{\rr},\varphi\rangle\dx\right|\\
& \quad  \le \ti{c}[\omega(\eps_{0})]^{\frac{t-s}{2(n+s-t)}}\label{aha}
\end{flalign}
holds whenever  $\varphi \in \mathbb{X}^{s,2}_{0}(\ttB_{1/16},\ttB_{1/8})\cap C^{0,t}(\er^n;\er^{N})$ with  $[\varphi]_{0,t;\er^n}\leq 1$, 
where $ \ti{c}\equiv  \ti{c}(\data)$. 
\end{lemma}
\begin{proof} In the following estimate, and in the rest of the proof, we shall use the basic properties of $\omega(\cdot)$  (it is concave, non-decreasing, $\omega(0)=0$, and therefore sublinear, i.e., $\omega(a+b)\leq \omega(a)+\omega(b)$, $\omega(\cdot)\leq 1$, $\omega(ct)\leq c  \omega(t)$ for $c\geq 1$) so that
\begin{flalign*}
\notag \int_{\ttB_{1/8}}\int_{\ttB_{1/8}}\omega_{\rr}\left(\frac{\snr{v_{\rr}(x)-v_{\rr}(y)}}{|x-y|^{s}}\right)^{\frac{2(n+s-t)}{t-s}}\dxy & \leq c\mint_{\ttB_{1/8}}\mint_{\ttB_{1/8}}\omega_{\rr}\left(\frac{\snr{v_{\rr}(x)}+\snr{v_{\rr}(y)}}{|x-y|^{s}}\right)\dxy\\ \notag & \leq c\mint_{\ttB_{1/8}}\mint_{\ttB_{1/8}}\omega_{\rr}\left(\frac{\snr{v_{\rr}(x)}}{|x-y|^{s}}\right)\dxy\\
& \notag \leq c\omega_{\rr}\left(\mint_{\ttB_{1/8}}\mint_{\ttB_{1/8}}\frac{\snr{v_{\rr}(x)}}{|x-y|^{s}}\dxy\right)\\
& \leq c\omega_{\rr}\left(c\int_{\ttB_{1}}\frac{\dz}{\snr{z}^{s}}\nr{v_\rr}_{L^1(\ttB_1)}\right) \notag \\
& \leq c \omega_{\rr}\left(\nra{v_{\rr}}_{L^2(\ttB_1)}\right) \notag \\
&\hspace{-2mm}\stackleq{richiama}  c\omega_{\rr} (1)\stackrel{\eqref{modulo}}{=}  c \omega(\sE_{u}(\ell;x_{0},\rr) )
\end{flalign*}
where we also used Jensen inequality. Note that we have used that since $\omega_\rr\le 1$ and $(2(n+s-t))/(t-s)>1$, then 
$\omega_\rr^{{2(n+s-t)}/{(t-s)}}\le \omega_\rr$.
Using the content of the last display in conjunction with \eqref{small0} we conclude with 
\eqn{ommega}
$$ 
\int_{\ttB_{1/8}}\int_{\ttB_{1/8}}\omega_{\rr}\left(\frac{\snr{v_{\rr}(x)-v_{\rr}(y)}}{|x-y|^{s}}\right)^{\frac{2(n+s-t)}{t-s}}\dxy \leq c\omega(\eps_{0}) 
$$
where $c$ only depends on $n$ and $s$. 
 With $\varphi$ as in the statement of the lemma\footnote{Note that $\varphi$ as considered in the statement of Lemma \ref{notte} are admissible in \eqref{cutoffata} as they can be extended to $W^{s,2}(\er^n;\er^N)$ following the discussion before \eqref{ht.10}.}, using \eqref{cutoffata} we split
\begin{flalign*}
\mbox{(I)}&:=\left|\int_{\mathbb{R}^{n}}\int_{\mathbb{R}^{n}}\langle\sA_{\rr,*}(x-y)\frac{\ti{v}_{\rr}(x)-\ti{v}_{\rr}(y)}{|x-y|^s},\frac{\varphi(x)-\varphi(y)}{|x-y|^s}\rangle\frac{\dxy}{|x-y|^{n}}-\int_{\ttB_{1/16}}\langle g_{\rr},\varphi\rangle\dx\right|\nonumber \\
&=\left|\int_{\mathbb{R}^{n}}\int_{\mathbb{R}^{n}}\langle(\sA_{\rr,*}(x-y)- \sA_{\rr}(x,y))\frac{\ti{v}_{\rr}(x)-\ti{v}_{\rr}(y)}{{|x-y|^s}},\frac{\varphi(x)-\varphi(y)}{|x-y|^s}\rangle\frac{\dxy}{|x-y|^{n}}\right|\nonumber \\
&\leq \int_{\ttB_{1/8}}\int_{\ttB_{1/8}}\snr{\sA_{\rr,*}(x-y)-\sA_{\rr}(x,y)}\frac{\snr{\ti{v}_{\rr}(x)-\ti{v}_{\rr}(y)}}{|x-y|^s}\frac{\snr{\varphi(x)-\varphi(y)}}{|x-y|^s}\frac{\dxy}{|x-y|^{n}}\nonumber \\
&\quad  +2\int_{\mathbb{R}^{n}\setminus \ttB_{1/8}}\int_{\ttB_{1/8}}\snr{\sA_{\rr,*}(x-y)-\sA_{\rr}(x,y)}\snr{\ti{v}_{\rr}(x)-\ti{v}_{\rr}(y)}\snr{\varphi(x)}\frac{\dxy}{|x-y|^{n+2s}}\\
&=:\mbox{(II)}+\mbox{(III)}.
\end{flalign*} 
Note that we have used the symmetry properties in \eqref{sim0} and \eqref{ass.1}$_1$. 
We continue estimating via H\"older and Jensen inequalities as follows:
\begin{eqnarray*}
\mbox{(II)}&\stackrel{\eqref{ass.1}_{3}}{\le}&c [\varphi]_{0,t;\er^n}\int_{\ttB_{1/8}}\int_{\ttB_{1/8}}\omega_{\rr}\left(\frac{\snr{v_{\rr}(x)-v_{\rr}(y)}}{|x-y|^{s}}\right)\frac{\snr{\ti{v}_{\rr}(x)-\ti{v}_{\rr}(y)}}{|x-y|^{n/2+s}}\frac{\dxy}{|x-y|^{n/2+s-t}}\nonumber \\
&\notag\leq &c[\ti{v}_{\rr}]_{s,2;\ttB_{1/8}}\left(\int_{\ttB_{1/8}}\int_{\ttB_{1/8}}\omega_{\rr}\left(\frac{\snr{v_{\rr}(x)-v_{\rr}(y)}}{|x-y|^{s}}\right)^{\frac{2(n+s-t)}{t-s}}\dxy\right)^{\frac{t-s}{2(n+s-t)}}\nonumber \\
&&\notag\quad \times\left(\int_{\ttB_{1/8}}\int_{\ttB_{1/8}}\frac{\dxy}{|x-y|^{n+s-t}}\right)^{\frac{n-2(t-s)}{2(n+s-t)}}\nonumber \\
&\stackrel{\eqref{cf.2}_1}{\le} &c\left(\frac{1}{t-s}\right)^{\frac{n-2(t-s)}{2(n+s-t)}}\left(\int_{\ttB_{1/8}}\int_{\ttB_{1/8}}\omega_{\rr}\left(\frac{\snr{v_{\rr}(x)-v_{\rr}(y)}}{|x-y|^{s}}\right)^{\frac{2(n+s-t)}{t-s}}\dxy\right)^{\frac{t-s}{2(n+s-t)}}\\
&\stackrel{\eqref{ommega}}{\le}& c \omega(\eps_{0})^{\frac{t-s}{2(n+s-t)}}
\nonumber 
\end{eqnarray*}
where $c$ depends on $\data$. Note that here we have used H\"older inequality with conjugate exponents 
$$\left(\frac{2(n+s-t)}{t-s},2,\frac{2(n+s-t)}{n-2(t-s)}\right).$$ 
In order to estimate the remaining term $\mbox{(III)}$ we note that 
$$
x\in \ttB_{1/16}, \ y\in \mathbb{R}^{n}\setminus \ttB_{1/8}\ \Longrightarrow \ \frac{\snr{y}}{|x-y|}\le 2,
$$ 
so, recalling that $\ti{v}_{\rr}\equiv 0$ outside $\ttB_{1/2}$ and that $\varphi\equiv 0$ outside $\ttB_{1/16}$\footnote{Since $\varphi\equiv 0$ outside $\ttB_{1/16}$ and $[\varphi]_{0,t;\R^n}\le 1$, we also have $\|\varphi\|_{L^\infty(\R^n)}\le c(n,t)$.}, we control

\begin{flalign*} 
\mbox{(III)}&\leq c \int_{\mathbb{R}^{n}\setminus \ttB_{1/8}}\int_{\ttB_{1/16}}\omega_{\rr}\left(\frac{\snr{v_{\rr}(x)}+\snr{v_{\rr}(y)}}{\snr{y}^{s}}\right)(\snr{\ti{v}_{\rr}(x)}+\snr{\ti{v}_{\rr}(y)})\frac{\dxy}{\snr{y}^{n+2s}}\nonumber \\
&\leq c \int_{\mathbb{R}^{n}\setminus \ttB_{1/8}}\int_{\ttB_{1/16}}\left(\omega_{\rr}\left(\snr{v_{\rr}(x)}+\snr{v_{\rr}(y)}\right)\right)(\snr{\ti{v}_{\rr}(x)}+\snr{\ti{v}_{\rr}(y)})\frac{\dxy}{\snr{y}^{n+2s}}\nonumber \\
&\leq c \int_{\er^n\setminus \ttB_{1/8}}\int_{\ttB_{1/16}}\left(\omega_{\rr}\left(\snr{v_{\rr}(x)}\right)+\omega_{\rr}\left(\snr{v_{\rr}(y)}\right)\right)(\snr{\ti{v}_{\rr}(x)}+\snr{\ti{v}_{\rr}(y)})\frac{\dxy}{\snr{y}^{n+2s}}\nonumber \\
&=c  \int_{\er^n\setminus \ttB_{1/8}}\int_{\ttB_{1/16}}\frac{\omega_{\rr}(\snr{v_{\rr}(x)})\snr{\ti{v}_{\rr}(x)}}{\snr{y}^{n+2s}}\dxy+c \int_{\er^n\setminus \ttB_{1/8}}\int_{\ttB_{1/16}}\frac{\omega_{\rr}(\snr{v_{\rr}(y)})\snr{\ti{v}_{\rr}(x)}}{\snr{y}^{n+2s}}\dxy\nonumber \\
&\quad +c \int_{\ttB_{1/2}\setminus \ttB_{1/8}}\int_{\ttB_{1/16}}\frac{\omega_{\rr}(\snr{v_{\rr}(x)})\snr{\ti{v}_{\rr}(y)}}{\snr{y}^{n+2s}}\dxy+c \int_{\ttB_{1/2}\setminus \ttB_{1/8}}\int_{\ttB_{1/16}}\frac{\omega_{\rr}(\snr{v_{\rr}(y)})\snr{\ti{v}_{\rr}(y)}}{\snr{y}^{n+2s}}\dxy\nonumber \\
&=:c  \sum_{k=1}^{ 4} \mbox{(III)}_k\,.
\end{flalign*}
Recalling \eqref{richiama} and arguing as in \eqref{ommega}, we get
\begin{flalign*}
\mbox{(III)}_1+\mbox{(III)}_4 &\leq c \int_{\ttB_{1}} \omega_{\rr}(\snr{v_{\rr}})\snr{\ti{v}_{\rr}}\dx\\ & \leq c 
\nr{v_{\rr}}_{L^2(\ttB_1)}\left(\mint_{\ttB_{1}}\omega_{\rr}(\snr{v_{\rr}})\dx\right)^{1/2}\\
& \leq c \nr{v_{\rr}}_{L^2(\ttB_1)}\omega_{\rr}\left(\nra{v_{\rr}}_{L^2(\ttB_1)}\right)^{1/2}\\
& \leq c \sqrt{\omega_{\rr}\left(1\right)}\leq c \sqrt{\omega\left(\eps_{0}\right)}.
\end{flalign*}
Using Jensen inequality with respect to the measure $\snr{y}^{-n-2s}\dy$ and again recalling  \eqref{scatailancora2} and  \eqref{richiama}, we find 
\begin{flalign*}
\mbox{(III)}_2 &\leq c \nr{v_{\rr}}_{L^1(\ttB_1)} \int_{\er^n\setminus \ttB_{1/8}}\frac{\omega_{\rr}(\snr{v_{\rr}(y)})}{\snr{y}^{n+2s}}\dy\\ & 
\leq c  \nr{v_{\rr}}_{L^1(\ttB_1)} \omega_{\rr}\left(c \tail(v_{\rr};\ttB_{1/8})\right)\\
& \leq c \nr{v_{\rr}}_{L^2(\ttB_1)} \omega_{\rr}\left( c\nra{v_{\rr}}_{L^1(\ttB_1)}+c\tail(v_{\rr};\ttB_{1})\right) 
\\ & \leq c\omega_{\rr}(1) \leq c 
\omega(\eps_{0})
\end{flalign*}
and 
\begin{flalign*}
\mbox{(III)}_3 & \leq c \nr{v_{\rr}}_{L^1(\ttB_1)} \mint_{\ttB_{1}}\omega_{\rr}(\snr{v_{\rr}})\dx \\
& \leq c
\nr{v_{\rr}}_{L^2(\ttB_1)}\omega_{\rr}\left(\nra{v_{\rr}}_{L^2(\ttB_1)}\right) \leq c 
\omega(\eps_{0})\,.
\end{flalign*}
All the constants mentioned above depend on $\data$. 
Collecting all the previous estimates finally leads to \eqref{aha}. 
\end{proof} 
 With $\gamma \geq 0$, in the following we shall need the Campanato type excess
 \eqn{eccessone}
  $$\mathcal{E}_{\gamma}(x_{0},\rr)\equiv \mathcal{E}_{u,\gamma}(x_{0},\rr):=\rr^{-\gamma}\inf_{\elly \textnormal{ affine}} \tx{E}_{u}(\elly;x_0,\rr)  \geq 0$$
  that in fact we already employed in the proof of Theorem \ref{cdg}. 
\begin{lemma}\label{piccolina} Given $\alpha \in (0,2s-1)$, there exist $\epsb{b}\equiv \epsb{b}( \data, \omega(\cdot), \alpha) \in (0,1)$ and $\tet\equiv \tet( \data,\alpha) \in (0,1/32)$, such that if
\eqn{smallee}
$$
\mathcal{E}_{s}(x_{0},\rr) < \epsb{b}\,,
$$
then 
\eqn{padr}
$$
 \mathcal{E}_{1+\alpha}(x_0,\tet \rr) \leq    \mathcal{E}_{1+\alpha}(x_{0},\rr) \quad \mbox{and}\quad \mathcal{E}_{s}(x_{0},\tet\rr) < \epsb{b} .
$$
\end{lemma}
\begin{proof} The main point here is to apply Lemma \ref{shar} to $\ti{v}_{\rr}$.  
With $\alpha$ as in the statement and with $\eps>0$ to be determined later, we start by determining the threshold $\epsb{b}>0$  needed to formulate \eqref{smallee} as a function of $\data, \omega(\cdot), \eps$. We will then make a suitable choice of $\eps$ reflected in the final dependence of the various constants displayed in the statement of the Lemma. The choice of $\epsb{b}$ will be finalized later as a function of $\data, \omega(\cdot), \alpha$, as prescribed in the statement towards the end of the proof, when we shall determine $\eps$ and $\tet $ (see \eqref{dja2} below). Specifically, we choose $\epsb{b}$ such that 
\eqn{sceltina}
$$
 \ti{c}[\omega(\epsb{b})]^{\frac{t-s}{2(n+s-t)}}\leq \left(\frac{\eps}{\cchh}\right)^{\frac{2p}{p-2}}\Longrightarrow   \epsb{b}\equiv \epsb{b}(\eps, \data, \omega(\cdot)).
$$ 
In the inequality in \eqref{sceltina}, the left-hand side appears in Lemma \ref{notte}, \eqref{aha}, and the right-hand side comes from Lemma \ref{shar}, \eqref{sh.1}. These can be applied thanks to \eqref{cf.2}; the numbers $t,p$ are taken from \eqref{tep}. For the reader’s convenience we summarize the choice and the quantitative tracking of the constants as follows:
\begin{itemize}
\item \mbox{Determine $\delta_0\equiv \delta_0(\data)\in (0, 1-s)$ as in \eqref{stimapp}, via fractional Gehring's lemma \cite{KMS}}.
\item \mbox{Use the constant $c_{*}\equiv c_{*}(\data)$ from \eqref{lastella} to determine $c_0\equiv c_0(\data):=c_{*}$ in \eqref{sh.0}}. 
\item Use  $\delta_0, c_0$ and $\texttt{d}=1/16$ to determine the constant $\cchh\equiv \cchh(\data)$ appearing in \eqref{chconst}.
\item Use $\delta_0$ to determine $t,p\equiv t,p(\data)$ as from Lemma \ref{shar} (done in \eqref{tep}). 
\item Use $\tilde c\equiv \tilde c (\data) $ from Lemma \ref{notte}, \eqref{aha}, to finally determine $\epsb{b}$ via \eqref{sceltina}. 
\end{itemize}
By the very definition in \eqref{eccessone}, \eqref{smallee} implies that for every $\delta>0$ such that $\mathcal{E}_{s}(x_{0},\rr) +\delta < \epsb{b}$  there exists an affine map $\ell$ such that
\eqn{infy}
$$
\sE_{u}(\ell;x_{0},\rr)<\mathcal{E}_{s}(x_{0},\rr) +\delta <\epsb{b} \,.$$
Applying Lemma \ref{notte} with $\eps_0=\epsb{b}$, and using \eqref{sceltina} in \eqref{aha} we obtain that 
$$
 \left|\int_{\mathbb{R}^{n}}\int_{\mathbb{R}^{n}}\langle\sA_{\rr,*}(x-y)\frac{\ti{v}_{\rr}(x)-\ti{v}_{\rr}(y)}{|x-y|^s},\frac{\varphi(x)-\varphi(y)}{|x-y|^s}\rangle\frac{\dxy}{|x-y|^{n}}-\int_{\ttB_{1/16}}\langle g_{\rr},\varphi\rangle\dx\right| \le \left(\frac{\eps}{\cchh}\right)^{\frac{2p}{p-2}}
$$
holds whenever $\varphi \in \mathbb{X}^{s,2}_{0}(\ttB_{1/16},\ttB_{1/8})\cap C^{0,t}(\er^n;\er^{N})$ with  $[\varphi]_{0,t;\er^n}\leq 1$. 
Therefore  
Lemma \ref{shar} applies as follows\footnote{The tail term required in \eqref{sh.0} with $\ti B=\ttB_{1/8}$ is controlled by
\eqref{cf.2} and \eqref{scatailancora2}.}: if $h\in \mathbb{X}^{s,2}_{\ti{v}_{\rr}}(\ttB_{1/16},\ttB_{1/8})$ solves (as usual, existence of $h$ comes from Lemma \ref{esiste})
$$
\begin{cases}
\ -\mathcal{L}_{\scalebox{0.6}{$\sA_{\rr,*}$}}h=g_{\rr}\quad &\mbox{in} \ \ \ttB_{1/16}\\
\ h=\ti{v}_{\rr}\quad &\mbox{in} \ \ \mathbb{R}^{n}\setminus \ttB_{1/16}, 
\end{cases}
$$
then \eqref{sh.5} yields
\eqn{confronti}
$$
\begin{cases}
\, \nr{h}_{W^{s,2}(\ttB_{1/8})}+\tail(h;\ttB_{1/16})\le c\equiv c (\data) \\[4pt]
 \nr{\ti{v}_{\rr}-h}_{L^2(\ttB_{1/16})}=\nr{v_{\rr}-h}_{L^2(\ttB_{1/16})}\leq c(\data)\eps\,.
\end{cases}
$$ 
The first line follows from \eqref{sh.5}, since the estimate of
$\tail(h-(h)_{\ttB_{1/16}};\ttB_{1/16})$, together with
$\|h\|_{L^2(\ttB_{1/8})}\le c$, also controls $\tail(h;\ttB_{1/16})$.
Applying \eqref{dg.2bis} with 
\eqn{ilbeta}
$$
\beta\equiv \beta(s, \alpha):= \frac{2s-1+\alpha}{2} \in (\alpha, 2s-1) 
$$
and then also using \rif{cf.1} and \eqref{confronti}, we find
\eqn{cabi}
$$  
\nr{Dh}_{L^\infty(\ttB_{1/32})} +[Dh]_{0,\beta; \ttB_{1/32}} \leq c  \nra{h}_{L^2(\ttB_{1/16})}+c\tail(h; \ttB_{1/16})+c \leq c(\data,\alpha).
$$
Defining the affine map $\ell_{h}$ as 
$\ell_{h}(x):=D h(0)x + h(0)$, and using Mean Value Theorem together with \rif{cabi}, we find that 
\eqn{cabi2} 
$$
\begin{cases}
 \nr{h-\ell_{h}}_{L^\infty(\ttB_{\lambda}) } \leq c(\data, \alpha)  \lambda ^{1+\beta} \\[6pt]
   \snr{D  h(0)}  + \snr{ h(0)} \leq c(\data,\alpha)
   \end{cases} $$
 holds for every $\lambda \in (0,1/32)$ (the dependence on $\beta$ has been incorporated in the dependence on $\alpha$ according to \eqref{ilbeta}).  The bound for $h(0)$ follows from the $L^2$ estimate for $h$ in \eqref{confronti} and the $L^\infty$ estimate for $Dh$ in \eqref{cabi}. 
  Note that (see also \cite[(3.24)]{parte1})
    \eqn{cabi4}
  $$
\begin{cases}
  \nra{\ell_{h}}_{L^2(\ttB_{\lambda})} \approx_{n,N}  \snr{D  h(0)} \lambda + \snr{h(0)} \leq  c (\data,\alpha)(\lambda +1) \\[6pt]
\tail(\ell_{h};\ttB_{\lambda}) \lesssim_{n,N} \frac{   \snr{D  h(0)}\lambda }{2s-1} + \frac{\snr{ h(0)}}{2s}\lesssim_{n,N,s}\snr{D  h(0)} \lambda  +\snr{ h(0)} \leq c (\data,\alpha)(\lambda +1)
\end{cases}
$$
holds for every $\lambda>  0$. In the second inequality we have clearly used that $s>1/2$. Combining \rif{richiama} and \rif{cabi4}, and recalling \rif{scatailancora2}, we obtain 
\eqn{cabi44}
$$
 \nra{v_{\rr}- \ell_{h}}_{L^2(\ttB_{\lambda}) }+\tail(v_{\rr}- \ell_{h};\ttB_{\lambda}) \leq c (\data,\alpha)\lambda^{-n}
$$
for every $\lambda \in (0,1]$. 
Using \rif{confronti}$_2$, \rif{cabi2}$_1$ and that $v_{\rr}\equiv \ti{v}_{\rr}$ on $\ttB_{1/4}$, see \eqref{cutty}, we have 
 \begin{flalign*}
	\tx{E}_{v_{\rr}}(\ell_{h};0,\tet )&\leq \nra{v_{\rr}- \ell_{h}}_{L^{2}(\ttB_{\tet })} + \tail(v_{\rr}- \ell_{h};\ttB_{\tet })\\	
	&\leq c \nra{h-\ell_{h}}_{L^2(\ttB_{\tet }) } + c \tet^{-n/2} \nra{v_{\rr}- h}_{L^2(\ttB_{1/16}) } + \tail(v_{\rr}- \ell_{h};\ttB_{\tet })\nonumber \\
	&\leq c \tet^{1+\beta} +c \tet ^{-n/2} \eps + c\tail(v_{\rr}- \ell_{h};\ttB_{\tet }).
	 \end{flalign*} 
We estimate the last $\tail$ term above as follows:
 \begin{eqnarray*}
 \tail(v_{\rr}- \ell_{h};\ttB_{\tet }) & \stackleq{scatailancora} &c\tet^{2s} \tail(v_{\rr}- \ell_{h};\ttB_{1/32}) +  
 \tet^{2s}  \nra{v_{\rr}- \ell_{h}}_{L^2(\ttB_{1/32}) }\\
 & &\quad + c\int_{\tet}^{1/32}\left(\frac{\tet}{\lambda}\right)^{2s}\nra{v_{\rr}- \ell_{h}}_{L^1(B_{\lambda})}\frac{\dlam}{\lambda}\\
 &  \stackleq{cabi44} & c \tet^{2s} + c\tet^{2s}\int_{\tet}^{1/32}\left(\frac{1}{\lambda}\right)^{2s+n/2}\frac{\dlam}{\lambda}\, \nra{v_{\rr}- h}_{L^2(B_{1/16})}\\
  & & \quad +c\tet^{2s}\int_{\tet}^{1/32}\left(\frac{1}{\lambda}\right)^{2s}\nra{h- \ell_{h}}_{L^2(B_{\lambda})}\frac{\dlam}{\lambda}\\
 & \stackrel{\eqref{confronti}_2, \eqref{cabi2}}{\leq} & c \tet^{2s} + c \eps\tet^{-n/2} + c \tet^{2s}\int_{\tet}^{1/32}\left(\frac{1}{\lambda}\right)^{2s-1-\beta}\frac{\dlam}{\lambda}. 
  \end{eqnarray*}	 
Matching the content of the last two displays, at recalling $1+\beta <2s$, we arrive at
\eqn{dja0}
$$
	\tx{E}_{v_{\rr}}(\ell_{h};0,\tet )\leq c_{1}\tet ^{1+\beta} +c_{2}\tet^{-n/2}\eps 
$$
where $c_{1},c_{2}$ depend on $\data$ and $\alpha$. Determine  $\tet \in (0, 1/32)$ such that 
\eqn{dja1}
$$c_{1}\tet ^{\beta-\alpha}< \frac 12 \Longrightarrow \tet \equiv \tet (\data, \alpha)$$ and then choose \eqn{dja2} 
$$
\eps := \frac{\tet ^{n/2+1+\alpha}}{2c_{2}+1} \Longrightarrow \eps\equiv \eps(\data, \alpha).
$$
Note that this choice finally fixes $\eps$ and therefore finally fixes $\epsb{b}$ as a function of $\data, \omega(\cdot)$ and $\alpha$ via \eqref{sceltina} as described at the beginning of the proof. 
Using \eqref{dja1}-\eqref{dja2} in \eqref{dja0} allows us to conclude with 
\eqn{tran}
$$
\tx{E}_{v_{\rr}}(\ell_{h};0,\tet ) < \tet ^{1+\alpha}.
$$
With $\tilde{\ell}$ being the affine map defined by 
$$
\tilde{\ell}(x):=\ell(x)+\tx{E}_{u}(\ell;x_{0},\rr)\,\ell_{h}\left(\frac{x-x_0}{\rr}\right)
$$
changing variables gives 
$$
\tx{E}_{v_{\rr}}(\ell_{h};0,\tet )= \frac{\tx{E}_{u}(\tilde{\ell};x_0,\tet \rr)}{\tx{E}_{u}(\ell;x_{0},\rr)}
$$ 
so that \eqref{tran} turns into
$$  
 \frac{\tx{E}_{u}(\tilde{\ell};x_0,\tet \rr)}{(\tet \rr)^{1+\alpha}} <   \frac{\tx{E}_{u}(\ell;x_{0},\rr)}{\rr^{1+\alpha}}
$$ 
and therefore, recalling the definition in \eqref{eccessone}, we have
$$
 \mathcal{E}_{1+\alpha}(x_0,\tet \rr) <   \frac{\tx{E}_{u}(\ell;x_{0},\rr)}{\rr^{1+\alpha}}.
$$
By using \eqref{infy} and again the definition in \eqref{eccessone}, we find
$$
 \frac{\tx{E}_{u}(\ell;x_{0},\rr)}{\rr^{1+\alpha}}<\frac 1{\rr^{1+\alpha-s}}\left(\mathcal{E}_{s}(x_{0},\rr)+\delta\right) = 
 \mathcal{E}_{1+\alpha}(x_{0},\rr)+\frac{\delta}{\rr^{1+\alpha-s}} 
$$
so that the first inequality in \eqref{padr} follows merging the content of the last two displays and letting $\delta \to 0$. This allows to prove the second inequality in \eqref{padr} as follows:
\begin{flalign*}
\mathcal{E}_{s}(x_{0},\tet \rr) &= (\tet \rr)^{1+\alpha-s} \mathcal{E}_{1+\alpha}(x_{0},\tet \rr)
\\ &\leq  (\tet \rr)^{1+\alpha-s} \mathcal{E}_{1+\alpha}(x_{0},\rr) \\
& = 
\tet ^{1+\alpha-s} \mathcal{E}_{s}(x_{0},\rr) <\tet ^{1+\alpha-s}\epsb{b} <\epsb{b}
\end{flalign*}
and the proof is complete. 
\end{proof}
\subsection{Proof of Proposition \ref{cor.1} completed.}\label{completa} We fix a positive $\alpha<2s-1$ and determine 
$\epsb{b}\equiv\epsb{b}(\data,\omega(\cdot),\alpha)$ and 
$\tet\equiv\tet(\data,\alpha)$ as prescribed in Lemma \ref{piccolina}. We assume that \eqref{small} is satisfied with such an $\epsb{b}$; letting $\rr_j:= \tet ^j\rr$ for every $j\geq 0$,  Lemma \ref{piccolina} applies and gives, by induction 
$$
 \mathcal{E}_{1+\alpha}(x_0, \rr_{j+1}) \leq    \mathcal{E}_{1+\alpha}(x_{0},\rr_j) \quad \mbox{and}\quad \mathcal{E}_{s}(x_{0},\rr_j) < \epsb{b} 
$$
for every $j\geq 0$. Therefore 
$\mathcal{E}_{1+\alpha}(x_0,\rr_j) \leq \mathcal{E}_{1+\alpha}(x_{0},\rr)
$ holds for every $j\geq 0$. 
This means that 
\begin{flalign*}
\inf_{\elly \textnormal{ affine}}  \nra{u-\elly}_{L^2(B_{\rr_j}(x_0))}&\leq \inf_{\elly \textnormal{ affine}} \tx{E}_{u}(\elly;x_0,\rr_j)  \\ &\leq  \mathcal{E}_{1+\alpha}(x_{0},\rr) \rr_j^{1+\alpha}\\
  &\leq 
 \rr^{-1-\alpha}\tx{E}_{u}(0;x_0,\rr)\rr_j^{1+\alpha}\\ & \leq \rr^{-1-\alpha}\left(
 \nra{u}_{L^{2}(B_{\rr}(x_0))}+ c\, \tail(u;B_{\rr}(x_0))
 \right)  \rr_j ^{1+\alpha} \,.
\end{flalign*}
With $\sigma \in (0,\rr]$, find $j\geq 0$ such that $\rr_{j+1}< \sigma \leq \rr_j$, and estimate 
$$
 \inf_{\elly \textnormal{ affine}}  \nra{u-\elly}_{L^2(B_{\sigma}(x_0))} \leq \tet^{-n/2}   
  \inf_{\elly \textnormal{ affine}}  \nra{u-\elly}_{L^2(B_{\rr_j}(x_0))} .
$$
Merging the content of the last two displays easily  yields \eqref{exx.5} with a constant $c\equiv c(\data, \omega(\cdot), \alpha)\approx  \tet ^{-n/2-1-\alpha} $.

\section{Partial regularity and Proof of Theorems \ref{mainth}, \ref{mainth2} and \ref{mainth3}}\label{rs.sec} 
We first give the proof of Theorem \ref{mainth3}, and then, using the notation and the arguments developed there, we obtain the proofs of the remaining results. 
\subsection{Proof of Theorem \ref{mainth3}}
Fix a positive $\alpha < 2s-1$. The smallness threshold $\epsb{b}$ is determined as in Proposition \ref{cor.1}. Let us take a point $x\in \Omega$ such that there exists a ball $B_{2\rr}(x)\Subset \Omega$ such that \eqref{condreg} holds. The continuity property observed in \eqref{contexcess} provides the existence of a (small) radius $r_{x}$ (with no loss of generality we can assume that $r_{x} \leq \rr/4$) 
such that
$$
\rr^{-s}\tx{E}_{u}(\ell;y,\rr)< \epsb{b}, \quad \mbox{for every $y\in B_{r_{x}}(x)$}.
$$
We have therefore checked that the smallness condition \eqref{small}, required in Proposition \ref{cor.1}, is satisfied in the whole ball $B_{r_{x}}(x)$ and therefore 
\eqn{exx.55}
$$
 \inf_{\elly \textnormal{ affine}}  \nra{u-\elly}_{L^2(B_{\sigma}(y))} \leq c\left(
 \nra{u}_{L^{2}(B_{\rr}(y))}+ \tail(u;B_{\rr}(y))\right)
\left(\frac{\sigma}{\rr}\right) ^{1+\alpha}  
$$
for every $y\in B_{r_{x}}(x)$ and $\sigma \leq \rr$, where $c\equiv c (\data, \omega(\cdot), \alpha)$. Using that $r_{x} \leq \rr/4$ and \eqref{scatail.1}, it easily follows (see also the proof after \rif{ulti}) that 
$$
\nra{u}_{L^{2}(B_{\rr}(y))}+\tail(u;B_{\rr}(y)) \lesssim_{n,s} \nra{u}_{L^{2}(B_{2\rr}(x))}+\tail(u;B_{2\rr}(x)) 
$$ 
holds whenever $y \in  B_{r_{x}}(x)$. At this stage \eqref{exx.55} implies 
$$
 \sup_{0<\sigma \leq \rr, y \in B_{r_{x}}(x)} \sigma^{-1-\alpha }\inf_{\elly \textnormal{ affine}}  \nra{u-\elly}_{L^2(B_{\sigma}(y))} \leq \frac{c}{\rr^{1+\alpha}}\left(
\nra{u}_{L^{2}(B_{2\rr}(x))}+\tail(u;B_{2\rr}(x))\right)
$$
Note that the constant $c$ does not depend on the point $y\in B_{r_{x}}(x)$. Applying Campanato's characterization of H\"older continuity \cite{cam00} it follows that 
$$
[Du]_{0, \alpha;B_{r_{x}}(x)}\leq  \frac{c}{\rr^{1+\alpha}}\left(
 \nra{u}_{L^{2}(B_{2\rr}(x))}+ \tail(u;B_{2\rr}(x))\right)
$$
where the involved constant  $c$ depends on $\data, \omega(\cdot)$ and $\alpha$. We have therefore proved that \eqref{condreg} implies \eqref{condreg2}. For the remaining implication, which is the easy one,  note that \eqref{condreg2} obviously implies
\eqn{limitino}
$$
\lim_{\sigma\to 0}\,  \sigma^{1-s}\nra{Du}_{L^2(B_{\sigma}(x))}=0
$$
which in turn implies, via Lemma \ref{prelirid}, that 
$$
\lim_{\sigma\to 0}\,   \sigma^{-s}  \tx{E}_{u}(\ell_{x, \sigma}^u;x,\sigma) =0, 
$$
where, according to the definition in \eqref{affinecanonica}, it is $\ell_{x, \sigma}^u(y):=(Du)_{B_{\sigma}(x)}(y-x)+ (u)_{B_{\sigma}(x)}$, whenever $y \in \er^n$ and $B_{\sigma}(x)\Subset \Omega$.   
It follows that \eqref{condreg} can now be satisfied for a suitably small ball $B_{2\rr}(x)$, and with $\ell \equiv \ell_{x, \rr}^u$. 
This completes the proof of Theorem \ref{mainth3}. 
\begin{remark}\label{pronto}\em {From the above proof it follows that if $x \in \Omega$ is such that \eqref{limitino} holds, then for every $\alpha < 2s-1$ there exists a radius $r_x$, depending also on $\alpha$, such that $Du\in C^{0, \alpha}(B_{r_{x}}(x);\er^{N\times n})$.}
\end{remark}
\subsection{Proof of Theorems \ref{mainth} and \ref{mainth2}} Denote
$$
\tilde \Omega_{u}:= \Big\{
x \in \Omega \, \colon\, 
 \lim_{\rr\to 0}\,  \rr^{1-s}\nra{Du}_{L^2(B_{\rr}(x))}=0 \Big\}.
$$
By Remark \ref{pronto} it follows that $\tilde \Omega_u \subset \Omega_u$, where $\Omega_u$ is defined in \rif{regularset}. On the other hand the reverse inclusion $ \Omega_u \subset \tilde \Omega_u$ is trivial and therefore we conclude that $\tilde \Omega_u = \Omega_u$, which proves that $u\in C^{1,\alpha}_{\loc}(\Omega_{u};\er^{N})$ for every $\alpha < 2s-1$, again by Remark \ref{pronto} (recall that $\Omega_u$ is by definition open, and therefore so is $\tilde \Omega_u$). This also completes the proof of Theorem \ref{mainth2} as \eqref{higherss} is in fact the content of Theorem \ref{maggiore}. We finally  prove that 
\eqn{mainresultino}
$$
\begin{cases}
\,  \ddim(\Omega\setminus \tilde \Omega_u)\leq n-2-\theta \ \  \mbox{when $n\geq 3$}\\
\, \Omega =\tilde \Omega_u  \quad \mbox{when $n=2$}
\end{cases}
$$
where $\theta\equiv \theta (\data)$ is as in \eqref{mainresult} and this completes the proof of 
\eqref{mainresult} and therefore of Theorem \ref{mainth}. For this we use \eqref{higherss}. Observe that, eventually decreasing $\mathfrak{s}$, we can always assume that $ 2(1-s+\mathfrak{s})<n$ holds  when $n\geq 3$. At this stage \eqref{mainresultino}$_1$ follows with $\theta := 2(\mathfrak{s}-s)$ by using Lemma \ref{riducilemma} with the choices $w=D_iu$, $i\in \{1, \ldots, n\}$, $\gamma\equiv 1-s$, $t\equiv \mathfrak{s}$ and $p\equiv 2$. On the other hand, when $n=2$, it trivially follows that $2(1-s+\mathfrak{s})>n$ and therefore  \eqref{mainresultino}$_2$ again follows from  Lemma \ref{riducilemma} applied with the same choices made in the previous lines.

\end{document}